\documentclass[a4paper,12pt]{amsart}
\usepackage{yhmath}
\usepackage{graphicx}
\usepackage{mathtools}
\usepackage{mathrsfs}
\usepackage{amsmath}
\usepackage{amsthm}
\usepackage{dsfont}
\usepackage{fancyhdr}
\usepackage{harpoon}
\usepackage{amssymb}
\usepackage{enumerate}
\usepackage[bookmarks=true,colorlinks,linkcolor=black]{hyperref}
\usepackage{geometry}
\usepackage{extarrows}
\usepackage{yfonts}
\usepackage{amsfonts}
\usepackage{tikz}
\usetikzlibrary{arrows}
\usepackage{blkarray}
\usepackage{multirow}
\usepackage{color}
\usepackage{float}
\newtheorem{thm}{Theorem}[section]
\newtheorem{cor}[thm]{Corollary}

\newtheorem{cla}[thm]{Claim}
\newtheorem{lem}[thm]{Lemma}
\newtheorem{prop}[thm]{Proposition}
\theoremstyle{definition}
\newtheorem{defn}[thm]{Definition}

\newtheorem{rem}[thm]{Remark}
\theoremstyle{definition}

\numberwithin{equation}{section}

\DeclareMathOperator{\id}{id}

\DeclareMathOperator{\Span}{Span}

\DeclareMathOperator{\stab}{stab}

\DeclareMathOperator{\Ad}{Ad}

\DeclareMathOperator{\Lie}{Lie}

\DeclareMathOperator{\supp}{supp}
\DeclareMathOperator{\Leb}{Leb}

\DeclareMathOperator{\BL}{BL}

    {\begin{center}%
     \textsc{Acknowledgements}
\end{center}}%
    {\null}
\begin{document}


\title{R\MakeLowercase{igidity of joinings for time-changes of unipotent flows on quotients of }L\MakeLowercase{orentz groups}}
\author{S\MakeLowercase{iyuan} T\MakeLowercase{ang}}%
\address{D\MakeLowercase{epartment of} M\MakeLowercase{athematics}, IU, B\MakeLowercase{loomington}, IN 47401}
\email{1992.siyuan.tang@gmail.com,  siyutang@indiana.edu}
\maketitle

\begin{abstract}
  Let $u_{X}^{t}$ be a unipotent flow on $X=SO(n,1)/\Gamma$, $u_{Y}^{t}$ be a unipotent flow on $Y=G/\Gamma^{\prime}$. Let $\tilde{u}_{X}^{t}$, $\tilde{u}_{Y}^{t}$ be time-changes of  $u_{X}^{t}$, $u_{Y}^{t}$ respectively. We show the disjointness (in the sense of Furstenberg) between $u_{X}^{t}$ and $\tilde{u}_{Y}^{t}$ (or $\tilde{u}_{X}^{t}$ and $u_{Y}^{t}$) in certain situations.

  Our method refines  the works of Ratner and extends a recent work of Dong, Kanigowski and Wei.
\end{abstract}

\tableofcontents

\section{Introduction}
\subsection{Main results}
 In this paper, we study the rigidity of joinings of time-changes of unipotent flows.
First, let
\begin{itemize}
  \item  $G_{X}=SO(n_{X},1)$, $G_{Y}$ be a semisimple Lie group  with finite center and no compact factors  and $\Gamma_{X}\subset G_{X}$,  $\Gamma_{Y}\subset G_{Y}$ be irreducible lattices,
  \item $(X,m_{X})$, $(Y,m_{Y})$ be the homogeneous spaces $X=G_{X}/\Gamma_{X}$,  $Y=G_{Y}/\Gamma_{Y}$ equipped with the Lebesgue measures $m_{X}$, $m_{Y}$ respectively,
  \item $u_{X}^{t}$, $u_{Y}^{t}$ be   unipotent flows on $X$ and $Y$ respectively,
  \item $\tau_{X}$, $\tau_{Y}$ be    positive   functions with integral $m_{X}(\tau_{X})=m_{Y}(\tau_{Y})=1$ under certain regularity on $X$ and $Y$ respectively,
  \item $\tilde{u}^{t}_{X}$, $\tilde{u}^{t}_{Y}$  be the \textit{time-changes}\index{time-change} of $u^{t}_{X}$, $u^{t}_{Y}$  induced by $\tau_{X}$, $\tau_{Y}$, respectively,
  \item    $d\mu=\tau_{X} dm_{X}$,  $d\nu=\tau_{Y} dm_{Y}$  be the  $\tilde{u}_{X}$-, $\tilde{u}_{Y}$-invariant measures respectively.
\end{itemize}
  We shall verify the disjointness and so classify the joinings of $u_{X}^{t}$ and $\tilde{u}_{Y}^{t}$ (or $\tilde{u}_{X}^{t}$ and $u_{Y}^{t}$) in certain situations.

 Recall that a \textit{joining}\index{joinings}   of $\tilde{u}_{X}^{t}$ and $\tilde{u}_{Y}^{t}$ is    a $(\tilde{u}_{X}^{t}\times\tilde{u}_{Y}^{t})$-invariant probability measure on $X\times Y$, whose marginals on $X$ and $Y$ are $\mu$ and $\nu$ respectively. It was first introduced by Furstenberg in \cite{furstenberg1981recurrence}, and is a natural generalization of \textit{measurable conjugacies}\index{measurable conjugacies}. The classical results on classifying joinings under this context were established by Ratner \cite{ratner1982rigidity}, \cite{ratner1983horocycle}, \cite{ratner1986rigidity}, \cite{ratner1987rigid},  \cite{ratner1990measure}. First, the most celebrated  \textit{Ratner's theorem}\index{Ratner's theorem} indicates that all joinings between $u_{X}^{t}$ and $u_{Y}^{t}$ have to be \textbf{algebraic}. Besides,  for $G_{X}=SO(2,1)$, Ratner studied the    \textit{H-property}\index{H-property} (or \textit{Ratner's property}\index{Ratner's property}) of horocycle flows $u_{X}^{t}$, as well as their time-changes $\tilde{u}_{X}^{t}$,
 and then showed  that    any nontrivial (i.e. not the product measure $\mu\times\nu$) ergodic joining of $\tilde{u}_{X}^{t}$ and $\tilde{u}_{Y}^{t}$ is a \textbf{finite extension} of $\nu$. (In fact, this is even true for any measure-preserving system on $(Y,\nu)$.)
 Using this, Ratner was able to show   that for $G_{X}=G_{Y}=SO(2,1)$, the existence of a nontrivial   ergodic joining of $\tilde{u}_{X}^{t}$ and $\tilde{u}_{Y}^{t}$ implies that $\tau_{X}$ and $\tau_{Y}$ are \textit{algebraically cohomologous}\index{algebraically cohomologous}. In other words,    whether $\tilde{u}_{X}^{t}$ and $\tilde{u}_{Y}^{t}$ are disjoint  is determined by cohomological equations.

 It is natural to ask if it is possible to extend the results to $G_{X}=SO(n_{X},1)$ for $n_{X}\geq 3$. The difficulty is that the time-change   $\tilde{u}^{t}_{X}$ needs not have the \textit{H-property}\index{H-property}. It is one of the main ingredient of unipotent flows. Roughly speaking, H-property states that the divergence of nearby unipotent orbits happens always along some direction from the centralizer $C_{G_{X}}(u_{X})$ of the flow $u^{t}_{X}$. In particular, for  $G_{X}=SO(2,1)$, the direction can only be the flow direction $u^{t}_{X}$ itself. Moreover, Ratner  \cite{ratner1987rigid} naturally extended this notion to the general measure-preserving systems and verified it for the time-changes $\tilde{u}^{t}_{X}$ of horocycle flows. However, for $n_{X}\geq 3$, it seems that there is no suitable way to describe the ``centralizer" of the time-change $\tilde{u}^{t}_{X}$. Thus, classifying joinings of $\tilde{u}_{X}^{t}$ and $\tilde{u}_{Y}^{t}$ for $n_{X}\geq 3$ becomes a difficult problem.

  Recently, Dong, Kanigowski and Wei \cite{dong2020rigidity} considered the case when $G_{X}=SO(2,1)$, $G_{Y}$ is semisimple as above, $\Gamma_{X}$ and $\Gamma_{Y}$ are cocompact lattices. After comparing the $H$-property of $\tilde{u}_{X}^{t}$ and  $u_{Y}^{t}$,  they showed that $\tilde{u}_{X}^{t}$ and  $u_{Y}^{t}$ are disjoint once the Lie algebra $\mathfrak{g}_{Y}$ of $G_{Y}$ contains at least one weight vector of weight at least $1$ other than the $\mathfrak{sl}_{2}$-triples generated by $u_{Y}^{t}$.

 In this paper, we try to generalize the results stated above for $n_{X}\geq 3$. First, we follow the idea of Ratner and study the H-property of $u_{X}^{t}$ and deduce:
\begin{thm}\label{joinings202106.158}  Let $(Y,\nu,S)$ be a measure-preserving system of some map $S:Y\rightarrow Y$, $\rho$ be an ergodic joining of $u_{X}^{1}$ and $S$. Then either $\rho=\mu\times\nu$ or $(u^{1}_{X}\times S,\rho)$ is a compact extension of $(S,\nu)$. More precisely, if $\rho\neq\mu\times \nu$, then there exists   a compact subgroup $C^{\rho}\subset C_{G_{X}}(u_{X})$, and   $n>0$ such that  for $\nu$-a.e. $y\in Y$, there exist $x_{1}^{y},\ldots,x_{n}^{y}$ in the support of $\rho_{y}$ with
  \[\rho_{y}(C^{\rho}x_{i}^{y})=\frac{1}{n}\]
  for $i=1,\ldots,n$, where  $ \rho=\int_{Y}\rho_{y}d\nu(y)$ is the disintegration along $Y$.
\end{thm}

By Theorem \ref{joinings202106.158}, for any nontrivial ergodic joining $\rho$ of $u_{X}^{t}$ and $\tilde{u}_{Y}^{t}$,
there are measurable maps $\psi_{1},\ldots,\psi_{n}:Y\rightarrow X$ such that
\begin{equation}\label{joinings202106.162}
  \rho(f)=\int_{Y} \int_{C^{\rho}}\frac{1}{n}\sum_{p=1}^{n}f(k\psi_{p}(y),y)dm(k)d\nu(y)
\end{equation}
for $f\in C(X\times Y)$ where $m$ is the Lebesgue measure of the compact group $C^{\rho}$. Projecting $\rho$ to $(C^{\rho}\backslash X)\times Y$, we get
\[\overline{\rho}(f)=\int_{Y}  \frac{1}{n}\sum_{p=1}^{n}f(\overline{\psi}_{p}(y),y)d\nu(y)\]
for $f\in C((C^{\rho}\backslash X)\times Y)$.
Then, we can study the rigidity of $\rho$ by thinking about $\overline{\psi}_{1},\ldots,\overline{\psi}_{n}$. Also, $\overline{\rho}$ is a nontrivial ergodic joining of $u_{X}^{t}$ and  $\tilde{u}^{t}_{Y}$.

Then we can establish the rigidity of $\overline{\psi}_{p}$ by studying the shearing of $u_{X}^{t}$. The idea comes from \cite{ratner1986rigidity}, \cite{tang2020new}.  We require the time-changes having the effective mixing property. Thus, let $\mathbf{K}(Y)$ be the set  of all positive integrable functions $\tau$ on $Y$ such that $\tau,\tau^{-1}$ are bounded and satisfies
       \[\left|\int_{Y}\tau(y)\tau(u_{Y}^{t}y)d\nu(y)-\left(\int_{Y}\tau(y)\nu(y)\right)^{2}\right|\leq D_{\tau}|t|^{-\kappa_{\tau}}\]
       for some $D_{\tau},\kappa_{\tau}>0$. In other words, elements $\tau\in\mathbf{K}(Y)$ have polynomial decay of correlations.  Let $\langle u_{X}, a_{X},\overline{u}_{X}\rangle$,  $\langle u_{Y}, a_{Y},\overline{u}_{Y}\rangle$ be $\mathfrak{sl}_{2}$-triples of $G_{X}$ and $G_{Y}$, respectively.
Then we obtain the following:

 \begin{thm}[Extra central invariance  of $\rho$]\label{joinings202106.159}  Let $\tau_{Y}\in\mathbf{K}(Y)$, $\tilde{u}^{t}_{Y}$ be the time-change of $u^{t}_{Y}$ induced by $\tau_{Y}$ and $\rho$ be a nontrivial ergodic joining of $u_{X}^{t},\tilde{u}_{Y}^{t}$. Then there exist maps $\alpha:N_{G_{Y}}(u_{Y})\times Y\rightarrow\mathbf{R}$, $\beta:N_{G_{Y}}(u_{Y})\rightarrow C_{G_{Y}}(u_{Y})$ such that
\begin{enumerate}
  \item  Restricted to the centralizer $C_{G_{Y}}(u_{Y})$,  $\alpha:C_{G_{Y}}(u_{Y})\times Y\rightarrow\mathbf{R}$ is a cocycle, $\beta:C_{G_{Y}}(u_{Y})\rightarrow C_{G_{X}}(u_{X})$ is a homomorphism. Besides,
            $\tau_{Y}(cy)$ and $\tau_{Y}( y)$ are (measurably)  cohomologous along $u_{Y}^{t}$ via the transfer function $\alpha(c,y)$ for all $c\in C_{G_{Y}}(u_{Y})$; in other words,
            \[\int_{0}^{T}\tau_{Y}(cu_{Y}^{t}y)-\tau_{Y}(u_{Y}^{t}y) dt=\alpha(c,u^{T}_{Y}y)-\alpha(c,y).\]
  \item For $c\in C_{G_{Y}}(u_{Y})$, the map $S_{c}:X\times Y\rightarrow X\times Y$ defined by
  \[S_{c}:(x,y)\mapsto (\beta(c)x,\tilde{u}_{Y}^{-\alpha(c,y)}(cy))\]
  commutes with $u_{X}^{t}\times \widetilde{u}_{Y}^{t}$, and is $\rho$-invariant.  Besides, $S_{c_{1}c_{2}}=S_{c_{1}}\circ S_{c_{2}}$ for any $c_{1},c_{2}\in C_{G_{Y}}(u_{Y})$, and  $S_{u_{Y}^{t}}=\id$ for $t\in\mathbf{R}$.
  \item For $r\in\mathbf{R}$, the map $S_{a^{r}_{Y}}:X\times Y\rightarrow X\times Y$ defined by
  \[S_{a^{r}_{Y}}:(x,y)\mapsto \left(\beta(a^{r}_{Y})a^{r}_{X}x,\tilde{u}_{Y}^{-\alpha(a^{r}_{Y},y)}(a^{r}_{Y}y)\right)\]
  satisfies
  \[ S_{a^{r}_{Y}}\circ(u_{X}^{t}\times \widetilde{u}_{Y}^{t})=(u_{X}^{e^{-r}t}\times \widetilde{u}_{Y}^{e^{-r}t})\circ S_{a^{r}_{Y}}\]
  and is $\rho$-invariant. Besides, $S_{a^{r_{1}+r_{2}}_{Y}}=S_{a^{r_{1}}_{Y}}S_{a^{r_{2}}_{Y}}$ for any $r_{1},r_{2}\in \mathbf{R}$, and
  \[S_{a_{Y}}\circ S_{c}\circ S_{a_{Y}^{-1}}=S_{a_{Y}ca_{Y}^{-1}}\]
  for any $c\in C_{G_{Y}}(u_{Y})$.
\end{enumerate}
 \end{thm}

 For the opposite unipotent direction $\overline{u}_{Y}$, we cannot obtain the invariance for $\rho$ directly. However, we can fix it by making the ``$a$-adjustment". Here we further require $\tau_{Y}$ being smooth and $\alpha(c,\cdot)$ being integrable.  The idea comes from \cite{ratner1987rigid}. Then since $\overline{u}_{Y}$ and $C_{G_{Y}}(u_{Y})$ generate the whole group $G_{Y}$,  we are able to use \textit{Ratner's theorem}\index{Ratner's theorem} to get the rigidity of  $\overline{\psi}_{1},\ldots,\overline{\psi}_{n}$.

\begin{thm}[Cohomological criterion]\label{joinings202106.160}
     Let
 $G_{X}=SO(n_{X},1)$,  $G_{Y}$ be a semisimple Lie group  with finite center and no compact factors  and $\Gamma_{X}\subset G_{X}$,  $\Gamma_{Y}\subset G_{Y}$ be irreducible lattices. Let $U_{Y}\in\mathfrak{g}_{Y}$ be a nilpotent vector so that $C_{\mathfrak{g}_{Y}}(U_{Y})$ only contains vectors of   weight at most $2$, and let $u_{Y}=\exp(U_{Y})$. Let $\tau_{Y}\in \mathbf{K}(Y)\cap C^{1}(Y)$ so that $\tau_{Y}(cy)$ and  $\tau_{Y}(y)$ are $L^{1}$-cohomologous along $u_{Y}^{t}$ for any $c=\exp(v)\in C_{G_{Y}}(u_{Y})$ with positive weight. If there is  a nontrivial ergodic joining  $\rho$ of $u_{X}^{t}$ and $\tilde{u}_{Y}^{t}$, then $\tau_{X}\equiv1$ and $\tau_{Y}$ are joint cohomologous.
 \end{thm}
 \begin{rem}
  When $\tau_{X}\equiv1$ and $\tau_{Y}$ are joint cohomologous, one can deduce that $1$ (on $Y$) and $\tau_{Y}$  are (measurably) cohomologous. See Proposition \ref{rigid reparametrizations42} for further discussion.
 \end{rem}

In \cite{tang2020new}, we see that for $G_{Y}=SO(n_{Y},1)$, some cocompact lattice $\Gamma_{Y}$, there exists a function $\tau_{Y}\in \mathbf{K}(Y)\cap C^{1}(Y)$ such that
\begin{itemize}
  \item  $\tau_{Y}$ and $1$ are not measurably cohomologous,
  \item  for any $c\in C_{G_{Y}}(u_{Y})$, $\tau_{Y}(cy)$ and $\tau_{Y}(y)$ are not  measurably cohomologous if they are not $L^{2}$-cohomologous.
\end{itemize}
Applying Theorem  \ref{joinings202106.159} (1)  and Theorem \ref{joinings202106.160} to $\tau_{Y}$, we get
\begin{cor}[Existence of nontrivial time-changes]  For $G_{Y}=SO(n_{Y},1)$, there exists a cocompact lattice $\Gamma_{Y}$, and a function $\tau_{Y}$ on $Y=G_{Y}/\Gamma_{Y}$ such that  $u_{X}^{t}$ and $\tilde{u}_{Y}^{t}$ are disjoint (i.e. the only joining of $u_{X}^{t}$ and $\tilde{u}_{Y}^{t}$ is the product measure $\mu\times\nu$).
\end{cor}

Besides, the homomorphism $\beta|_{C_{G_{Y}}(u_{Y})}$ obtained by Theorem \ref{joinings202106.159} also  provide some information. Combining   Ratner's theorem, we conclude that the existence of nontrivial joinings requires the algebraic structure $G_{Y}$ to be similar to $G_{X}$.
\begin{thm}[Algebraic criterion]\label{joinings202106.161}   Let the notation and assumptions be as in Theorem \ref{joinings202106.160}. If there is  a nontrivial ergodic joining  $\rho$ of $u_{X}^{t}$ and $\tilde{u}_{Y}^{t}$, then $\rho$ is a finite extension of $\nu$ (i.e. the $C^{\rho}$ provided by Theorem \ref{joinings202106.158} is trivial). Besides, consider the decomposition (see (\ref{time change184})):
 \[ C_{\mathfrak{g}_{Y}}(U_{Y})=\mathbf{R}U_{Y}\oplus V^{\perp}_{C_{Y}},\ \ \ C_{\mathfrak{g}_{X}}(U_{X})=\mathbf{R}U_{X}\oplus V^{\perp}_{C_{X}}.\]
    Then the derivative  $d\beta|_{V^{\perp}_{C}}: V^{\perp}_{C_{Y}}\rightarrow V^{\perp}_{C_{X}}$ is an injective Lie algebra homomorphism.
\end{thm}
\begin{rem}
   Theorem \ref{joinings202106.160} and \ref{joinings202106.161} provide criteria for the disjointness of $u_{X}^{t}$ and $\tilde{u}_{Y}^{t}$. However, they require  that the functions $\tau_{Y}(cy)$ and $\tau_{Y}(y)$ are $L^{1}$-cohomologous for all $c\in C_{G_{Y}}(u_{Y})$ with positive weight (Theorem \ref{joinings202106.159} (1) indicates that they are always measurably cohomologous whenever $u_{X}^{t}$ and $\tilde{u}_{Y}^{t}$ are not disjoint). This condition seems in general is not easy to verify.
\end{rem}

On the other hand,  when the time-changes happen on quotients $X$ of Lorentz groups, we no longer have Theorem \ref{joinings202106.158}, because of the lack of H-property. However, if there exists a joining $\rho$ as in (\ref{joinings202106.162}),  we can follow the same idea as in Theorem \ref{joinings202106.159}  and obtain the rigidity in certain situations:
\begin{thm}\label{joinings202106.163}  Let
 $G_{X}=SO(n_{X},1)$,  $G_{Y}$ be a semisimple Lie group  with finite center and no compact factors  and $\Gamma_{X}\subset G_{X}$,  $\Gamma_{Y}\subset G_{Y}$ be irreducible lattices. Let $U_{Y}\in\mathfrak{g}_{Y}$ be nilpotent.   Let $\tau_{Y}\equiv 1$ and $\tau_{X}\in \mathbf{K}(X)$. Suppose that there exists an ergodic joining  $\rho$ of $\tilde{u}^{t}_{X}$ and  $u^{t}_{Y}$ that is a compact extension of $\nu$, i.e.  satisfies (\ref{joinings202106.162}). Then there exist maps $\alpha:N_{G_{Y}}(u_{Y})\times Y\rightarrow\mathbf{R}$, $\beta:N_{G_{Y}}(u_{Y})\rightarrow C_{G_{Y}}(u_{Y})$ such that
\begin{enumerate}
  \item  Restricted to the centralizer $C_{G_{Y}}(u_{Y})$,  $\alpha:C_{G_{Y}}(u_{Y})\times Y\rightarrow\mathbf{R}$ is a cocycle, $\beta:C_{G_{Y}}(u_{Y})\rightarrow C_{G_{X}}(u_{X})$ is a homomorphism. Besides,
            $\tau_{X}(cx)$ and $\tau_{X}( x)$ are (measurably)  cohomologous for all $c\in C_{G_{X}}(u_{X})$.
  \item For $c\in C_{G_{Y}}(u_{Y})$, the map $\widetilde{S}_{c}:X\times Y\rightarrow X\times Y$ defined by
  \[\widetilde{S}_{c}:(x,y) \mapsto ( u_{X}^{\alpha(c,y)}\beta(c)x,cy)\]
  commutes with $\tilde{u}_{X}^{t}\times u_{Y}^{t}$, and is $\rho$-invariant.  Besides, $\widetilde{S}_{c_{1}c_{2}}=\widetilde{S}_{c_{1}}\circ \widetilde{S}_{c_{2}}$ for any $c_{1},c_{2}\in C_{G_{Y}}(u_{Y})$, and  $\widetilde{S}_{u_{Y}^{t}}=\tilde{u}^{t}_{X}$ for $t\in\mathbf{R}$.
  \item  The map $S_{a_{Y}}:X\times Y\rightarrow X\times Y$ defined by   for $r\in\mathbf{R}$,
   \[\widetilde{S}_{a^{r}_{Y}}:(x,y)\mapsto \left(u_{X}^{\alpha(a^{r}_{Y},y)}\beta(a^{r}_{Y})a^{r}_{X}x,a^{r}_{Y}y\right)\]
   is $\rho$-invariant.  Besides, $\widetilde{S}_{a^{r_{1}+r_{2}}_{Y}}=\widetilde{S}_{a^{r_{1}}_{Y}}\widetilde{S}_{a^{r_{2}}_{Y}}$ for any $r_{1},r_{2}\in \mathbf{R}$, and
  \[\widetilde{S}_{a_{Y}}\circ \widetilde{S}_{c}\circ \widetilde{S}_{a_{Y}^{-1}}=\widetilde{S}_{a_{Y}ca_{Y}^{-1}}\]
  for any $c\in C_{G_{Y}}(u_{Y})$.
\end{enumerate}
 Moreover, for any weight vector $v\in V^{\perp}_{C_{Y}}$ of positive weight, the derivative
 \begin{equation}\label{joinings202106.169}
  d\beta|_{V^{\perp}_{C}}(v)\neq0.
 \end{equation}
\end{thm}
\begin{rem} In other words,  (\ref{joinings202106.169}) asserts that $d\beta$ is injective on the nilpotent part of $V^{\perp}_{C_{Y}}$.
One direct consequence of (\ref{joinings202106.169}) is that $C_{\mathfrak{g}_{Y}}(U_{Y})$ (under the assumptions of Theorem  \ref{joinings202106.163}) does not contain any weight vector of   weight $\neq0, 2$ (see Lemma \ref{joinings202106.170}).
\end{rem}

 In particular, recall that \cite{ratner1987rigid} showed that when  $G_{X}=SO(2,1)$, any time-change $\tilde{u}^{t}_{X}$ has $H$-property. It meets all the requirements of  Theorem \ref{joinings202106.163}. Then combining \cite{ratner1987rigid},  we obtain a slight extension of \cite{dong2020rigidity}:
 \begin{thm}\label{joinings202106.173}
 Let  the notation and assumptions be as in Theorem \ref{joinings202106.163}. Let $n_{X}=2$ and $\tau_{X}\in \mathbf{K}(X)\cap C^{1}(X)$.  If the Lie algebra $\mathfrak{g}_{Y}\ncong\mathfrak{sl}_{2}$, then $\tilde{u}_{X}^{t}$ and $u_{Y}^{t}$ are disjoint.
 \end{thm}

\subsection{Structure of the paper}
In Section \ref{joinings202106.164} we recall basic definitions, including  some basic material on the Lie algebra $\mathfrak{so}(n,1)$ (in Section \ref{joinings202105.23}, Section \ref{joinings202104.3}), as well as time-changes (Section \ref{joinings202103.14}) and
coboundaries (Section \ref{joinings202106.165}). In Section \ref{joinings202105.3}, we make use of the H-property of unipotent flows and deduce  Theorem \ref{joinings202106.158}. This requires studying the shearing property of $u_{X}^{t}$ for nearby points of the form $(x,y)$ and $(gx,y)$.  In Section \ref{joinings202103.13} we state and prove
a number of results which will be used as tools to prove the extra invariance of joinings $\rho$ (Theorem \ref{joinings202106.159}), in particular Proposition \ref{joinings202104.57} which pulls  the shearing phenomenon on the homogeneous space $X$ back to the Lie group $G_{X}$. We also give a quantitative estimate of the difference between two nearby points in terms of the length of the shearing (Lemma \ref{joinings202104.35}). In Section \ref{joinings202106.166}, we present the proof of Theorem \ref{joinings202106.159}  (Section \ref{joinings202105.24} Section \ref{joinings202106.167}) and a technical result for the opposite unipotent direction (Theorem \ref{joinings202106.122}). The latter result also requires studying the H-property of unipotent flows. Finally, using the results we got and Ratner's theorem,  we present in Section \ref{joinings202106.168} the proof of Theorem \ref{joinings202106.160}, \ref{joinings202106.161} (in Section \ref{joinings202106.146}), \ref{joinings202106.163} and \ref{joinings202106.173} (in Section \ref{joinings202106.172}).

 \noindent
 \textbf{Acknowledgements.}
 The original motivation of this paper came from the questions that Adam Kanigowski asked during the conversations. I am thankful to him for asking the questions.
 The paper was written under the guidance of my advisor David Fisher for my PhD thesis, and I am sincerely grateful for his help. I would also like to thank    Livio Flaminio for helpful discussions. 

\section{Preliminaries}\label{joinings202106.164}
\subsection{Definitions}\label{joinings202105.23}
Let $G\coloneqq SO(n,1)$ be the set of $g\in SL_{n+1}(\mathbf{R})$ satisfying
    \[ \left[
            \begin{array}{ccc}
              I_{n} &    \\
                 &   -1  \\
            \end{array}
          \right]g^{T}\left[
            \begin{array}{ccc}
              I_{n} &    \\
                 &   -1 \\
            \end{array}
          \right]=g^{-1} \]
           where $I_{n}$ is the $n\times n$ identity matrix. The corresponding Lie algebra $\mathfrak{g}$ then consists of $v\in \mathfrak{sl}_{n+1}(\mathbf{R})$ satisfying
           \[\left[
            \begin{array}{ccc}
              I_{n} &    \\
                 &   -1  \\
            \end{array}
          \right]v^{T}\left[
            \begin{array}{ccc}
              I_{n} &    \\
                 &   -1 \\
            \end{array}
          \right]=-v.\]
    Then the \textit{Cartan decomposition}\index{Cartan decomposition} can be given by
    \[\mathfrak{g}=\mathfrak{l}\oplus\mathfrak{p}=\left\{\left[
            \begin{array}{ccc}
              \mathbf{l} &    \\
                 &   0  \\
            \end{array}
          \right] :\mathbf{l}\in \mathfrak{so}(n) \right\}\oplus\left\{ \left[
            \begin{array}{ccc}
              0  &  \mathbf{p}  \\
               \mathbf{p}^{T}  &  0    \\
            \end{array}
          \right]: \mathbf{p}\in\mathbf{R}^{n}\right\}.\]
          Let $E_{ij}$ be the $(n\times n)$-matrix with $1$ in the $(i,j)$-entry and $0$ otherwise. Let $e_{k}\in\mathbf{R}^{n}$ be the $k$-th standard basis vector. Set
          \[Y_{k}\coloneqq\left[
            \begin{array}{ccc}
              0 & e_{k}    \\
               e_{k}^{T}  &   0 \\
            \end{array}
          \right],\ \ \ \Theta_{ij}\coloneqq\left[
            \begin{array}{ccc}
              E_{ji}-E_{ij} &  0    \\
               0  &   0 \\
            \end{array}
          \right].\]
  Then $Y_{i},\Theta_{ij}$ form a basis of $\mathfrak{g}=\mathfrak{so}(n,1)$.

 Let $\mathfrak{a}=\mathbf{R}Y_{n}\subset\mathfrak{p}$ be a maximal abelian subspace of $\mathfrak{p}$. Then the root space decomposition of $\mathfrak{g}$ is given by
\begin{equation}\label{time change187}
  \mathfrak{g}=\mathfrak{g}_{-1}\oplus\mathfrak{m}\oplus\mathfrak{a}\oplus\mathfrak{g}_{1}.
\end{equation}
Denote by $\mathfrak{n}\coloneqq\mathfrak{g}_{1}$ the sum of the positive root spaces.
Let $\rho$ be the half sum of positive roots. We also adopt the convention by identifying $\mathfrak{a}^{\ast}$ with $\mathbf{C}$ via $\lambda\mapsto\lambda(Y_{n})$. Thus, $\rho=\rho(Y_{n})=(n-1)/2$.

 Let  $\Gamma\subset G$ be a lattice,  $X\coloneqq   G/\Gamma$, $\mu$ be the Haar probability measure on $X$. Fix a nilpotent $U\in  \mathfrak{g}_{-1}$. On $G/\Gamma$, denote by
  \begin{itemize}
    \item  $\phi^{Y_{n}}_{t}(x)\coloneqq\exp(tY_{n})x=a^{t}x$   a geodesic flow,
    \item $\phi^{U}_{t}(x)\coloneqq\exp(tU)x=u^{t}x$ a unipotent flow.
  \end{itemize}
 It is worth noting that
\[[Y_{n},U]=-U.\]
Then   there exists $\overline{U}\in\mathfrak{g}$ such that $\{U,Y_{n},\overline{U}\}$ is an \textit{$\mathfrak{sl}_{2}$-triple}. Denote
\[\overline{u}^{t}\coloneqq\exp(t\overline{U}).\]
For convenience, we   choose
\begin{equation}\label{time change211}
  U\coloneqq\left[
            \begin{array}{ccc}
              0 & e_{n-1}   & e_{n-1}    \\
              -e_{n-1}^{T} & 0   & 0   \\
               e_{n-1}^{T}  & 0 &   0 \\
            \end{array}
          \right],\ \ \ \overline{U}\coloneqq\left[
            \begin{array}{ccc}
              0 & -e_{n-1}   & e_{n-1}    \\
              e_{n-1}^{T} & 0   & 0   \\
               e_{n-1}^{T}  & 0 &   0 \\
            \end{array}
          \right].
\end{equation}
Then $\langle u^{t},a^{t},\overline{u}^{t}\rangle$ generates $SO(2,1)\subset SO(n,1)$.

\subsection{$\mathfrak{sl}_{2}$-weight decomposition}\label{joinings202104.3}
First,     consider an arbitrary Lie algebra $\mathfrak{g}$ as a $\mathfrak{sl}_{2}$-representation via the adjoint map (after identifying an image of $\mathfrak{sl}_{2}$ by \textit{Jacobson–Morozov theorem}\index{Jacobson–Morozov theorem}), then by  the complete reducibility of $\mathfrak{sl}_{2}$, there is a    decomposition of $\mathfrak{sl}_{2}$-representations
\begin{equation}\label{joinings202103.2}
  \mathfrak{g}=\mathfrak{sl}_{2}\oplus V^{\perp}
\end{equation}
where $V^{\perp}\subset\mathfrak{g}$ is the sum of $\mathfrak{sl}_{2}$-irreducible representations other than   $\mathfrak{sl}_{2}$. In particular, for $\mathfrak{g}=\mathfrak{so}(n,1)$, we have
\begin{equation}\label{joinings202105.14}
 V^{\perp}=\sum_{i} V_{i}^{0}\oplus\sum_{j} V_{j}^{2}
\end{equation}
where $V_{i}^{0}$ and $V_{j}^{2}$ are $\mathfrak{sl}_{2}$-irreducible representations  with highest weights $0$ and $2$. More precisely, we have

\begin{lem}\label{time change179}
   By the weight decomposition, an irreducible $\mathfrak{sl}_{2}$-representation $V^{\varsigma}$ is the direct sum of weight spaces, each of which is $1$ dimensional. More precisely, there exists a basis $v_{0},\ldots,v_{\varsigma}\in V^{\varsigma}$ such that
   \[U.v_{i}=(i+1)v_{i+1},\ \ \ Y_{n}.v_{i}= \frac{\varsigma-2i}{2}v_{i}.\]
\end{lem}
Thus, if $V^{\varsigma}$ is an irreducible representation of  $\mathfrak{sl}_{2}$ with the highest weight $\varsigma\leq 2$, then for any $v=b_{0}v_{0}+\cdots+b_{\varsigma}v_{\varsigma}\in V^{\varsigma}$, we have
 \begin{align}
\exp(tU).v=&\sum_{j=0}^{\varsigma}\sum_{i=0}^{j}b_{i}\binom{j}{i}t^{j-i}v_{j},  \; \label{dynamical systems5}\\
 \exp(\omega Y_{n}).v=&  \sum_{j=0}^{\varsigma} b_{j}e^{ (\varsigma-2j)\omega /2} v_{j}. \;  \label{joinings202106.129}
\end{align}

For elements $g\in\exp\mathfrak{g}$, we decompose
\[g=h\exp(v),\ \ \ h\in SO_{0}(2,1),\ \ \ v\in V^{\perp}.\]
Moreover, it is convenient to think about $h\in SO_{0}(2,1)$ as a $(2\times 2)$-matrix with determinant $1$. Thus, consider the two-to-one   isogeny $\iota:SL_{2}(\mathbf{R})\rightarrow SO(2,1)\subset G$ induced by $\mathfrak{sl}_{2}(\mathbf{R})\rightarrow\Span\{U,Y_{n},\bar{U}\}\subset\mathfrak{g}$.    In the following, for $h\in SO_{0}(2,1)$ and $v$ in an irreducible representation, we write
\[h=\left[
            \begin{array}{ccc}
              a &   b  \\
            c &   d\\
            \end{array}
          \right],\ \ \ v=b_{0}v_{0}+\cdots+b_{\varsigma}v_{\varsigma}\]
where $v_{i}$ are weight vectors in $\mathfrak{g}$ of weight $i$. Notice that $h$ should more appropriately be written as $\iota(h)$. Besides, for notational simplicity, we shall usually assume that $v\in V^{\perp}$ lies in a single irreducible representation, since the proofs will mostly focus on the $\Ad u^{t}$-action and so the general case will be identical but tedious to write down.

For  the centralizer $ C_{\mathfrak{g}}(U)$ (for an arbitrary Lie algebra $\mathfrak{g}$), we  have the corresponding decomposition:
\begin{equation}\label{time change184}
  C_{\mathfrak{g}}(U)=\mathbf{R}U\oplus V^{\perp}_{C}
\end{equation}
where  $V^{\perp}_{C}\subset V^{\perp}$ consists of  highest weight vectors other than $U$. In particular, for $\mathfrak{g}=\mathfrak{so}(n,1)$, under the setting (\ref{time change211}), one may calculate
\begin{align}
C_{\mathfrak{g}}(U)=&\mathbf{R}U\oplus V^{\perp}_{C}=\mathbf{R}U\oplus \mathfrak{k}^{\perp}_{C}\oplus \mathfrak{n}^{\perp}_{C} \;\nonumber\\
=&  \mathbf{R}U\oplus\left[
            \begin{array}{ccc}
             \mathfrak{so}(n-2) &    \\
                 &   0  \\
            \end{array}
          \right] \oplus\left\{ \left[
            \begin{array}{cccc}
              0&   0  &  \mathbf{u} &  \mathbf{u}  \\
                  0&   0  &  0 & 0  \\
             -\mathbf{u}^{T} &   0 &  0 &   0\\
               \mathbf{u}^{T}  &0 &0 &  0    \\
            \end{array}
          \right]:  \mathbf{u}\in\mathbf{R}^{n-2}\right\}. \;  \label{joinings202103.12}
\end{align}
Note that  $\mathfrak{k}^{\perp}_{C}$ consists of semisimple elements, and $\mathfrak{n}^{\perp}_{C}$ consists of nilpotent elements, and they satisfy $[\mathfrak{k}^{\perp}_{C},\mathfrak{n}^{\perp}_{C}]=\mathfrak{n}^{\perp}_{C}$.

\subsection{Time-changes}\label{joinings202103.14}
 Let $Y$ be a homogeneous space and $U$ be a nilpotent. Let $\phi^{U,\tau}_{t}$ be a \textit{time change}\index{time change} for the unipotent flow $\phi^{U}_{t}$, $t\in\mathbf{R}$. Thus, we assume that
          \begin{itemize}
            \item  $\tau:Y\rightarrow\mathbf{R}^{+}$ is a integrable nonnegative function on $Y$ satisfying
          \[\int_{Y}\tau(y)d\mu(y)=1,\]
          \item  $\xi:Y\times\mathbf{R}\rightarrow\mathbf{R}$ is the cocycle determined by
             \[t=\int_{0}^{\xi(y,t)}\tau(u^{s}y)ds=\int_{0}^{\xi(y,t)}\tau(\phi_{t}^{U} y)ds.\]
          \item $\phi^{U,\tau}_{t}:Y\rightarrow Y$ is given by the relation
          \[\phi^{U,\tau}_{t}(y)\coloneqq u^{\xi(y,t)}y.\]
          \end{itemize}
          \begin{rem}\label{joinings202104.59}
  Note that $\phi^{U,1}_{t}=\phi^{U}_{t}$. Besides, one can check that  $\phi^{U,\tau}_{t}$ preserves the probability measure on $Y$ defined by  $d\nu\coloneqq\tau dm_{Y}$ where $m_{Y}$ is the Lebesgue measure on $Y$. On the other hand, if  $\tau$ is smooth, then the time-change $\phi_{t}^{U,\tau}$ is the flow on $Y$ generated by the smooth vector field $U_{\tau}\coloneqq U/\tau$.
\end{rem}
          In practice, we define $z:Y\times\mathbf{R}\rightarrow\mathbf{R}$ by
          \[z(y,t)=\int_{0}^{t}\tau(u^{s}y)ds.\]
          It follows that
          \begin{equation}\label{joinings202105.9}
            t=z(y,\xi(y,t)),\ \ \ \phi^{U,\tau}_{z(y,t)}(x)=\phi^{U}_{t}(y)= u^{t}y.
          \end{equation}

       Let $\kappa>0$ and $\mathbf{K}_{\kappa}(Y)$ be the collection of all positive integrable functions $\tau$ on $Y$ such that $\tau,\tau^{-1}$ are bounded and satisfies
       \begin{equation}\label{joinings202104.26}
       \left|\int_{Y}\tau(y)\tau(u^{t}y)d\nu(y)-\left(\int_{Y}\tau(y)\nu(y)\right)^{2}\right|\leq D_{\tau}|t|^{-\kappa}
       \end{equation}
       for some $D_{\tau}>0$. Let $\mathbf{K}(Y)=\bigcup_{\kappa>0}\mathbf{K}_{\kappa}(Y)$.  This is the effective mixing property of the unipotent flow $\phi_{t}^{U}$. Note that \cite{kleinbock1999logarithm} (see also \cite{venkatesh2010sparse})  has shown that there is $\kappa>0$ such that
       \[\left|\langle\phi^{U}_{t}(f),g\rangle-\left(\int_{Y}f(y)\nu(y)\right)\left(\int_{Y}g(y)\nu(y)\right)\right|\ll(1+|t|)^{-\kappa}\|f\|_{W^{s}}\|g\|_{W^{s}}\]
       for $f,g\in C^{\infty}(X)$, where $s\geq \dim(K)$ and $W^{s}$ denotes the Sobolev norm on $Y=G/\Gamma$. According to   Lemma 3.1 \cite{ratner1986rigidity}, when $\tau\in\mathbf{K}_{\kappa}(Y)$, we have the effective ergodicity:   there is $K\subset Y$ with $\nu(K)>1-\sigma$ and $t_{K}>0$ such that
         \begin{equation}\label{joinings202105.45}
           |t-z(y,t)|=O(t^{1-\kappa})
         \end{equation}
         for   all $t\geq t_{K}$ and $y\in K$.
    Later on, we shall make use of the effective mixing/ergodicity to study the shearing property of unipotent flows (see Section \ref{joinings202103.13} (\ref{joinings202104.24})).

\subsection{Cohomology}\label{joinings202106.165}
We first introduce the $1$-coboundary of two functions.
\begin{defn}[Cohomology]
   We say that two   functions $\tau_{1},\tau_{2}$ on $Y$ are  \textit{measurable (respectively $L^{2}$, smooth, etc.) cohomologous over the flow $\phi_{t}$}\index{cohomologous} if there exists a measurable (respectively $L^{2}$, smooth, etc.) function  $f$ on $Y$, called the \textit{transfer function}\index{transfer function}, such that
\begin{equation}\label{joinings202106.138}
  \int_{0}^{T}\tau_{1}(\phi_{t}y)-\tau_{2}(\phi_{t}y)dt=f(\phi_{T}y)-f(y).
\end{equation}
\end{defn}

For $i\in\{1,2\}$, let  $(Y_{i},\mathcal{Y}_{i},\nu_{i},\phi_{t}^{(i)})$ be  measure-preserving flows, and let  $\tau_{i}:Y_{i}\rightarrow\mathbf{R}$ be measurable functions on $Y_{i}$. Besides, we extend $\tau_{i}$ to $Y_{1}\times Y_{2}$ by setting
\[\tau_{i}: (y_{1},y_{2})\mapsto\tau_{i}(y_{i}),\ \ \ i=1,2.\]
\begin{defn}[Joint cohomology]
   Let $\rho\in J(\phi_{t}^{(1)},\phi_{t}^{(2)})$ be a joining of $\phi_{t}^{(1)}$ and  $\phi_{t}^{(2)}$. We say that $\tau_{1}$ and $\tau_{2}$ are \textit{jointly cohomologous via $\rho$}\index{jointly cohomologous} if $\tau_{1}$ and $\tau_{2}$ (considered as functions on $Y_{1}\times Y_{2}$) are cohomologous over $\phi_{t}^{(1)}\times\phi_{t}^{(2)}$ on $(Y_{1}\times Y_{2},\rho)$.
   More specifically, if $\tau_{1}$ and $\tau_{2}$ are cohomologous over $\phi_{t}^{(1)}\times\phi_{t}^{(2)}$ with a transfer function $f:Y_{1}\times Y_{2}\rightarrow\mathbf{R}$, then we say that $\tau_{1}$ and  $\tau_{2}$ are \textit{jointly cohomologous via $(\rho,f)$}\index{jointly cohomologous}, and we have
   \begin{equation}\label{rigid reparametrizations43}
 \int_{0}^{T}(\tau_{1}-\tau_{2})(\phi^{(1)}_{t}y_{1}, \phi^{(2)}_{t}y_{2})dt=   f(\phi^{(1)}_{T}y_{1}, \phi^{(2)}_{T}y_{2})-f(y_{1}, y_{2})
   \end{equation}
   for $\rho$-a.e. $(y_{1},y_{2})\in Y_{1}\times Y_{2}$ and all $T\in\mathbf{R}$.
\end{defn}
Let $\mathcal{A}_{1}\coloneqq\{A\times Y_{2}:A\in\mathcal{Y}_{1}\}$,  $\mathcal{A}_{2}\coloneqq\{Y_{1}\times A:A\in\mathcal{Y}_{2}\}$. Then there is a unique family $\{\rho_{y_{1}}^{\mathcal{A}_{1}}:y_{1}\in Y_{1}\}$ of probability measure, called the \textit{conditional measures}\index{conditional measure}, on $Y_{2}$ such that
\begin{equation}\label{rigid reparametrizations42}
  E^{\rho}(g|\mathcal{A}_{1})(y_{1})=\int_{Y_{2}}g(y_{1},y_{2})d\rho_{y_{1}}^{\mathcal{A}_{1}}(y_{2}),\ \ \ \rho^{\mathcal{A}_{1}}_{\phi^{(1)}_{t}y_{1}}=(\phi^{(2)}_{t})_{\ast}\rho^{\mathcal{A}_{1}}_{y_{1}}
\end{equation}
 for   every $g\in L^{1}(Y_{1}\times Y_{2},\rho)$, $t\in\mathbf{R}$,  and $\nu_{1}$-a.e. $y_{1}\in Y_{1}$.
Taking the integration over $\rho_{y_{1}}^{\mathcal{A}_{1}}$, expressions (\ref{rigid reparametrizations43}) and (\ref{rigid reparametrizations42}) show that if the transfer function $f(y_{1},\cdot)\in L^{1}(Y_{2},\rho^{\mathcal{A}_{1}}_{y_{1}})$ for $\nu_{1}$-a.e. $y_{1}\in Y_{1}$, then $\tau_{1}$ and $E^{\rho}(\tau_{2}|\mathcal{A}_{1})$ are cohomologous along $\phi^{(1)}_{t}$ via $E^{\rho}(f|\mathcal{A}_{1})$. We have just proved the following:
\begin{prop}
  Let  $\tau_{i}:Y_{i}\rightarrow\mathbf{R}$ be measurable functions on $Y_{i}$, $i=1,2$. Suppose that $\tau_{1}$ and $\tau_{2}$ are jointly cohomologous via $(\rho,f)$ with $f(y_{1},\cdot)\in L^{1}(Y_{2},\rho^{\mathcal{A}_{1}}_{y_{1}})$ for $\mu_{1}$-a.e. $y_{1}\in Y_{1}$. Then $\tau_{1}$ and $E^{\rho}(\tau_{2}|\mathcal{A}_{1})$ are cohomologous over $\phi^{(1)}_{t}$ via $E^{\rho}(\tau_{2}|\mathcal{A}_{1})$.
\end{prop}

\section{Shearing property I, H-flow on one factor}\label{joinings202105.3}
\subsection{Joinings}
Let $G=SO(n,1)$, $\Gamma$ be a lattice of $G$, $(X,\mu)$ be the homogeneous space $X=G/\Gamma$ equipped with the Lebesgue measure $\mu$, and let $\phi^{U}_{t}$ be a unipotent flow on $X$. Let $(Y,\nu,S)$ be a measure-preserving system. We want to study the joinings of $(X,\mu,\phi_{1}^{U})$ and  $(Y,\nu,S)$. Thus, let $\rho$ be an ergodic \textit{joining} of $\phi_{1}^{U}$ and  $S$\index{joinings}, i.e. $\rho$ is a probability measure on $X\times Y$, whose marginals on $X$ and $Y$ are $\mu$ and $\nu$ respectively, and which is $(\phi_{1}^{U}\times S)$-ergodic.

Let $C(\phi_{1}^{U})$ be the \textit{commutant} of $\phi_{1}^{U}$\index{commutant of transformations}, i.e. collection of all measure-preserving transformations on $X$ that commute with $\phi_{1}^{U}$. The following is a basic criterion for $\rho$ in terms of the commutant of $\phi_{1}^{U}$:
\begin{lem}\label{joinings202012.3}
Let the notation and assumptions be as above. Assume further that  $T\in C(\phi_{1}^{U})$ is   ergodic   on $(X,\mu)$.
Then \[\text{either }(T\times\id)_{\ast}\rho\perp\rho\ \ \text{ or }\ \ \rho=\mu\times\nu.\]
\end{lem}
\begin{proof}
   First, by the commutative property of $T$, we easily see that $(T\times\id)_{\ast}\rho$ is again $(\phi_{1}^{U}\times S)$-ergodic on $X\times Y$. It implies that  either $(T\times\id)_{\ast}\rho\perp\rho$ or $(T\times\id)_{\ast}\rho=\rho$. Now assume that  $(T\times\id)_{\ast}\rho=\rho$, i.e. $\rho$ is $(T\times\id)$-invariant. Then via disintegration, we know that $\rho_{y}$ is $T$-invariant on $X$ for $\nu$-a.e. $y\in Y$, where
   \begin{equation}\label{joinings202012.8}
     \rho=\int_{Y}\rho_{y}d\nu(y).
   \end{equation}

   Now assume for contradiction  that there exists $B\subset Y$ with $\nu(B)>0$ such that $\rho_{y}\neq\mu$ for  $y\in B$.
   It follows that for $y\in B$, there is $A_{y}\subset X$ with $\mu(A_{y})>0$ such that for $x\in A_{y}$, we have
   \begin{equation}\label{joinings202012.2}
    (\rho_{y})_{x}^{\mathcal{E}}\neq\mu
   \end{equation}
    where $(\rho_{y})_{x}^{\mathcal{E}}$ is given by the $T$-ergodic decomposition
  \[\rho_{y}=\int_{X}(\rho_{y})_{x}^{\mathcal{E}}d\mu(x).\]
  Notice that by the ergodicity, there is a $\mu$-conull set $\Omega\subset X$, namely the set of $T$-generic points of $\mu$, such that $(\rho_{y})_{x}^{\mathcal{E}}(\Omega)=0$ for the measures $(\rho_{y})_{x}^{\mathcal{E}}$ in (\ref{joinings202012.2}). Then by the assumption of joining, we have
  \begin{align}
   \mu(\Omega)=\rho(\pi_{X}^{-1}(\Omega))=& \int_{Y}\rho_{y} (\Omega)d\nu(y)\;\nonumber\\
   =& \int_{B}\rho_{y}(\Omega)d\nu(y)+\int_{Y\setminus B}\rho_{y}(\Omega)d\nu(y)\;\nonumber\\
   \leq& \int_{B}\int_{X}(\rho_{y})_{x}^{\mathcal{E}}(\Omega)d\mu(x)d\nu(y)+ \nu(Y\setminus B)\;\nonumber\\
      =& \int_{B}\int_{X\setminus A_{y}}(\rho_{y})_{x}^{\mathcal{E}}(\Omega)d\mu(x)d\nu(y)+ \nu(Y\setminus B)\;\nonumber\\
       \leq & \int_{B} \mu(X\setminus A_{y})d\nu(y)+ \nu(Y\setminus B)\;\nonumber\\
   <&\nu(B)+\nu(Y\setminus B)=1\nonumber
  \end{align}
  which is a contradiction. Thus, we conclude that $\rho_{y}=\mu$ for $\nu$-a.e. $y\in Y$ and so $\rho=\mu\times\nu$.
\end{proof}

By \textit{Moore’s ergodicity theorem}\index{Moore’s ergodicity theorem}, we deduce that
\begin{cor}\label{joinings202103.3}
   If $w\in C_{\mathfrak{g}}(U)$ so that $\langle\exp tw\rangle_{t\in\mathbf{R}}$ is not compact, then
   \[\text{either }(\phi^{w}_{1}\times\id)_{\ast}\rho\perp\rho\ \ \text{ or }\ \ \rho=\mu\times\nu.\]
\end{cor}

\subsection{H-property}\label{joinings202103.10}

In this section, we want to introduce the \textit{$H$-property} (or \textit{Ratner property}) in order to study the joining $\rho$ in terms of the unipotent flow $\phi^{U}_{t}$ on $X$.  The classic $H$-property can be formulated as follows:
\begin{thm}[H-property, \cite{witte1985rigidity}]\label{h-property202103.1}
   Let $u$ be a unipotent element of $G$. Given any neighborhood $Q$ of $e$ in $C_{G}(u)$, there is a compact subset $\partial Q$ of $Q\setminus\{e\}$ such that for any $\epsilon>0$ and $M>0$, there are $\alpha=\alpha(u,Q,\epsilon)>0$ and $\delta=\delta(u,Q,\epsilon,M)>0$  such that if $x_{1},x_{2}\in X$ with $d_{X}(x_{1},x_{2})<\delta$ then one of the following holds:
  \begin{itemize}
    \item  $x_{2}=cx_{1}$ for some $c\in C_{G}(u)$ with $d_{G}(e,c)<\delta$,
    \item there are   $L>M/\alpha$ and $q\in\partial Q$ such that
    \begin{equation}\label{joinings202103.5}
      d_{X}(u^{n}x_{2},qu^{n}x_{1})<\epsilon
    \end{equation}
        whenever $n\in[L,(1+\alpha)L]$.
  \end{itemize}
\end{thm}
\begin{rem}\label{joinings202103.4}
   In fact, for $x_{2}=gx_{1}$ with $g=\exp(v)\in B^{G}_{\delta}$, the element $q\in C_{\mathfrak{g}}(U)$ in  Theorem \ref{h-property202103.1} is chosen by
   \begin{equation}\label{joinings202106.116}
    q=\pi_{C_{\mathfrak{g}}(U)}\exp(LU).v
   \end{equation}
   where $\pi_{C_{\mathfrak{g}}(U)}: \mathfrak{g}\rightarrow C_{\mathfrak{g}}(U)$ is the natural projection and $\exp(LU).v$ is the adjoint representation (see (\ref{dynamical systems5})). We often call $q$ as the \textit{fastest relative motion}\index{fastest relative motion} between $x_{1},x_{2}$; see \cite{morris2005ratner} for more discussion. In what follows, we   choose $Q=B_{\lambda}^{C_{G}(u)}$ to be the ball of radius $\lambda$ of $e$ in $C_{G}(u)$ for sufficiently small $\lambda$ (independent of $\epsilon$), and then $\partial Q$ is  the sphere of radius $\lambda$. Now by (\ref{joinings202103.2}) (\ref{joinings202105.14}), we have the decomposition
   \[v=v_{0}+v_{2}\]
    where $v_{0}\in\sum_{i} V_{i}^{0}$ and  $v_{2}\in \mathfrak{sl}_{2}+\sum_{j} V_{j}^{2}$. Thus, $\|v_{0}\|,\|v_{2}\|<\delta$ and
    \[q=v_{0}+\pi_{C_{\mathfrak{g}}(U)}\exp(LU).v_{2}.\]
    Since $\|q\|=\lambda$, we see that $v_{0}$ is negligible. In other words, we can replace $q$ by
    \begin{equation}\label{joinings202106.102}
      q^{\prime}\coloneqq\pi_{C_{\mathfrak{g}}(U)}\exp(LU).v_{2}
    \end{equation}
     and then Theorem \ref{h-property202103.1} still holds. On the other hand,  note that $q^{\prime}\in \mathfrak{n}= \mathbf{R}U+\mathfrak{n}_{C}^{\perp}$ (cf. (\ref{joinings202103.12})).  Thus, the one-parameter group $\langle\exp(tq^{\prime})\rangle_{t\in\mathbf{R}}$ generated by $q^{\prime}$ is not compact.
\end{rem}

In the following, we shall generalize the idea in \cite{ratner1983horocycle} and prove Theorem \ref{joinings202106.158}.
\begin{thm}\label{joinings202012.1}
  Let the notation and assumptions be as above. Then either $\rho=\mu\times\nu$ or $(\phi^{U}_{1}\times S,\rho)$ is a compact extension of $(S,\nu)$. More precisely, if $\rho\neq\mu\times \nu$, then there exists a $\nu$-conull set $\Theta\subset Y$, a compact subgroup $C^{\rho}\subset C_{G}(u)$, and   $n>0$ such that  for any $y\in\Theta$, there exist $x_{1}^{y},\ldots,x_{n}^{y}$ in the support of $\rho_{y}$ with
  \[\rho_{y}(C^{\rho}x_{i}^{y})=\frac{1}{n}\]
  for $i=1,\ldots,n$, where  $ \rho=\int_{Y}\rho_{y}d\nu(y)$ is the disintegration along $Y$ (cf. (\ref{joinings202012.8})).
\end{thm}

Assume that $\rho\neq\mu\times\nu$. Then by Corollary \ref{joinings202103.3}, there is a $\rho$-conull set $\Omega\subset X\times Y$, namely the set of $(\phi_{1}^{U}\times S)$-generic points, such that $(\phi^{w}_{1}\times\id) (\Omega)\cap \Omega=\emptyset$
 for all $w\in  \mathfrak{n}$.  Given a sufficiently small $\lambda>0$, we define the sphere of radius $\lambda$ of $0$ by
\[ B^{\mathfrak{n}}_{\lambda}\coloneqq\{w\in \mathfrak{n}:\|w\|=\lambda\}.\]
Then,   one can find a       compact    subset $K_{1}\subset \Omega$  with $\mu(K_{1})>199/200$. Then
\[ \bigcup_{w\in B^{\mathfrak{n}}_{\lambda}}(\phi^{w}_{1}\times \id)(K_{1}) \]
is compact.
 Thus, there are $\epsilon>0$ and $K_{2}\subset  K_{1}$ with $\mu(K_{2})>99/100$ such that
\[d_{X\times Y}\left(K_{2},\bigcup_{w\in B^{\mathfrak{n}}_{\lambda}}(\phi^{w}_{1}\times \id)(K_{1}) \right)>\epsilon.\]
 It follows that  if $(x_{1},y),(x_{2},y)\in K_{2}$ then
\begin{equation}\label{joinings202012.5}
  d_{X}(x_{2},\phi^{w}_{1}x_{1})\geq\epsilon
\end{equation}
for all $w\in B^{\mathfrak{n}}_{\lambda}$. Let $\alpha=\alpha(\epsilon)>0$ be as in Theorem \ref{h-property202103.1}. Comparing (\ref{joinings202012.5}) with (\ref{joinings202103.5}), we conclude

\begin{lem}\label{joinings202012.6}
   There is a positive number $\delta=\delta(\epsilon)>0$, a measurable set $K_{4}\subset \Omega$ with $\rho(K_{4})>0$ such that if $(x_{1},y),(x_{2},y)\in K_{4}$ and $d_{X}(x_{1},x_{2})<\delta$, then $x_{2}\in C_{G}(u)x_{1}$.
\end{lem}
\begin{proof}
   Suppose  that $M$, $\delta$, $K_{4}$ are given, and $x_{2}\not\in C_{G}(u)x_{1}$ with $d_{X}(x_{1},x_{2})<\delta$. Then by the $H$-property of the unipotent flow (Theorem \ref{h-property202103.1} and Remark \ref{joinings202103.4}), we know that there are   $L>M/\alpha$ and $w\in B^{\mathfrak{n}}_{\lambda}$  such that
      \begin{equation}\label{joinings202012.7}
    d_{X}(\phi^{U}_{n}x_{1},\phi^{w}_{1}\phi^{U}_{n}x_{2})<\epsilon
   \end{equation}
        for $n\in[L,(1+\alpha)L]$.
   Next, we shall find some qualified $x_{1},x_{2}\in X$ such that the distance between $\phi^{U}_{n}x_{1}$ and $\phi^{w}_{1}\phi^{U}_{n}x_{2}$ is at least $\epsilon$. This will lead to a contradiction.

   First, applying the ergodic theorem, there is a measurable set $K_{3}\subset \Omega$ with $\rho(K_{3})>1-\alpha/2(100+\alpha)$, a number $M_{1}>0$ such that
   \begin{equation}\label{joinings202103.7}
     \frac{1}{n}\left|\{k\in [0,n]:(\phi^{U}_{1}\times S)^{k}(x,y)\in K_{2}\}\right|>   \frac{9}{10}
   \end{equation}
  for $(x,y)\in K_{3}$ and $n>M_{1}$. Applying the ergodic theorem one more time, there is a measurable set $K_{4}\subset \Omega$ with $\rho(K_{4})>0$, a number $M_{2}>0$ such that
  \begin{equation}\label{joinings202103.6}
    \frac{1}{n}\left|\{k\in [0,n]:(\phi^{U}_{1}\times S)^{k}(x,y)\in K_{3}\}\right|>   1-\frac{\alpha}{10+\alpha}
  \end{equation}
  for $(x,y)\in K_{4}$ and $n>M_{2}$.

   Choose $M=\max\{M_{1},M_{2}\}$ and then $L>M/\alpha$ and $\delta=\delta(\epsilon,M)>0$ as obtained from the H-property (Theorem \ref{h-property202103.1}). Let  $(x_{1},y),(x_{2},y)\in K_{4}$ with $d_{X}(x_{1},x_{2})<\delta$.
   Then replacing $n$ by $(1+\alpha/10)L$ and applying (\ref{joinings202103.6}), we know that
     \[(\phi^{U}_{1}\times S)^{s}(x_{1},y),(\phi^{U}_{1}\times S)^{t}(x_{2},y)\in K_{3}\]
    for some  integers $s,t\in [L,(1+\alpha/10)L]$. Further,  replacing the interval  $[0,n]$ by $[s,(1+\alpha)L]$ (resp. $[t,(1+\alpha)L]$) and applying (\ref{joinings202103.7}), we know that
    \[\frac{1}{(1+\alpha)L-s}\left|\{k\in [s,(1+\alpha)L]:(\phi^{U}_{1}\times S)^{k}(x_{1},y)\in K_{2}\}\right|>   \frac{9}{10}\]
     \[\frac{1}{(1+\alpha)L-t}\left|\{k\in [t,(1+\alpha)L]:(\phi^{U}_{1}\times S)^{k}(x_{2},y)\in K_{2}\}\right|>   \frac{9}{10}.\]
  It follows that   there exists $n\in [(1+\alpha/10)L,(1+\alpha)L]$ such that
  \[(\phi^{U}_{1}\times S)^{n}(x_{1},y),\  (\phi^{U}_{1}\times S)^{n}(x_{2},y)\in K_{2}.\]
Then by (\ref{joinings202012.5}), we have
 \[ d_{X}(\phi^{U}_{n}x_{1},\phi^{w}_{1}\phi^{U}_{n}x_{2})\geq\epsilon\]
 which contradicts (\ref{joinings202012.7}).
\end{proof}

Recall that via disintegration (cf. (\ref{joinings202012.8})), we have
    \[ \rho=\int_{Y}\rho_{y}d\nu(y).\]
Then by the ergodic theory, we have
\begin{lem}\label{joinings202103.8} There exists a $\nu$-conull set $\Theta\subset Y$  and   $n>0$ such that   for any $y\in \Theta$, there exist $x_{1}^{y},\ldots,x_{n}^{y}$ in the support of $\rho_{y}$ with
  \[\rho_{y}(C_{G}(u)x_{i}^{y})=\frac{1}{n}\]
  for $i=1,\ldots,n$.
\end{lem}
\begin{proof}
   Let $f:Y\rightarrow\mathbf{R}$ be defined by
   \[f:y\mapsto\sup_{x\in X}\rho_{y}(C_{G}(u)x).\]
   By Lemma \ref{joinings202012.6}, we know that for $y\in K_{4}^{Y}\coloneqq\{y\in Y:\rho_{y}\{x\in X:(x,y)\in K_{4}\}>0\}$, $f(y)>0$. Note also that $\nu(K_{4}^{Y})>0$ and $f$ is $S$-invariant. By the ergodicity,   $f$ is a positive constant, say $f\equiv c$, on a $\nu$-conull set $\Theta_{1}\subset Y$.

   Next, consider
     \[D\coloneqq\{(x,y)\in X\times Y:y\in\Theta_{1},\ \rho_{y}(C_{G}(u)x)=c\}.\]
     Then $D$ is $(\phi^{U}_{1}\times S)$-invariant and $\rho(D)>0$. Thus, $\rho(D)=1$. Next, define
     \[\Theta\coloneqq\{y\in \Theta_{1}:\rho_{y}\{x\in X:(x,y)\in D\}=1\}.\]
     Then $\Theta\subset Y$ is an $S$-invariant $\nu$-conull set. Thus, for any $y\in\Theta$,  we have
     \[\rho_{y}(C_{G}(u)x)\equiv c\]
     for any $x\in X$ with $(x,y)\in D$. It forces $n=1/c$ to be an integer. Besides, for any $y\in\Theta$, there are only finitely many points  $x_{1}^{y},\ldots,x_{n}^{y}$ with
  \[\rho_{y}(C_{G}(u)x_{i}^{y})=\frac{1}{n}\]
  for $i=1,\ldots,n$.
\end{proof}

 Thus, by Lemma \ref{joinings202103.8}, we see that $\rho_{y}$ supports on $\bigsqcup_{i=1}^{n}C_{G}(u)x_{i}^{y}$ whenever $y\in\Theta$. With a further effort, we observe that these $\rho_{y}$ must have a compact support.

\begin{proof}[Proof of Theorem \ref{joinings202012.1}]
   For a Borel measurable subset $A\subset C_{G}(u)$, consider the map $f_{A}:X\times Y\rightarrow\mathbf{R}^{+}$ be defined by
   \[f_{A}:(x,y)\mapsto \rho_{y}(Ax).\]
    Note that since $\rho$ is $(\phi^{U}_{1}\times S)$-invariant, we have
  \[(\phi^{U}_{1})_{\ast}\rho_{y}=\rho_{Sy}.\]
  It follows that
  \[f_{A}(x,y)=\rho_{y}(Ax)=\rho_{Sy}(\phi^{U}_{1}Ax)=\rho_{Sy}(A\phi^{U}_{1}x)=f_{A}(\phi^{U}_{1}x,Sy).\]
  In other words, $f_{A}$ is $(\phi^{U}_{1}\times S)$-invariant and therefore is $\rho$-a.e. a constant, say $m(A)$. Thus, for any $A\in\mathcal{B}(C_{G}(u))$, there exists a $\rho$-conull set $\Omega_{A}\subset X\times Y$, such that
  \begin{equation}\label{joinings202103.9}
   \rho_{y}(Ax)\equiv m(A)
  \end{equation}
  for $(x,y)\in \Omega_{A}$.

  Next, we consider the fundamental domain, i.e. a Borel subset $F\subset C_{G}(u)$ such that the natural map $F\rightarrow C_{G}(u)/(C_{G}(u)\cap \Gamma)$ defined by $g\mapsto g\Gamma$ is bijective. Then since $\mathcal{B}(F)$ is countably generated, by \textit{Carath\'{e}odory's extension theorem}\index{Carath\'{e}odory's extension theorem}, we know that
 $m:\mathcal{B}(F)\rightarrow \mathbf{R}^{+}$ is a measure. Besides, it follows from (\ref{joinings202103.9}) that there exists a $\rho$-conull set $\Omega\subset X\times Y$, such that
  \begin{equation}\label{joinings202012.11}
   \rho_{y}(Ax)\equiv m(A)
  \end{equation}
  for $(x,y)\in \Omega$, $A\in\mathcal{B}(F)$.

   Now assume that (\ref{joinings202012.11}) holds for $(x,y),(gx,y)\in \Omega$ and $g\in C_{G}(u)$. Then
   \[m(A)=\rho_{y}(Agx)=m(Ag)\]
   for $A\in\mathcal{B}(F)$. In other words, $m$ is $g$-(right) invariant and so is (right) Haar. Note that $C_{G}(u)$ is unimodular (since its Lie algebra $C_{\mathfrak{g}}(U)$ is a direct sum of a compact and a nilpotent Lie subalgebra). We conclude that $m$ is also a (left) Haar measure, and therefore $\rho_{y}$ is (left) Haar on $C_{G}(u)x$ for $(x,y)\in\Omega$.

   Let $C^{\rho}$ be the stabilizer of $m$. Then the above result shows that $\rho$ is $(C^{\rho}\times\id)$-invariant. Thus, according to Corollary \ref{joinings202103.3}, $C^{\rho}$ must be compact. This finishes the proof of Theorem \ref{joinings202012.1}.
\end{proof}

Using Theorem \ref{joinings202012.1}, for any   ergodic joining $\rho$  of $\phi_{1}^{U}$ and  $S$ on $X\times Y$, we obtain an ergodic joining $\overline{\rho}\coloneqq\pi_{\ast}\rho$ of $\phi_{1}^{U}$ and  $S$ on $C^{\rho}\backslash X\times Y$ under the natural projection $\pi:   X\times Y\rightarrow C^{\rho}\backslash X\times Y$. Moreover, when $\overline{\rho}\neq \overline{\mu}\times\nu$ is not the product measure, it is a finite extension of $\nu$, i.e. $\supp\overline{\rho}_{y}$ consists of exactly $n$ points $\overline{x}_{1}^{y},\ldots,\overline{x}_{n}^{y}$ for $\nu$-a.e. $y\in Y$ (without loss of generality, we shall assume that it holds for all $y\in Y$). Note that $y\mapsto\overline{x}_{i}^{y}$ need not be measurable. However, this can be resolved by using \textit{Kunugui's theorem} (see \cite{Kunugui1940SurUP}, \cite{Kallman1975CertainQS}).

Therefore, let $\overline{X}\coloneqq C^{\rho}\backslash X$, $\pi_{X}:X\times Y\rightarrow X$, $\pi_{\overline{X}}:\overline{X}\times Y\rightarrow \overline{X}$, $\pi_{Y}:\overline{X}\times Y\rightarrow Y$ be the natural projections.
By Kunugui's theorem, we are able to find $\hat{\psi}_{i}:Y\rightarrow \overline{X}\times Y$ for $i=1,\ldots,n$ such that $\pi_{Y}\circ\hat{\psi}_{i}=\id$ and $\hat{\psi}_{i}(Y)\cap \hat{\psi}_{j}(Y)=\emptyset$ whenever $i\neq j$. Let
\begin{equation}\label{joinings202104.23}
 \Omega_{i}\coloneqq \hat{\psi}_{i}(Y),\ \ \ \overline{\psi}_{i}\coloneqq\pi_{\overline{X}}\circ\hat{\psi}_{i}.
\end{equation}
 Then $\rho(\Omega_{i})=1/n$,   $ \bigcup \Omega_{i}=\supp\overline{\rho}$, and $\Omega\cap\supp\overline{\rho}_{y}$ consists of exactly one point.
Next, we can apply Kunugui's theorem again and obtain $\psi_{i}:Y\rightarrow X$ so that $P_{X}\circ \psi_{i}=\overline{\psi}_{i}$ where $P_{X}:X\rightarrow\overline{X}$.

\section{Shearing property II, time changes of unipotent flows}\label{joinings202103.13}
We continue to study the  shearing property of  unipotent flows. More precisely, we shall study the shearing in directions different from Section \ref{joinings202103.10} and deduce the following Proposition   \ref{joinings202104.57}. In fact, in Section \ref{joinings202103.10}, we study the shearing between points of the form $(x,y),(gx,y)\in X\times Y$ for some $g\in G_{X}$ sufficiently close to the identity.  Thus, the information basically comes from the $X$-factor. However, in this section, we shall study the shearing between points of the form $(\psi(y),y),(\psi(gy),gy)\in C^{\rho}\backslash X \times Y$   where $\psi:Y\rightarrow C^{\rho}\backslash X$ is a measurable map and $g\in G_{X}$ is sufficiently close to the identity. Thus, the time-change   on $Y$ comes into play.   The technique used in Proposition \ref{joinings202104.57} generalizes the ideas in \cite{ratner1986rigidity} \cite{tang2020new}, and provides us a quantitative estimate of a unipotent shearing on the double quotient space $C^{\rho}\backslash G_{X}/\Gamma_{X}$. Roughly speaking, Proposition \ref{joinings202104.57} helps us better understand the non-shifting time under a unipotent shearing.

\subsection{Preliminaries}
We start with a combinatorial result.  Let $I$ be an interval in $\mathbf{R}$ and let $J_{i},J_{j}$ be disjoint subintervals of $I,J_{i}=[x_{i},y_{i}]$, $y_{i}<x_{j}$ if $i<j$. Denote
\[d(J_{i},J_{j})\coloneqq \Leb[y_{i},x_{j}]=x_{j}-y_{i}.\]
For a collection $\beta$ of finitely many intervals, we define
\[|\beta|\coloneqq \Leb\left(\bigcup_{J\in \beta}J\right).\]
Besides, for a collection $\beta$ of finitely many intervals, an interval $I$, let
\[\beta\cap I\coloneqq\{I\cap J:J\in\beta\}.\]
\begin{prop}
  [Existence of large intervals,  Solovay  \cite{ratner1979cartesian}]\label{joinings202104.27}
   Given $\eta\in(0,1)$, $\zeta\in(0,1)$, there is $\theta=\theta(\zeta,\eta)\in(0,1)$ such that if $I$ is an interval of length $\lambda\gg 1 $   and $\alpha=\{J_{1},\ldots,J_{n}\}=\mathcal{G}\cup\mathcal{B}$ is a partition of $I$ into good and bad intervals such that
  \begin{enumerate}[\ \ \ (1)]
    \item   for any two good intervals $J_{i},J_{j}\in\mathcal{G}$, we have
    \begin{equation}\label{joinings202104.28}
      d(J_{i},J_{j})\geq[\min\{\Leb(J_{i}),\Leb(J_{j})\}]^{1+\eta},
    \end{equation}
    \item $ \Leb(J)\leq \zeta\lambda$ for any good interval $J\in\mathcal{G}$,
    \item $ \Leb(J)\geq 1$ for any bad interval $J\in\mathcal{B}$,
  \end{enumerate}
  then the measure of bad intervals $\Leb(\bigcup_{J\in \mathcal{B}}J)\geq \theta\lambda$. More precisely, we can take
   \[\theta=\theta(\zeta,\eta)= \prod_{n=0}^{\infty} \left(1+ C\zeta^{n\eta}\right)^{-1}\]
   for some constant $C>0$ (independent of $\zeta,\eta$).
\end{prop}
\begin{proof}
 Assume that  $\zeta^{1-k} \leq \lambda\leq \zeta^{-k}$ for some $k\geq1$.  Let \[\mathcal{G}_{n}\coloneqq\{J\in\mathcal{G}:\zeta^{n+1}\lambda\leq |J|\leq \zeta^{n}\lambda\},\]
  $\mathcal{G}_{\leq n}\coloneqq\bigcup_{i=1}^{n}\mathcal{G}_{i}$, and $\mathcal{B}_{\leq n}$ be the collection of remaining intervals forming $I\setminus\bigcup_{J\in\mathcal{G}_{\leq n}}J$.  Then  for $n\in\mathbf{N}$, $J\in \mathcal{B}_{\leq n}$, by (\ref{joinings202104.28}), we have
    \begin{align}
\frac{|\mathcal{B}_{\leq n+1}\cap J|}{\Leb(J)}=&\frac{|\mathcal{B}_{\leq n+1}\cap J|}{|\mathcal{G}_{n+1}\cap J|+|\mathcal{B}_{\leq n+1}\cap J|}= \left(1+\frac{|\mathcal{G}_{n+1}\cap J|}{|\mathcal{B}_{\leq n+1}\cap J|}\right)^{-1}  \;\nonumber\\
\geq& \left(1+\frac{l\zeta^{n+1}\lambda}{(l-1) \zeta^{(n+2)(1+\eta)}\lambda^{1+\eta}}\right)^{-1} = \left(1+ C\zeta^{(k-n)\eta}\right)^{-1}  \;  \nonumber
\end{align}
where $l\geq 2$ is the number of intervals in $\mathcal{G}_{n+1}\cap J$, and $C>0$ is some constant depending on $\eta$ and $\zeta$. One can also show that when $k=0,1$, we have a similar relation. By summing over $J\in\mathcal{B}_{\leq n}$, we obtain
\[\frac{|\mathcal{B}_{\leq n+1}|}{|\mathcal{B}_{\leq n}|}\geq\left(1+ C\zeta^{(k-n)\eta}\right)^{-1}. \]
Note that by (2), $|\mathcal{B}_{\leq 0}|=\lambda$, and by (3), $\mathcal{B}_{\leq n}=\mathcal{B}_{\leq n+1}$ for all $n\geq k$. We calculate
\[|\mathcal{B}|=|\bigcap_{k\geq 0}\mathcal{B}_{\leq k}|=\lim_{k\rightarrow\infty}|\mathcal{B}_{\leq k}|=\prod_{n=0}^{\infty}\frac{|\mathcal{B}_{\leq n+1}|}{|\mathcal{B}_{\leq n}|}\cdot \lambda\geq \prod_{n=0}^{k} \left(1+ C\zeta^{(k-n)\eta}\right)^{-1}\cdot \lambda.\]
Now note that
\[\theta(\zeta,\eta)=\prod_{n=0}^{\infty} \left(1+ C\zeta^{n\eta}\right)^{-1}\leq\prod_{n=0}^{k} \left(1+ C\zeta^{(k-n)\eta}\right)^{-1} \]
 and the proposition follows.
\end{proof}

In light of (\ref{joinings202104.28}), we make the following definition.
\begin{defn}[Effective gaps between intervals]
  We say that two intervals $I,J\subset\mathbf{R}$ have an \textit{effective gap}\index{effective gap} if
\[   d(I,J)\geq[\min\{\Leb(I),\Leb(J)\}]^{1+\eta}\]
for some $\eta>0$.
Later, we shall obtain some quantitative results relative to the effective gap.
\end{defn}
\begin{rem}\label{joinings202103.15}
  It is worth noting that if $\mathcal{A}$ and $\mathcal{B}$ are collections of intervals with effective gaps, then the intersection $\mathcal{A}\cap\mathcal{B}\coloneqq\{I\cap J:I\in\mathcal{A},\ J\in\mathcal{B}\}$ also have effective gaps. More generally, assume that $\mathcal{A}$ and $\mathcal{B}$ are collections of intervals. If $J_{1},J_{2}\in\mathcal{A}\cap\mathcal{B}$ have an effective gap, then there is a pair of intervals $I_{1},I_{2}$, either in $\mathcal{A}$ or in $\mathcal{B}$,  such that $J_{1}\subset I_{1}$, $J_{2}\subset I_{2}$ and  $I_{1},I_{2}$ have an effective gap.
\end{rem}

In the following, we shall use the asymptotic notation:
\begin{itemize}
  \item  $A\ll B$ or $A=O(B)$ means there is a constant $C>0$ such that $A\leq CB$ (we also write $A\ll_{\kappa} B$ if the constant $C(\kappa)$ depends on some coefficient $\kappa$);
  \item  $A=o(B)$ means that $A/B\rightarrow0$ as $B\rightarrow0$;
  \item  $A\asymp B$ means there is a constant $C>1$ such that $C^{-1}B\leq A\leq CB$;
  \item $A\approx0$ means $A\in (0,1)$ close to $0$, and $A\approx1$ means $A\in (0,1)$ close to $1$.
\end{itemize}
 Similar to  \cite{tang2020new}, we need to following quantitative property of polynomials.

\begin{lem}\label{dynamical systems4} Fix numbers $R_{0}>0, \kappa\in(0,1]$, a real polynomial $p(x)=v_{0}+v_{1}x+\cdots+v_{k}x^{k}\in \mathbf{R}[x]$. Assume further that there exist intervals $[0,\overline{l}_{1}]\cup[l_{2},\overline{l}_{2}]\cup\cdots\cup[l_{m},\overline{l}_{m}]$ such that
\begin{equation}\label{dynamical systems3}
  |p(t)|\ll\max\{R_{0},t^{1-\kappa}\}\ \  \text{ iff } \ \ t\in[0,\overline{l}_{1}]\cup[l_{2},\overline{l}_{2}]\cup\cdots\cup[l_{m},\overline{l}_{m}]
\end{equation}
Then $\overline{l}_{1}$  has the  lower bound $l$  depending on $\max_{i}|v_{i}|$, $R_{0}$, $\kappa$ and the implicit constant such that $l\nearrow\infty$ as $\max_{i}|v_{i}|\searrow0$ for  fixed $R_{0},\kappa$.
Besides, $m\leq k$ and we have
   \begin{enumerate}[\ \ \ (1)]
     \item  $|v_{i}|\ll_{k,\kappa} R_{0}\overline{l}_{1}^{1-i-\kappa} $ for all $1\leq i\leq k$;
     \item Fix $\eta\approx0$. For $1\leq j\leq k-1$, sufficiently large $\overline{l}_{j}$, assume that the intervals $[0,\overline{l}_{j}]$ and $[l_{j+1},\overline{l}_{j+1}]$ do not have an effective gap:
     \begin{equation}\label{dynamical systems6}
       l_{j+1}-\overline{l}_{j}\leq \min\{\overline{l}_{j},\overline{l}_{j+1}-l_{j+1}\}^{1+\eta}.
     \end{equation}
     Then there exists $1\approx\xi(\eta,k)\in(0,1)$ with $\xi(\eta,k)\rightarrow1$ as $\eta\rightarrow0$ such that
     \[|v_{i}|\ll_{k,\kappa} \overline{l}_{j}^{\xi(\eta,k)(1-i-\kappa)}\]
for all $1\leq i\leq k$.
   \end{enumerate}
\end{lem}
\begin{proof}
The number $m$ of intervals in (\ref{dynamical systems3}) can be bounded by $k$ via an elementary argument of polynomials.

(1)   Let $F (x)\coloneqq  v_{1}(\overline{l}_{1}x)^{\kappa} +\cdots+v_{k}(\overline{l}_{1}x)^{k-1+\kappa}$ for $x\in[0,1]$.  Then we have
\[\left(
            \begin{array}{c}
             v_{1}\overline{l}_{1}^{\kappa} \\
              v_{2}\overline{l}_{1}^{1+\kappa} \\
             \vdots \\
             v_{k}\overline{l}_{1}^{k-1+\kappa} \\
            \end{array}
          \right)=\left[
            \begin{array}{cccc}
              (1/k)^{\kappa}&  (1/k)^{1+\kappa} &  \cdots &  (1/k)^{k-1+\kappa} \\
              (2/k)^{\kappa}&  (2/k)^{1+\kappa} &  \cdots &  (2/k)^{k-1+\kappa} \\
                \vdots&  \vdots &  \ddots &  \vdots \\
                 1&  1 &  \cdots &  1 \\
            \end{array}
          \right]^{-1}\left(
            \begin{array}{c}
              F(1/k) \\
              F(2/k) \\
             \vdots \\
             F(1) \\
            \end{array}
          \right).\]
   By (\ref{dynamical systems3}), we know that   $|F (1/k)|,|F (2/k)|,\cdots,|F(1)|\ll R_{0}$. Thus, we obtain  $|v_{i}|\ll_{k,\kappa} R_{0} \overline{l}_{1}^{1-i-\kappa} $ for all $1\leq i\leq k$.

   (2) This follows  by induction. Assume that the statement holds for $j-1$. For $j$,  the only difficult situation is when $\overline{l}_{j}\leq l_{j+1}-\overline{l}_{j}$ and $\overline{l}_{j+1}-l_{j+1}\leq l_{j+1}-\overline{l}_{j}$.   If this is the case, then
\[ \overline{l}_{j+1}= (\overline{l}_{j+1}-l_{j+1})+(l_{j+1}-\overline{l}_{j})+\overline{l}_{j}\leq 3\overline{l}_{j}^{1+\eta}.\]
Thus, by induction hypothesis, we get
\[|v_{i}|\ll \overline{l}_{j}^{\xi(\eta,j)(1-i-\kappa)}\ll \overline{l}_{j+1}^{\frac{\xi(\eta,j)}{1+\eta}(1-i-\kappa)}\]
for all $1\leq i\leq k$.
    \end{proof}

\subsection{Effective estimates of shearing phenomena}\label{joinings202104.33}
Now we begin to study the shearing between two nearby orbits of time-changes of unipotent flows. Let $G=SO(n,1)$.
First, since all maximal compact subgroups of $C_{G}(U)$ are conjugate, we can assume without loss of generality that  $C^{\rho}$ is in the compact group generated by $\mathfrak{k}^{\perp}_{C}$ . Thus, via (\ref{joinings202103.2}) (\ref{joinings202105.14}) and  (\ref{joinings202103.12}), we consider the   decomposition
\[\mathfrak{g}=\mathfrak{sl}_{2}\oplus V^{\perp\rho}\oplus\Lie(C^{\rho}),\ \ \ V^{\perp\rho}=\sum_{i} V_{i}^{0\perp\rho}\oplus\sum_{j} V_{j}^{2}\]
\[\mathfrak{k}^{\perp}_{C}=\mathfrak{k}^{\perp\rho}_{C}\oplus\Lie(C^{\rho})\]
where $\Lie(C^{\rho})$ denotes the Lie algebra of $C^{\rho}$ and note that $\Lie(C^{\rho})$ consists of weight $0$ spaces. Since $C^{\rho}$ is compact, there  is a $G$-right invariant metric $d_{C^{\rho}\backslash G}(\cdot,\cdot)$ on $C^{\rho}\backslash G$. Let $P:G\rightarrow C^{\rho}\backslash G$ be the natural projection
\[P:g\mapsto C^{\rho}g\eqqcolon\overline{g}.\]
 Then, for $g_{x},g_{y}\in  G$, we have
\[d_{C^{\rho}\backslash G}(\overline{g_{x}},\overline{g_{y}})=d_{C^{\rho}\backslash G}( C^{\rho}g_{x} ,C^{\rho}g_{y})=d_{C^{\rho}\backslash G}(C^{\rho}g_{x}g_{y}^{-1},C^{\rho})=d_{C^{\rho}\backslash G}(\overline{g_{x}g_{y}^{-1}},\overline{e}).\]
Moreover, $dP$ induces an isometry between $\mathfrak{sl}_{2}+ V^{\perp\rho}$ and $T_{\overline{e}}(C^{\rho}\backslash G)$. See for example  \cite{gallier2019differential} for more details.

Assume  $\overline{g}\in B_{C^{\rho}\backslash G}(e,\epsilon)$   for sufficiently small $0<\epsilon$. Since $C^{\rho}$   in fact commutes with $SO_{0}(2,1)$, we can identify
\begin{equation}\label{dynamical systems1}
  \overline{g}=C^{\rho}h\exp v
\end{equation}
for some $h\in B_{SO_{0}(2,1)}(e,\epsilon)$  and $v\in B_{V^{\perp\rho}}(0,\epsilon)$. Besides, for $h=\left[
            \begin{array}{ccc}
              a&  b \\
             c&  d\\
            \end{array}
          \right]\in B_{SO_{0}(2,1)}(e,\epsilon)$, we must have  $|b|,|c|<\epsilon$, $1-\epsilon<|a|,|d|<1+\epsilon$.

Next, let $t(s)\in\mathbf{R}^{+}$ be a function of $s\in\mathbf{R}^{+}$. Then we want to study the difference $ u^{t}\overline{g}u^{-s}$ of two nearby orbits of time-changes of unipotent flows. By (\ref{dynamical systems5}), we have
\begin{multline}
  u^{t}\overline{g}u^{-s}=C^{\rho}u^{t}h\exp v u^{-s}=C^{\rho}(u^{t}hu^{-s})(u^{s}\exp(v)u^{-s}) \\ \label{joinings202105.21}
 = C^{\rho}(u^{t}hu^{-s})\exp(\Ad u^{s}.v) =
 C^{\rho}(u^{t}hu^{-s})\exp\left(\sum_{n=0}^{\varsigma}\sum_{i=0}^{n}b_{i}\binom{n}{i}s^{n-i}v_{n}\right).
\end{multline}
Then one may conclude that  $u^{t}\overline{g}u^{-s}< \epsilon$ if and only if
\begin{equation}\label{joinings202104.61}
  u^{t}hu^{-s}\ll\epsilon,\ \ \ \Ad u^{s}.v=\sum_{n=0}^{\varsigma}\sum_{i=0}^{n}b_{i}\binom{n}{i}s^{n-i}v_{n}\ll\epsilon
\end{equation}
where $\overline{g}\ll\epsilon$ for $g\in G$ means  $d_{C^{\rho}\backslash G}(\overline{g},e)\ll\epsilon$.
Therefore, later on, we shall split the elements closing to the identity into two parts, say the $SO(2,1)$-part and the $V^{\perp\rho}$-part.

As shown in (\ref{joinings202104.61}), we   consider the elements of the form $u^{t}hu^{-s}\in B_{SO(2,1)}(e,\epsilon)$.
   One may calculates
          \begin{align}
u^{t}hu^{-s}= &\left[
            \begin{array}{ccc}
              1&   \\
             t&  1\\
            \end{array}
          \right]\left[
            \begin{array}{ccc}
              a&  b \\
             c&  d\\
            \end{array}
          \right]\left[
            \begin{array}{ccc}
              1&    \\
            -s&  1\\
            \end{array}
          \right]\;\nonumber\\
=& \left[
            \begin{array}{ccc}
              a-bs&  b \\
             c+(a-d)s-bs^{2}+(t-s)(a-bs)&  d+bt\\
            \end{array}
          \right]. \;\label{joinings202104.60}
\end{align}
   If we further impose the H\"{o}lder inequality $|s-t|\ll_{\kappa}\max\{R_{0},s^{1-\kappa}\}$ for some $R_{0}>\epsilon$ (see Section \ref{joinings202103.14} or (\ref{joinings202104.5})), then we have the  crude estimate
   \begin{align}
 &|-bs^{2}+(a-d)s+c+(-bs+a)(t-s)|<\epsilon \;\nonumber\\
\Rightarrow\ \ \ & |-bs^{2}+(a-d)s|-|c|-|(-bs+a)(t-s)|<\epsilon\;\nonumber\\
\Rightarrow\ \ \ & |-bs^{2}+(a-d)s|<2\epsilon+2|t-s|\;\nonumber\\
\Rightarrow\ \ \ & |-bs^{2}+(a-d)s|\ll_{\kappa}\max\{R_{0},s^{1-\kappa}\}. \;\nonumber
\end{align}
By Lemma \ref{dynamical systems4}, we immediately obtain
\begin{lem}[Estimates for $SO_{0}(2,1)$-coefficients]\label{dynamical systems8}  Given $\kappa\approx0$, $R_{0}>0$, $\epsilon\approx0$, a matrix $h=\left[
            \begin{array}{ccc}
              a&  b \\
             c&  d\\
            \end{array}
          \right]\in B_{SO(2,1)}(e,\epsilon)$, then the solutions $s\in[0,\infty)$ of the following inequality
          \begin{equation}\label{timechange17}
            |-bs^{2}+(a-d)s|\ll_{\kappa} \max\{R_{0},s^{1-\kappa}\}
          \end{equation}
   consist of at most two intervals, say $[0,\overline{l}_{1}(h)]\cup[l_{2}(h),\overline{l}_{2}(h)]$,  where  $\overline{l}_{1}$  has the  lower bound $l(\epsilon,R_{0},\kappa)$ such that $l(\epsilon,R_{0},\kappa)\nearrow\infty$ as $\epsilon\searrow0$ for fixed $R_{0},\kappa$.  Moreover, we have
   \begin{enumerate}[\ \ \ (1)]
     \item  $|b|\ll_{\kappa} \overline{l}_{1}^{-1-\kappa}$ and $|a-d|\ll_{\kappa} \overline{l}_{1}^{-\kappa}$;
     \item If we further assume     that the intervals $[0,\overline{l}_{1}]$ and $[l_{2},\overline{l}_{2}]$ do not have an effective gap (\ref{dynamical systems6}), i.e.   $l_{2}-\overline{l}_{1}\leq \min\{\overline{l}_{1},\overline{l}_{2}-l_{2}\}^{1+\eta}$ for some $\eta\approx0$, then
         \[|b|\ll_{\kappa} \overline{l}_{2}^{\xi(\eta)(-1-\kappa)},\ \ \ |a-d|\ll_{\kappa} \overline{l}_{2}^{\xi(\eta)(-\kappa)}.\]
   \end{enumerate}
\end{lem}
Next, we study the situation when $Ad u^{s}.v\ll \epsilon$. Again by Lemma \ref{dynamical systems4}, we have
\begin{lem}[Estimates for $V^{\perp\rho}$-coefficients]\label{dynamical systems7}
  Fix $v=b_{0}v_{0}+\cdots+b_{\varsigma}v_{\varsigma}\in B_{V_{\varsigma}}(0,\epsilon)$. Assume that
   \[Ad u^{s}.v\ll \epsilon\ \ \ \text{ iff }\ \ \ s\in[0,\overline{l}_{1}(v)] \cup\cdots\cup[l_{m}(v),\overline{l}_{m}(v)] \]
   where  $\overline{l}_{1}$  has the  lower bound $l(\epsilon,R_{0},\kappa)$ such that $l(\epsilon,R_{0},\kappa)\nearrow\infty$ as $\epsilon\searrow0$ for fixed $R_{0},\kappa$. Then  $m=m(v)$ is bounded by a constant depending on $\varsigma$. Moreover, for   $1\leq j\leq \varsigma-1$, the
   intervals $[0,\overline{l}_{j}]$ and $[l_{j+1},\overline{l}_{j+1}]$ do not have an effective gap (\ref{dynamical systems6}), i.e.   $l_{j+1}-\overline{l}_{j}\leq \min\{\overline{l}_{j},\overline{l}_{j+1}-l_{j+1}\}^{1+\eta}$, then  we have
   \[|b_{i}|\ll_{\varsigma,\kappa} \overline{l}_{j}^{\xi(\eta,\varsigma)(-\varsigma+i)}.\]

\end{lem}

Next, we shall combine the results of Lemma \ref{dynamical systems8} and   \ref{dynamical systems7}.   The basic idea is to consider the intersection of the collections of intervals obtained from the above lemmas. For simplicity, we assume that ``$V^{\perp\rho}$-part" consists of a single $\mathfrak{sl}_{2}$-irreducible representation. For the general case, we can  repeat  the argument for each $\mathfrak{sl}_{2}$-irreducible representation (cf. Section \ref{joinings202104.3}). First, for $\overline{g}=C^{\rho}h\exp(v)\in C^{\rho}\backslash G$, we write as in Lemma \ref{dynamical systems8} and   \ref{dynamical systems7}
\begin{align}
 u^{t}hu^{-s}\ll\epsilon &\text{ iff } s\in [0,\overline{l}_{1}(h)]\cup[l_{2}(h),\overline{l}_{2}(h)]\;\nonumber\\
 Ad u^{s}.v\ll \epsilon  &\text{ iff } s\in[0,\overline{l}_{1}(v)]\cup\cdots\cup[l_{m(v)}(v),\overline{l}_{m(v)}(v)].\; \nonumber
\end{align}
Write $l_{1}(h)=l_{1}(v)=0$ and we shall consider the family of intervals
\begin{equation}\label{time change174}
  \{[l_{k}(g),\overline{l}_{k}(g)]\}_{k}\coloneqq\{[l_{i}(h),\overline{l}_{i}(h)]\cap[l_{j}(v),\overline{l}_{j}(v)]\}_{i,j}
\end{equation}
where $\overline{l}_{k}(g)<l_{k+1}(g)$ for all $k$. Thus, in particular, $l_{1}(g)=0$ and $[0,\overline{l}_{1}(g)]=[0,\overline{l}_{1}(h)]\cap[0,\overline{l}_{1}(v)]$.

Now assume that there exists $k$ such that $[0,\overline{l}_{k}(g)]$  and $[l_{k+1}(g),\overline{l}_{k+1}(g)]$ do not have an effective gap (\ref{dynamical systems6}), i.e.
\[  l_{k+1}(g)-\overline{l}_{k}(g)\leq \min\{\overline{l}_{k}(g),\overline{l}_{k+1}(g)-l_{k+1}(g)\}^{1+\eta}.\]
Then by Remark \ref{joinings202103.15}, the corresponding ``$SO(2,1)$-part" and ``$V^{\perp\rho}$-part" should not have effective gaps either. More precisely, for the $SO(2,1)$-part, we define
\[i_{\geq k}\coloneqq\min\{i\in\{1,2\}:\overline{l}_{k}(g)\leq \overline{l}_{i}(h)\},\ \ \ i_{\leq k+1}\coloneqq\max\{i\in\{1,2\}:l_{k+1}(g)\geq l_{i}(h)\}.\]
Thus, we know
\[[0,\overline{l}_{k}(g)]\subset[0,\overline{l}_{i_{\geq k}}(h)],\ \ \ [l_{k+1}(g),\overline{l}_{k+1}(g)]\subset[l_{i_{\leq k+1}}(h),\overline{l}_{i_{\leq k+1}}(h)] \]
and hence $[0,\overline{l}_{i_{\geq k}}(h)]$  and $[l_{i_{\leq k+1}}(h),\overline{l}_{i_{\leq k+1}}(h)] $ do not have an effective gap (\ref{dynamical systems6}). Similarly, for the $V^{\perp\rho}$-part, we define
\[j_{\geq k}\coloneqq\min\{j:\overline{l}_{k}(g)\leq \overline{l}_{j}(v)\},\ \ \ j_{\leq k+1}\coloneqq\max\{j:l_{k+1}(g)\geq l_{j}(v)\}.\]
Then we know
\[[0,\overline{l}_{k}(g)]\subset[0,\overline{l}_{j_{\geq k}}(v)],\ \ \ [l_{k+1}(g),\overline{l}_{k+1}(g)]\subset[l_{j_{\leq k+1}}(v),\overline{l}_{j_{\leq k+1}}(v)] \]
and hence $[0,\overline{l}_{j_{\geq k}}(v)]$  and $[l_{j_{\leq k+1}}(v),\overline{l}_{j_{\leq k+1}}(v)] $ do not have an effective gap (\ref{dynamical systems6}). Further, one observes
\begin{align}
[0,\overline{l}_{k}(g)]=&[0,\overline{l}_{i_{\geq k}}(h)]\cap [0,\overline{l}_{j_{\geq k}}(v)]\;\nonumber\\
 [l_{k+1}(g),\overline{l}_{k+1}(g)]= & [l_{i_{\leq k+1}}(h),\overline{l}_{i_{\leq k+1}}(h)]\cap [l_{j_{\leq k+1}}(v),\overline{l}_{j_{\leq k+1}}(v)].\; \nonumber
\end{align}
Now recall by the definition that the number (\ref{time change174}) of intervals in $\{[l_{k}(g),\overline{l}_{k}(g)]\}_{k}$ is bounded by a constant $c(\varsigma)>0$ because the numbers of intervals $\{[l_{i}(h),\overline{l}_{i}(h)]\}_{i}$, $\{[l_{j}(v),\overline{l}_{j}(v)]\}_{j}$   are. Since $\varsigma\leq 2$ when $\mathfrak{g}=\mathfrak{so}(n,1)$, we see that $c(\varsigma)$ is uniformly bounded for all $\varsigma$. Thus, we conclude that the number of intervals in $\{[l_{k}(g),\overline{l}_{k}(g)]\}_{k}$ is uniformly bounded for all $g\in G$.
Then, combining   Lemma \ref{dynamical systems7} and  \ref{dynamical systems8}, we obtain
\begin{lem}[Estimates for $C^{\rho}\backslash G$-coefficients]\label{time change176}
     Let $\kappa\approx0$, $R_{0}>0$, $\epsilon\approx0$,  $\overline{g}=C^{\rho}h\exp v\in B_{C^{\rho}\backslash G}(e,\epsilon)$ be as above, where
   \[h=\left[
            \begin{array}{ccc}
              a&  b \\
             c&  d\\
            \end{array}
          \right]\in SO_{0}(2,1),\ \ \ v=b_{0}v_{0}+\cdots+b_{\varsigma}v_{\varsigma}\in V_{\varsigma}.\]
          Next, let $t(s)\in\mathbf{R}^{+}$ be  a function of $s\in\mathbf{R}^{+}$ which satisfies the effectiveness \[|s-t(s)|\ll_{\kappa}\max\{R_{0},s^{1-\kappa}\}.\]
      Then there exist intervals $\{[l_{k}(g),\overline{l}_{k}(g)]\}_{k}$ such that
          \begin{equation}\label{joinings202106.175}
            u^{t}\overline{g}u^{-s}<\epsilon,\ \ \  \text{ implies }\ \ \ s\in \bigcup_{k}[l_{k}(g),\overline{l}_{k}(g)]
          \end{equation}
           where $\overline{l}_{1}$  has the  lower bound $l(\epsilon,R_{0},\kappa)$ such that $l(\epsilon,R_{0},\kappa)\nearrow\infty$ as $\epsilon\searrow0$ for fixed $R_{0},\kappa$. Besides, $k\leq c$ for some constant $c=c(\mathfrak{g})>0$, and
          \begin{enumerate}[\ \ \ (1)]
            \item   $|b|\ll_{\kappa} \overline{l}_{1}(g)^{-1-\kappa}$, $|a-d|\ll_{\kappa} \overline{l}_{1}(g)^{-\kappa}$, $|b_{i}|\ll_{\varsigma,\kappa} \overline{l}_{1}(g)^{-\varsigma+i}$ for all $0\leq i\leq \varsigma$;
     \item If we further assume     that the intervals $[0,\overline{l}_{k}(g)]$ and $[l_{k+1}(g),\overline{l}_{k+1}(g)]$ do not have an effective gap (\ref{dynamical systems6}). Then there exists $1\approx\xi=\xi(\eta)\in(0,1)$ with $\xi\rightarrow1$ as $\eta\rightarrow0$ such that
      \[|b|\ll_{\kappa} \overline{l}_{k}(g)^{-\xi(1+\kappa)},\ \ \ |a-d|\ll_{\kappa} \overline{l}_{k}(g)^{-\xi\kappa},\ \ \ |b_{i}|\ll_{\varsigma,\kappa} \overline{l}_{k}(g)^{-\xi (\varsigma-i)}\]
for all $1\leq i\leq \varsigma$.
          \end{enumerate}
\end{lem}

In practical use, we consider two strictly increasing functions $t(r),s(r)\in\mathbf{R}^{+}$  of $r\in\mathbf{R}^{+}$ satisfying the effective estimates
\begin{equation}\label{joinings202104.9}
|r-t(r)|\ll_{\kappa}\max\{R_{0},r^{1-\kappa}\},\ \ \ |r-s(r)|\ll_{\kappa}\max\{R_{0},r^{1-\kappa}\}.
\end{equation}
It follows that $t$ is also an increasing  function of $s$ and satisfies
\[|t(r)-s(r)|\leq|t(r)-r|+|r-s(r)|\ll_{\kappa}\max\{R_{0},r^{1-\kappa}\}\ll_{\kappa}\max\{R_{0},s(r)^{1-\kappa}\}.\]
Then by Lemma \ref{time change176} and the monotonic nature, we deduce that
\begin{cor}[Change of variables]\label{joinings202104.10}
 Let $\kappa\approx0$, $R_{0}>0$, $\epsilon\approx0$, $\overline{g}=C^{\rho}h\exp v\in B_{C^{\rho}\backslash G}(e,\epsilon)$ be as above, where
   \[h=\left[
            \begin{array}{ccc}
              a&  b \\
             c&  d\\
            \end{array}
          \right]\in SO_{0}(2,1),\ \ \ v=b_{0}v_{0}+\cdots+b_{\varsigma}v_{\varsigma}\in V_{\varsigma}.\]
       Assume  that we have (\ref{joinings202104.9}).
      Then there exist intervals $\{[l_{k}(g),\overline{l}_{k}(g)]\}_{k}$ such that
\begin{equation}\label{joinings202104.8}
  u^{t(r)}\overline{g}u^{-s(r)}<\epsilon\ \ \ \text{ implies }\ \ \ r\in \bigcup_{k}[L_{k}(g),\overline{L}_{k}(g)]
\end{equation}
         where $\overline{L}_{1}$ has the  lower bound $L(\epsilon,R_{0},\kappa)$ such that $L(\epsilon,R_{0},\kappa)\nearrow\infty$ as $\epsilon\searrow0$ for fixed $R_{0},\kappa$. Then we have $k\leq c$ for some constant $c=c(\mathfrak{g})>0$, and
          \begin{enumerate}[\ \ \ (1)]
            \item   $|b|\ll_{\kappa} \overline{L}_{1}(g)^{-1-\kappa}$, $|a-d|\ll_{\kappa} \overline{L}_{1}(g)^{-\kappa}$, $|b_{i}|\ll_{\varsigma,\kappa} \overline{L}_{1}(g)^{-\varsigma+i}$ for all $0\leq i\leq \varsigma$;
     \item If we further assume     that the intervals $[0,\overline{L}_{k}(g)]$ and $[L_{k+1}(g),\overline{L}_{k+1}(g)]$ do not have an effective gap (\ref{dynamical systems6}). Then there exists $1\approx\xi=\xi(\eta)\in(0,1)$ with $\xi\rightarrow1$ as $\eta\rightarrow0$ such that
      \[|b|\ll_{\kappa} \overline{L}_{k}(g)^{-\xi(1+\kappa)},\ \ \ |a-d|\ll_{\kappa} \overline{L}_{k}(g)^{-\xi\kappa},\ \ \ |b_{i}|\ll_{\varsigma,\kappa} \overline{L}_{k}(g)^{-\xi (\varsigma-i)}\]
for all $1\leq i\leq \varsigma$.
          \end{enumerate}
\end{cor}

\subsection{$\epsilon$-blocks and effective gaps}
Let $x\in \overline{X}$, $y\in B_{\overline{X}}(x,\epsilon)$. We say that \textit{$(\overline{g_{x}},\overline{g_{y}})\in C^{\rho}\backslash G\times C^{\rho}\backslash G$ covers $(x,y)$} if $ d_{C^{\rho}\backslash G}(\overline{g_{x}},\overline{g_{y}})<\epsilon$ and $\overline{P}(\overline{g_{x}})=x$, $\overline{P}(\overline{g_{y}})=y$, where $\overline{P}:C^{\rho}\backslash G\rightarrow C^{\rho}\backslash G/\Gamma$ is the projection.
Since
$\Lie(C^{\rho}\backslash G)\cong \mathfrak{sl}_{2}+ V^{\perp\rho}$, given a representative $g_{x}$ of $\overline{g_{x}}$, we may choose $g_{y}\in G$ such that $P(g_{y})=\overline{g_{y}}$ and
\[\log (g_{y}g_{x}^{-1})\in \mathfrak{sl}_{2}+ V^{\perp\rho}.\]
We shall always make such a choice if no further explanation.

\begin{defn}[$\epsilon$-block]\label{joinings202104.11}  Suppose that  $x\in \overline{X}$, $y\in B_{\overline{X}}(x,\epsilon)$,   $(\overline{g_{x}},\overline{g_{y}})$ covers $(x,y)$, and $R\in (0,\infty]$ satisfies
\[d_{C^{\rho}\backslash G}(u^{s(R)}\overline{g_{x}},u^{t(R)}\overline{g_{y}})<\epsilon.\]
 Then we define    the \textit{$\epsilon$-block of $\overline{g_{x}},\overline{g_{y}}$ of length $r$}\index{$\epsilon$-block of $g_{x},g_{y}$ of length $r$}    by
 \[\BL(g_{x},g_{y})\coloneqq  \{(u^{s(r)}\overline{g_{x}},u^{t(r)}\overline{g_{y}})\in C^{\rho}\backslash G\times C^{\rho}\backslash G:0\leq r\leq R\}.\]
 Similarly, we define the \textit{$\epsilon$-block of $x,y$ of length $r$} by
  \[\BL(x,y)\coloneqq P( \BL(g_{x},g_{y}))=\{(u^{s(r)}\overline{g_{x}},u^{t(r)}\overline{g_{y}})\in \overline{X}\times \overline{X}:0\leq r\leq R\}.\]
  In either case, we call $[0,R]$ the corresponding \textit{time interval}\index{time interval of $\epsilon$-blocks} and  define the \textit{length}\index{length of $\epsilon$-blocks} $|\BL|$ of $\BL$ by
 \[|\BL|\coloneqq R.\]
   We also write
 \[\BL(x,y)=\{(x,y),(u^{s(R)}x ,u^{t(R)}y)\}=\{(x,y),(\overline{x},\overline{y})\}\]
 emphasizing that $(x,y)$ is the first and $(\overline{x},\overline{y})$ is the last pair of the block $\BL(x,y)$.
\end{defn}

  For  a pair of $\epsilon$-blocks, a shifting problem may occur.
 \begin{defn}[Shifting]\label{dynamical systems3004} Let $\overline{\BL}^{\prime}=\{(x^{\prime},y^{\prime}),(\overline{x}^{\prime},\overline{y}^{\prime})\}$, $\overline{\BL}^{\prime\prime}=\{(x^{\prime\prime},y^{\prime\prime}),(\overline{x}^{\prime\prime},\overline{y}^{\prime\prime})\}$  be two  $\epsilon$-blocks. Then
$x^{\prime\prime}= u^{s}g_{x^{\prime}}$, $y^{\prime\prime}= u^{t}y^{\prime}$ for some $s,t>0$. Further, there is a unique $\gamma\in\Gamma$ such that
 \begin{equation}\label{dynamical systems1014}
  d_{C^{\rho}\backslash G}(\overline{g_{x^{\prime\prime}}},\overline{g_{y^{\prime\prime}}}\gamma)<\epsilon
 \end{equation}
 where  $g_{x^{\prime\prime}}\coloneqq u^{s}g_{x^{\prime}}$, $g_{y^{\prime\prime}}\coloneqq u^{t}g_{y^{\prime}}$. We  define
 \begin{itemize}
   \item (Shifting) $(x^{\prime},y^{\prime})\overset{\Gamma}{\sim}(x^{\prime\prime},y^{\prime\prime})$ if $\gamma\neq e$ in  (\ref{dynamical systems1014}),
   \item (Non-shifting) $(x^{\prime},y^{\prime})\overset{e}{\sim}(x^{\prime\prime},y^{\prime\prime})$ if $\gamma=e$ in  (\ref{dynamical systems1014}).
 \end{itemize}
\end{defn}

The key observation here is that whenever the difference of $\overline{g_{x}},\overline{g_{y}}$ can be estimated by the length in an appropriate way, a shifting must lead to an effective gap between two $\epsilon$-blocks. This follows from the natural renormalization of unipotent flows via diagonal flows.

\begin{prop}[Shiftings imply effective gaps]\label{joinings202104.29} There are quantities $\eta_{0}\approx0$, $\sigma_{0}\approx0$, $\epsilon_{0}\approx0$, $r_{0}>0$ determined orderly  such that for any
\begin{itemize}
  \item  $\eta\in(0,\eta_{0})$,
  \item  $\sigma\in(0,\sigma_{0}(\eta))$,
  \item  $\epsilon\in(0,\epsilon_{0}(\sigma))$,
\end{itemize}   there exists a compact set $K\subset \overline{X}$ with $\overline{\mu}(K)>1-\sigma$ such that the following holds (see Figure \ref{joinings202107.1}):
\begin{enumerate}[\ \ \ ]
  \item   Assume that  there are two  $\epsilon$-blocks $\overline{\BL}^{\prime}=\{(x^{\prime},y^{\prime}),(\overline{x}^{\prime},\overline{y}^{\prime})\}$, $\overline{\BL}^{\prime\prime}=\{(x^{\prime\prime},y^{\prime\prime}),(\overline{x}^{\prime\prime},\overline{y}^{\prime\prime})\}$ such that the $y$-endpoints lie in $K$ (i.e. $y^{\prime},\overline{y}^{\prime},y^{\prime\prime},\overline{y}^{\prime\prime} \in K$) and satisfy
\begin{equation}\label{joinings202104.13}
   g_{y^{\prime}}= h^{\prime}\exp(v^{\prime})g_{x^{\prime}},\ \ \ g_{y^{\prime\prime}}= h^{\prime\prime}\exp(v^{\prime\prime})g_{x^{\prime\prime}}
\end{equation}
 where $h^{\prime} ,h^{\prime\prime} \in SO_{0}(2,1)$, $v^{\prime} ,v^{\prime\prime} \in V_{\varsigma}$ can be estimated by
 \begin{equation}\label{joinings202104.12}
  h^{\prime} ,h^{\prime\prime}=\left[
            \begin{array}{ccc}
               1+O(r^{-2\eta}) &   O(r^{-1-2\eta})  \\
              O(\epsilon) &   1+O(r^{-2\eta})\\
            \end{array}
          \right],\ \ \ v^{\prime}, v^{\prime\prime} =O(r^{-\xi\varsigma})v_{0}+\cdots+O(\epsilon)v_{\varsigma}
 \end{equation} for some $r> r_{0}(\sigma,\epsilon_{0})$,
 where $\xi=\xi(\eta)\approx1$ is given by Corollary \ref{joinings202104.10}.
  Assume further that  $x^{\prime\prime}=u^{s}\overline{x}^{\prime}$, $y^{\prime\prime}=u^{t}\overline{y}^{\prime}$ and $t\asymp s$. If $\overline{\BL}^{\prime}\overset{\Gamma}{\sim}\overline{\BL}^{\prime\prime}$, then
\begin{equation}\label{joinings202104.18}
 s,t> r^{1+\eta}.
\end{equation}
\end{enumerate}
\end{prop} 
  
  \begin{figure}[H]
\centering
 \tikzstyle{init} = [pin edge={to-,thin,black}]
\begin{tikzpicture}[node distance=4cm,auto,>=latex']
\draw (0,0) coordinate (a0) --node[below] {$\overline{\BL}^{\prime}$} (3,0) coordinate (a2) ;
\filldraw[black] (a0) circle(1pt) node[below] {$x^{\prime}$};
\filldraw[black] (a2) circle(1pt) node[below] {$\overline{x}^{\prime}$};

 \draw [dashed] (3,0) coordinate (a2) -- (9,0) coordinate (a3);
 \draw (9,0) coordinate (a3) --node[below] {$\overline{\BL}^{\prime\prime}$} (12,0) coordinate (a5) ;
\filldraw[black] (a3) circle(1pt) node[below] {$x^{\prime\prime}$};
\filldraw[black] (a5) circle(1pt) node[below] {$\overline{x}^{\prime\prime}$};

 \path[->] (a2) edge [bend left] node [right]{$s$} (a3);

  \draw (0,-1) coordinate (b0) --  (3,-1) coordinate (b2) ;
\filldraw[black] (b0) circle(1pt) node[above] {$y^{\prime}$};
\filldraw[black] (b2) circle(1pt) node[above] {$\overline{y}^{\prime}$};

 \draw [dashed] (3,-1) coordinate (b2) -- (9,-1) coordinate (b3);
 \draw (9,-1) coordinate (b3) --  (12,-1) coordinate (b5) ;
\filldraw[black] (b3) circle(1pt) node[above] {$y^{\prime\prime}$};
\filldraw[black] (b5) circle(1pt) node[above] {$\overline{y}^{\prime\prime}$};

 \path[->] (b2) edge [bend right] node [right]{$t$} (b3);

\end{tikzpicture}
\caption{The solid straight lines are the unipotent orbits in the $\overline{\BL}^{\prime}$ and $\overline{\BL}^{\prime\prime}$ respectively, and the dashed lines are the rest of the unipotent orbits. The bent curves indicate the   length defined by the letters.}
\label{joinings202107.1}
\end{figure}
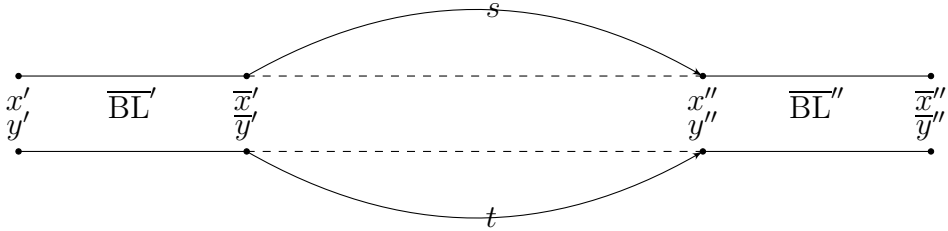

\begin{proof} We only consider $\varsigma=2$. Denote
\begin{equation}\label{joinings202104.14}
 g_{\overline{y}^{\prime}}= \overline{h}^{\prime}\exp(\overline{v}^{\prime})g_{\overline{x}^{\prime}}
\end{equation}
 for $\overline{h}^{\prime} \in SO_{0}(2,1)$, $\overline{v}^{\prime} \in V_{2}$. By Definition \ref{joinings202104.11}, we know that $g_{\overline{y}^{\prime}}, g_{\overline{x}^{\prime}}$ are obtained by the unipotent action on $g_{y^{\prime}}, g_{x^{\prime}}$, and the difference of $g_{\overline{y}^{\prime}}, g_{\overline{x}^{\prime}}$  is controlled by $\epsilon$. Combining   (\ref{joinings202104.12}), we get that
 \begin{equation}\label{joinings202104.16}
 \overline{h}^{\prime} =\left[
            \begin{array}{ccc}
               1+O(\epsilon) &   O(r^{-1-2\eta})  \\
              O(\epsilon) &   1+O(\epsilon)\\
            \end{array}
          \right],\ \ \ \overline{v}^{\prime} =O(r^{-2\xi })v_{0}+O(\epsilon)v_{1}+ O(\epsilon)v_{2}.
 \end{equation}
     Since  $\overline{\BL}^{\prime}\overset{\Gamma}{\sim}\overline{\BL}^{\prime\prime}$  and $g_{x^{\prime\prime}}=u^{s}g_{\overline{x}^{\prime}}$, we get that
  \begin{equation}\label{dynamical systems1033}
   g_{y^{\prime\prime}}= cu^{t}g_{\overline{y}^{\prime}}\gamma\ \ \ \text{ for some }e\neq \gamma\in\Gamma,\ c\in C^{\rho}.
  \end{equation}
 Then by (\ref{joinings202104.13}) (\ref{joinings202104.14}) (\ref{dynamical systems1033}), we have
  \begin{align}
   g_{\overline{y}^{\prime}}=&  \overline{h}^{\prime}\exp (\overline{v}^{\prime})  u^{-s} g_{x^{\prime\prime}}\;\nonumber\\
 g_{\overline{y}^{\prime}} \gamma=&  c^{-1} u^{-t}  h^{\prime\prime}\exp (v^{\prime\prime}) g_{x^{\prime\prime}} . \label{joinings202104.15}
\end{align}
Assume that one of $s,t$ is not greater than $r^{1+\eta}$. Then since $s\asymp t$,  we know
  \begin{equation}\label{joinings202104.17}
    0<s,t\leq O(r^{1+\eta}).
  \end{equation}

Next, we determine the quantities for the proposition.
\begin{itemize}
   \item     (Choice of $\eta$, $\delta$ (also $\eta_{0}$)) Choose a small $\eta\approx0$ that satisfies
  \begin{equation}\label{joinings202104.6}
   1+2\delta<1+2\eta<2\xi(2\eta)
  \end{equation}
    where $\xi(2\eta)$ was defined in  Corollary \ref{joinings202104.10}, and $\delta\coloneqq 3\eta/4$. Here $\eta_{0}\approx0$ can be defined to be the maximal $\eta$ so that (\ref{joinings202104.6}) holds.
  \item  (Choice of $\sigma$) Then  $\sigma=\sigma(\eta)>0$  can be chosen as
  \begin{equation}\label{joinings202104.7}
   \sigma<\frac{3\eta}{4+6\eta}.
  \end{equation}
  \item (Choice of $\epsilon_{0}$, $K_{1}$; injectivity radius)  Since $\Gamma$ is discrete, there is a compact subset $K_{1}\subset \overline{X}$, $\overline{\mu}(K_{1})>1- \frac{1}{4}\sigma$ and $ \epsilon_{0}>0$ such that for any $\overline{g_{y}}\in \overline{P}^{-1}(K_{1})$    satisfying
 \begin{equation}\label{joinings202104.22}
  d_{C^{\rho}\backslash G}( \overline{g_{y}},  \overline{g_{y}}\gamma)<O(\epsilon_{0})
 \end{equation}
 for some $\gamma\in\Gamma$, then $\gamma=e$. Here the constants hidden in $O(\epsilon_{0})$ will be determined after the estimate (\ref{joinings202104.20}) (see also (\ref{joinings202104.21})).
 \item (Choice of $K_{2}$, $K$, $T_{0}$,  $r_{0}$; ergodicity of $a^{T}$)
Since the diagonal action $a^{T}$ is ergodic on $(\overline{X},\overline{\mu})$,  there is a compact subset $K_{2}\subset  \overline{X}$, $\overline{\mu}(K_{2})>1-\frac{1}{4}\sigma$ and $T_{0}=T_{0}(K_{2})>0$ such that  the relative length measure $K_{2}$ on $[y,a^{T}y]$ (and $[a^{-T}y,y]$) is greater than $1-\sigma$ for any $y\in K_{2}$,   $|T|\geq T_{0}$.   Assume that
\begin{equation}\label{dynamical systems1026}
  K\coloneqq K_{1}\cap K_{2},\ \ \    r_{0}> e^{(1+2\delta)^{-1}T_{0}} .
\end{equation}
Note that $\overline{\mu}(K)>1-\sigma$. The quantity $r_{0}$ will be even larger and determined by $\epsilon_{0}$ if necessary (see (\ref{joinings202104.21})).
\end{itemize}

Now we are in the position to apply the renomalization via the diagonal action $a^{w}$. Since $r>r_{0}=e^{(1+2\delta)^{-1}T_{0}}$, let $e^{\omega_{0}}\coloneqq   r^{1+2\delta}$ and we know $\omega_{0}>T_{0}$. Since $\overline{y}^{\prime}\in K\subset K_{2}$, it follows from the choice of $K_{2}$ and $T_{0}$  that the relative length measure of $K$ on $[\overline{y}^{\prime},a^{\omega_{0}}\overline{y}^{\prime}]$ is greater than $1-\sigma$. This implies that there is $\omega$ satisfying
  \[(1-\sigma)\omega_{0}<\omega\leq\omega_{0}\]
  such that $a^{\omega}\overline{y}^{\prime}\in K$ and therefore
  \begin{equation}\label{dynamical systems1031}
    a^{\omega}\overline{g_{\overline{y}^{\prime}}}\in \overline{P}^{-1}(K).
  \end{equation}

 By (\ref{joinings202104.15}), we have
  \begin{align}
  a^{\omega} g_{\overline{y}^{\prime}} =&  (a^{\omega} \overline{h}^{\prime} a^{-\omega})\exp (\Ad a^{\omega}. \overline{v}^{\prime}) (a^{\omega}u^{-s}a^{-\omega}) a^{\omega} g_{x^{\prime\prime}}\;\nonumber\\
  a^{\omega} g_{\overline{y}^{\prime}}\gamma=&    c^{-1}(a^{\omega}u^{-t} a^{-\omega}) (a^{\omega} h^{\prime\prime} a^{-\omega})\exp (\Ad a^{\omega}.v^{\prime\prime}) a^{\omega}g_{x^{\prime\prime}}  \label{joinings202104.19}
\end{align}
Then by (\ref{joinings202104.16}) (\ref{joinings202104.12}) (\ref{joinings202104.17}), we estimate
\begin{align}
   a^{w} \overline{h}^{\prime}a^{-w}= &  \left[
            \begin{array}{ccc}
               1+O(\epsilon) &   O(r^{2\delta-2\eta})  \\
              O(\epsilon) &   1+O(\epsilon)\\
            \end{array}
          \right]\;\nonumber\\
a^{w} \overline{h}^{\prime}a^{-w}= &  \left[
            \begin{array}{ccc}
               1+O(r^{-2\eta}) &   O(r^{2\delta-2\eta})  \\
              O(\epsilon) &   1+O(r^{-2\eta})\\
            \end{array}
          \right]\;\nonumber\\
\Ad a^{\omega}. \overline{v}^{\prime}= & O(r^{ -2\xi+1+2\delta  })v_{0}+O(\epsilon)v_{1}+O(\epsilon)v_{2} \;\nonumber\\
\Ad a^{\omega}.  v^{\prime\prime}= & O(r^{ -2\xi+1+2\delta  })v_{0}+O(\epsilon)v_{1}+O(r^{ - (1-\sigma)(1+2\delta) })v_{2} \;\label{joinings202104.20}\\
a^{\omega}u^{-t}a^{-\omega}= & u^{-te^{-\omega}}= u^{O(r^{1+\eta}r^{-(1-\sigma)(1+2\delta)})}\;\nonumber\\
a^{\omega}u^{-s}a^{-\omega}=&  u^{-se^{-\omega}}= u^{O(r^{1+\eta}r^{-(1-\sigma)(1+2\delta)})} . \nonumber
\end{align}
Notice that by the choice of $\sigma,\delta$ (see (\ref{joinings202104.6}) (\ref{joinings202104.7})), we have
\[1+\eta-(1-\sigma)(1+2\delta)=1+\eta-(1-\sigma)(1+\frac{3}{2}\eta)<-\frac{1}{4}\eta.\]
Also, by (\ref{joinings202104.6}), we have
\[2\delta-2\eta<0,\ \ \ -2\xi+1+2\delta<0.\]
Thus, by enlarging $r_{0}$ if necessary, all terms of (\ref{joinings202104.20}) can be quantitatively dominated by $O(\epsilon_{0})$.
Then by (\ref{joinings202104.19}), we have
\begin{equation}\label{joinings202104.21}
  d_{C^{\rho}\backslash G}(\overline{a^{\omega} g_{\overline{y}^{\prime}}}\gamma,\overline{a^{\omega} g_{\overline{y}^{\prime}}})=d_{C^{\rho}\backslash G}(\overline{a^{\omega} g_{\overline{y}^{\prime}}}\gamma (a^{\omega} g_{x^{\prime\prime}})^{-1},\overline{a^{\omega} g_{\overline{y}^{\prime}}}(a^{\omega} g_{x^{\prime\prime}})^{-1})<O(\epsilon_{0}).
\end{equation}
Thus, by (\ref{joinings202104.22}), we get $\gamma=e$, which contradicts our assumptions.
\end{proof}

\subsection{Construction of $\epsilon$-blocks}
In light of Proposition \ref{joinings202104.29}, we try to construct a collection of $\epsilon$-blocks based on the unipotent flows between two nearby points so that each pair of $\epsilon$-blocks has an effective gap.

First, given $\eta_{0}\approx0$ as in Proposition \ref{joinings202104.29}, we  fix a sufficiently small $\kappa\in(0,2\eta_{0})$, and then choose  $\eta=\eta(\kappa)\approx0$ such that
\begin{equation}\label{joinings202104.31}
  \frac{1+2\eta}{\xi(2\eta)}<1+\kappa<1+2\eta_{0}
\end{equation}
 where $\xi(2\eta)\approx1$ is given by Corollary \ref{joinings202104.10}.
Then $\sigma_{0}=\sigma_{0}(\eta)\approx0$ given in Proposition \ref{joinings202104.29} has been determined. Next, assume that there exist
\begin{itemize}
  \item $\sigma\in(0,\sigma_{0})$,
  \item $R_{0}>1$,
  \item $\epsilon_{0}=\epsilon_{0}(\sigma)\approx0$, $\epsilon=\epsilon(R_{0})\in (0,\epsilon_{0})$ so small that
  \begin{equation}\label{joinings202104.34}
  \overline{L}_{1}(g)\geq L(\epsilon,R_{0},\kappa)>\max\{r_{0}(\sigma,\epsilon_{0}),R_{0}\}
  \end{equation}
  whenever $g\in B_{G}(e,\epsilon)$, where $\overline{L}_{1},L$ are defined by Corollary \ref{joinings202104.10},
\end{itemize}
such that
 for $K\subset \overline{X}$ with $\overline{\mu}(K)>1-\sigma$ given by Proposition \ref{joinings202104.29}, $x,y\in \overline{X}$,  we have $A=A(x,y)\subset\mathbf{R}^{+}$ such that
\begin{enumerate}[\ \ \ (i)]
  \item if $r\in A$, then
  \begin{equation}\label{joinings202104.30}
     u^{t(r)}y\in  K\ \ \text{ and }\ \  d_{\overline{X}}(u^{s(r)}x, u^{t(r)}y)<\epsilon
  \end{equation}
  for  continuous increasing functions  $t,s:[0,\infty)\rightarrow[0,\infty)$;
  \item   we have the H\"{o}lder inequalities:
  \begin{align}
 |(t(r^{\prime})-t(r))-(r^{\prime}-r)|\ll&  |r^{\prime}-r|^{1-\kappa}\;\label{joinings202104.5}\\
  |(s(r^{\prime})-s(r))-(r^{\prime}-r)|\ll&  |r^{\prime}-r|^{1-\kappa} \; \nonumber
\end{align}
      for all $r,r^{\prime}\in A$ with $r^{\prime}>r$, $ r^{\prime}-r \geq R_{0}$.
\end{enumerate}
  It is worth noting from (\ref{joinings202104.22}) that points in $K$ have injectivity radius at least $\epsilon_{0}$.  For simplicity, we shall assume that $0\in A$ in what follows.
\begin{rem}\label{joinings202104.32}
  For the condition (i) (ii), the quantities $s,t$ are symmetric. Thus, for instance, one can also consider $s$ as an increasing function of $t$, and obtain similar H\"{o}lder inequalities. We have already made  such a change of variables in Section \ref{joinings202104.33}, for notational simplicity.

  On the other hand,   the assumptions (\ref{joinings202104.30}) (\ref{joinings202104.5}) coincide with (\ref{joinings202104.9})  (\ref{joinings202104.8}). So Corollary \ref{joinings202104.10} can apply.
\end{rem}

\noindent
\textbf{Construction of  $\beta_{1}$.}
  For $\lambda\in A$ denote
$A_{\lambda}\coloneqq A\cap[0,\lambda]$.
Now we construct  a collection $\beta_{1}(A_{\lambda})$ of  $\epsilon$-blocks. Let $x_{1}\coloneqq x$, $y_{1}\coloneqq y$. We follow the assumptions (\ref{joinings202104.30}) (\ref{joinings202104.5}). Suppose that $(\overline{g_{x_{1}}}, \overline{g_{y_{1}}})\in C^{\rho}\backslash G\times C^{\rho}\backslash G$ covers $(x_{1},y_{1})$ and
\[ \overline{r}_{1}\coloneqq  \sup\{r\in A_{\lambda}\cap[0,\overline{L}_{1}(g_{y_{1}}g_{x_{1}}^{-1})]: d_{G}( u^{t(r)}g_{y_{1}}  ,u^{s(r)}g_{x_{1}}  )<\epsilon \},\ \ \ \overline{s}_{1}\coloneqq s(\overline{r}_{1})\]
where $\overline{L}_{1}$ is defined by Corollary \ref{joinings202104.10}.
Let $\BL_{1}$ be the $\epsilon$-block of $x_{1},y_{1}$ of length $\overline{r}_{1}$, $\BL_{1}=\{(x_{1},y_{1}),(\overline{x}_{1},\overline{y}_{1})\}$.
To define $\BL_{2}$,  we take
\[r_{2}\coloneqq\inf\{r\in A_{\lambda}:r>\overline{r}_{1} \},\ \ \ s_{2}\coloneqq s(r_{2}) \]
and apply the above procedure to
\[x_{2}\coloneqq u^{s(r_{2})}x_{1},\ \ \ y_{2}\coloneqq  u^{t(r_{2})}y_{1}\]
(Note that by (\ref{joinings202104.8}), $r_{2}>\overline{r}_{1}$). This process defines   a collection $\beta_{1}(A_{\lambda})=\{\BL_{1},\ldots, \BL_{n}\}$ of $\epsilon$-blocks on the orbit intervals $[x_{1},u^{s(\lambda)}x_{1}]$, $[y_{1},u^{t(\lambda)}y_{1}]$ (see Figure \ref{joinings202106.176}):
 \[x_{i}= u^{s_{i}}x_{1},\ \ \ \overline{x}_{i}=u^{\overline{s}_{i}}x_{1}, \ \ \ y_{i}=u^{t_{i}} y_{1},\ \ \ \overline{y}_{i}=u^{\overline{t}_{i}}y_{1}\]
  \[s_{i}= s(r_{i}),\ \ \ \overline{s}_{i}= s(\overline{r}_{i}), \ \ \ t_{i}= t(r_{i}),\ \ \ \overline{t}_{i}= t(\overline{r}_{i}).\]
 Note also that by the assumption of $A$, we have $x_{i},\overline{x}_{i}\in K$ for all $i$, the corresponding time interval of $BL_{i}$ is $[r_{i},\overline{r}_{i}]$ and  the length $|\BL_{i}|$ of $\BL_{i}$ is
 \[|\BL_{i}|\coloneqq \overline{r}_{i}-r_{i}.\]
  \begin{figure}
\centering
  \tikzstyle{init} = [pin edge={to-,thin,black}]
\begin{tikzpicture}[node distance=4cm,auto,>=latex']
\draw (0,0) coordinate (a0) --node[below] {$\BL_{1}$} (2,0) coordinate (a2) ;
\filldraw[black] (a0) circle(1pt) node[below] {$x_{1}$};
\filldraw[black] (a2) circle(1pt) node[below] {$\overline{x}_{1}$};

 \draw [dashed] (2,0) coordinate (a2) -- (3,0) coordinate (a3);
 \draw (3,0) coordinate (a3) --node[below] {$\BL_{2}$} (5,0) coordinate (a5) ;
\filldraw[black] (a3) circle(1pt) node[below] {$x_{2}$};
\filldraw[black] (a5) circle(1pt) node[below] {$\overline{x}_{2}$};

 \draw [dashed] (5,0) coordinate (a5) -- (6,0) coordinate (a6);
 \draw (6,0) coordinate (a6) --node[below] {$\BL_{3}$} (8,0) coordinate (a8) ;
\filldraw[black] (a6) circle(1pt) node[below] {$x_{3}$};
\filldraw[black] (a8) circle(1pt) node[below] {$\overline{x}_{3}$};

 \draw [dashed] (8,0) coordinate (a8) -- (12,0) coordinate (a12);
 \filldraw[black] (a12) circle(1pt) node[below] {$\overline{x}_{n}$};

 \path[->] (a0) edge [bend left] node [right]{$s_{2}$} (a3);
 \path[->] (a0) edge [bend left] node [below]{$\overline{s}_{2}$} (a5);

 \path[->] (a0) edge [bend left] node [right]{$s_{3}$} (a6);
 \path[->] (a0) edge [bend left] node [right]{$\overline{s}_{3}$} (a8);
  \path[->] (a0) edge [bend left] node [right]{$s(\lambda)$} (a12);

  \draw (0,-1) coordinate (b0) --  (2,-1) coordinate (b2) ;
\filldraw[black] (b0) circle(1pt) node[above] {$y_{1}$};
\filldraw[black] (b2) circle(1pt) node[above] {$\overline{y}_{1}$};

 \draw [dashed] (2,-1) coordinate (b2) -- (3,-1) coordinate (b3);
 \draw (3,-1) coordinate (b3) --  (5,-1) coordinate (b5) ;
\filldraw[black] (b3) circle(1pt) node[above] {$y_{2}$};
\filldraw[black] (b5) circle(1pt) node[above] {$\overline{y}_{2}$};

 \draw [dashed] (5,-1) coordinate (b5) -- (6,-1) coordinate (b6);
 \draw (6,-1) coordinate (b6) --  (8,-1) coordinate (b8) ;
\filldraw[black] (b6) circle(1pt) node[above] {$y_{3}$};
\filldraw[black] (b8) circle(1pt) node[above] {$\overline{y}_{3}$};

 \draw [dashed] (8,-1) coordinate (b8) -- (12,-1) coordinate (b12);
 \filldraw[black] (b12) circle(1pt) node[above] {$\overline{y}_{n}$};

 \path[->] (b0) edge [bend right] node [right]{$t_{2}$} (b3);
 \path[->] (b0) edge [bend right] node [right]{$\overline{t}_{2}$} (b5);

 \path[->] (b0) edge [bend right] node [right]{$t_{3}$} (b6);
 \path[->] (b0) edge [bend right] node [right]{$\overline{t}_{3}$} (b8);
  \path[->] (b0) edge [bend right] node [right]{$t(\lambda)$} (b12);
\end{tikzpicture}
\caption{A collection of $\epsilon$-blocks $\{\BL_{1},\ldots,\BL_{n}\}$.  The solid straight lines are the unipotent orbits in the $\epsilon$-blocks and the dashed lines are the rest of the unipotent orbits. The bent curves indicate the   length defined by the letters.}
\label{joinings202106.176}
\end{figure}
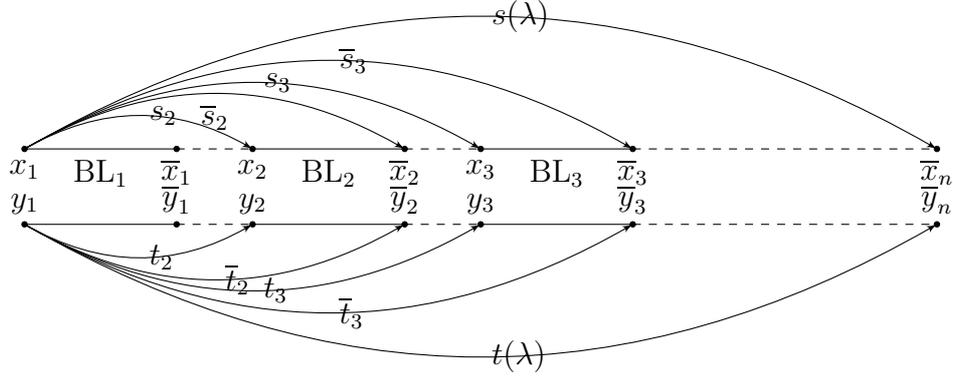

  Note that any  $\BL_{i}=\{(x_{i},y_{i}),(\overline{x}_{i},\overline{y}_{i})\}\in\beta_{1}(A_{\lambda})$ has length $|\BL_{i}|\leq \overline{L}_{1}(g_{y_{i}}g_{x_{i}}^{-1})$. By  Corollary \ref{joinings202104.10}, we immediately obtain an estimate for the difference of $g_{x_{i}}$ and $g_{y_{i}}$ in terms of the length of $\epsilon$-blocks.
  \begin{cor}[Difference of $\beta_{1}(A_{\lambda})$]\label{joinings202104.2}
   Assume that  $\overline{g_{y_{i}}g_{x_{i}}^{-1}}=C^{\rho}h_{i}\exp(v_{i})$, where
    \[h_{i}=\left[
            \begin{array}{ccc}
              a&  b \\
             c&  d\\
            \end{array}
          \right]\in SO_{0}(2,1),\ \ \ v_{i}=b_{0}v_{0}+\cdots+b_{\varsigma}v_{\varsigma}\in V_{\varsigma}.\]
  Then we have
\[h_{i}=\left[
            \begin{array}{ccc}
               1+O(\mathbf{r}_{i}^{-\kappa}) &   O(\mathbf{r}_{i}^{-1-\kappa})  \\
              O(\epsilon) &   1+O(\mathbf{r}_{i}^{-\kappa})\\
            \end{array}
          \right],\ \ \ v_{i}=O(\mathbf{r}_{i}^{-\varsigma})v_{0}+\cdots+O(\epsilon)v_{\varsigma}\]
  for some  $\mathbf{r}_{i}\geq\max\{r_{0},R_{0},|\BL_{i}|\}$.
  \end{cor}

  We then immediately conclude from Proposition \ref{joinings202104.29}  that for any  $\BL^{\prime},\BL^{\prime\prime}\in\beta_{1}(A_{\lambda})$ with  $\BL^{\prime}\overset{\Gamma}{\sim}\BL^{\prime\prime}$, there is an effective gap between them, i.e.
\[d(\BL^{\prime},\BL^{\prime\prime})\geq[\min\{|\BL^{\prime}|,|\BL^{\prime\prime}|\}]^{1+\kappa/2}.\]
However, when   $\BL^{\prime}\overset{e}{\sim}\BL^{\prime\prime}$, they do not necessarily have an effective gap. This enlighten us to connect these $\epsilon$-blocks and generate a new collection $\beta_{2}(A_{\lambda})$.

\noindent
\textbf{Construction of  $\beta_{2}$.}
  Now we construct a new collection $\beta_{2}(A_{\lambda})=\{\overline{\BL}_{1},\ldots,\overline{\BL}_{N}\}$ by the following procedure. The idea is to connect $\epsilon$-blocks in $\beta_{1}(A_{\lambda})=\{\BL_{1},\ldots,\BL_{n}\}$ so that each pair of new blocks must have an effective gap.
  Let $\BL_{1}\in\beta_{1}(A_{\lambda})$,   $g_{y_{1}}=h\exp(v)g_{x_{1}}$ and
       \[h=\left[
            \begin{array}{ccc}
              a&  b \\
             c&  d\\
            \end{array}
          \right]\in SO(2,1),\ \ \ v=b_{0}v_{0}+\cdots+b_{\varsigma}v_{\varsigma}\in V_{\varsigma}.\]
  Then by  Corollary \ref{joinings202104.10}, one can write $u^{t(r)}gu^{-s(r)}\in B_{G}(e,\epsilon)$ for
  \begin{equation}\label{time change177}
   r\in   \bigcup_{k}[L_{k}(g),\overline{L}_{k}(g)]
  \end{equation}
  where $k\leq c$ is uniformly bounded for all $g\in G$.
  Then consider the following two cases:
    \begin{enumerate}[\ \ \ (i)]
    \item  There is no $j\in\{2,\ldots,n\}$ such that $(x_{1},y_{1})\overset{e}{\sim}(x_{j},y_{j})$.
     \item There is $j\in\{2,\ldots,n\}$ such that   $(x_{1},y_{1})\overset{e}{\sim}(x_{j},y_{j})$.
    \end{enumerate}

  In case (i), we set $\overline{\BL}_{1}=\BL_{1}$. Then by Corollary \ref{joinings202104.2}, we have
  \begin{equation}\label{time change180}
    |b|\ll\overline{L}_{1}(g_{y_{1}}g_{x_{1}}^{-1})^{-1-\kappa},\ \ \ |a-d|\leq  \overline{L}_{1}(g_{y_{1}}g_{x_{1}}^{-1})^{-\kappa}
  \end{equation}

   In case (ii),  suppose that   $g_{x_{j}}=u^{s_{j}}g_{x_{1}}$, $g_{y_{j}}=u^{t_{j}}g_{y_{1}}$. Clearly, by the construction, $\overline{r}_{j}>\overline{L}_{1}(g_{y_{1}}g_{x_{1}}^{-1})$. On the other hand, by (\ref{time change177}), we get
    \[\overline{r}_{j}\in   \bigcup_{k}[L_{k}(g_{y_{1}}g_{x_{1}}^{-1}),\overline{L}_{k}(g_{y_{1}}g_{x_{1}}^{-1})]\]
    and $k\leq C$ is uniformly bounded for all $g\in G$.
    Assume that $j_{\max}$ is the maximal $j$ among $\overline{r}_{j}\in[L_{2}(g_{y_{1}}g_{x_{1}}^{-1}),\overline{L}_{2}(g_{y_{1}}g_{x_{1}}^{-1})]$.  Whether $[0,\overline{L}_{1}(g_{y_{1}}g_{x_{1}}^{-1})]$ and $[L_{2}(g_{y_{1}}g_{x_{1}}^{-1}),\overline{L}_{2}(g_{y_{1}}g_{x_{1}}^{-1})]$ have an effective gap leads to a dichotomy of choices:
   \[\overline{\BL}_{1}=\left\{\begin{array}{ll}
 \text{remains unchange}&,\text{ if } L_{2}(g_{y_{1}}g_{x_{1}}^{-1})-\overline{L}_{1}(g_{y_{1}}g_{x_{1}}^{-1})>\overline{L}_{1}(g_{y_{1}}g_{x_{1}}^{-1})^{1+2\eta}\\
 \{(x_{1},y_{1}),(\overline{x}_{j_{\max}},\overline{y}_{j_{\max}})\} &,\text{ otherwise}
\end{array}\right. .\]
If the first case occurs, we will not change $\overline{\BL}_{1}$ anymore.
If the second case occurs, i.e. we  redefine $\overline{\BL}_{1}=\{(x_{1},y_{1}),(\overline{x}_{j_{\max}},\overline{y}_{j_{\max}})\}$, then we repeat the construction for the new $\overline{\BL}_{1}$ again:
\begin{enumerate}[\ \ \ ]
  \item  Suppose that there is  $\overline{r}_{j}>\overline{L}_{2}(g_{y_{1}}g_{x_{1}}^{-1})$. Then assume $j_{\max}$ to be the maximal $j$ among $\overline{r}_{j}\in[L_{3}(g_{y_{1}}g_{x_{1}}^{-1}),\overline{L}_{3}(g_{y_{1}}g_{x_{1}}^{-1})]$. Then again, we set
  \[\overline{\BL}_{1}=\left\{\begin{array}{ll}
\text{remains unchange}&,\text{ if } L_{3}(g_{y}g_{x}^{-1})-\overline{L}_{3}(g_{y_{1}}g_{x_{1}}^{-1})>\overline{L}_{2}(g_{y_{1}}g_{x_{1}}^{-1})^{1+2\eta}\\
 \{(x_{1},y_{1}),(\overline{x}_{j_{\max}},\overline{y}_{j_{\max}})\} &,\text{ otherwise}
\end{array}\right. \]
and so on.
\end{enumerate}
The process will stop since the number of intervals is uniformly bounded for all $g\in G$. Now $\overline{BL}_{1}\in\beta_{2}(A_{\lambda})$ has been constructed. By the choice
of $\overline{\BL}_{1}$ and  Corollary \ref{joinings202104.10}, we conclude  that
\begin{equation}\label{time change178}
 |b|\ll_{\kappa} |\BL_{1}|^{-\xi(1+\kappa)},\ \ \ |a-d|\ll_{\kappa} |\BL_{1}|^{-\xi\kappa},\ \ \ |b_{i}|\ll_{\varsigma,\kappa} |\BL_{1}|^{-\xi (\varsigma-i)}
\end{equation}
for  $\xi=\xi(2\eta)\approx1$ and for all $1\leq i\leq \varsigma$.

Next, we repeat the above argument to construct $\overline{\BL}_{m+1}$. More precisely, suppose that $   \overline{\BL}_{m}=\{(x_{j_{m-1}+1},y_{j_{m-1}+1}),(\overline{x}_{j_{m}},\overline{y}_{j_{m}})\}\in\beta_{2}(A_{\lambda})$ has been constructed. To define $\overline{\BL}_{m+1}$, we repeat the above argument  to  $\BL_{j_{m}+1}\in\beta_{1}(A_{\lambda})$. Thus, $\beta_{2}(A_{\lambda})$ is completely defined. Further, one may conclude the difference of points of $\epsilon$-blocks in $\beta_{2}(A_{\lambda})$:
\begin{lem}[Difference of $\beta_{2}(A_{\lambda})$]\label{joinings202104.35}
   For any $\overline{\BL}_{i}=\{(x_{i}^{\prime},y_{i}^{\prime}),(\overline{x}_{i}^{\prime},\overline{y}_{i}^{\prime})\}$ in the collection  $\beta_{2}(A_{\lambda})=\{\overline{\BL}_{1},\ldots,\overline{\BL}_{N}\}$ of $\epsilon$-blocks, we have
   \[ \overline{g_{y_{i}^{\prime}}g_{x_{i}^{\prime}}^{-1}}=C^{\rho}h_{i}\exp(v_{i})\]
    where
   \begin{equation}\label{joinings202104.36}
h_{i}=\left[
            \begin{array}{ccc}
               1+O(\mathbf{r}_{i}^{-2\eta}) &   O(\mathbf{r}_{i}^{-1-2\eta})  \\
              O(\epsilon) &   1+O(\mathbf{r}_{i}^{-2\eta})\\
            \end{array}
          \right],\ \ \ v_{i}=O(\mathbf{r}_{i}^{-\xi\varsigma})v_{0}+\cdots+O(\epsilon)v_{\varsigma}
   \end{equation}
  for some  $\mathbf{r}_{i}\geq\max\{r_{0},R_{0},|\overline{\BL}_{i}|\}$.
\end{lem}
\begin{proof}
 (\ref{joinings202104.36}) follows immediately from (\ref{time change180}), (\ref{time change178}), (\ref{joinings202104.31}), (\ref{joinings202104.34}).
\end{proof}
Then, recall that by the construction of $\beta_{2}(A_{\lambda})$, for any  $\overline{\BL}^{\prime},\overline{\BL}^{\prime\prime}\in\beta_{2}(A_{\lambda})$ with  $\overline{\BL}^{\prime}\overset{e}{\sim}\overline{\BL}^{\prime\prime}$, there is an effective gap between them, i.e.
\[d(\overline{\BL}^{\prime},\overline{\BL}^{\prime\prime})\geq\left[\max\{r_{0},R_{0},\min\{|\overline{\BL}^{\prime}|,|\overline{\BL}^{\prime\prime}|\}\}\right]^{1+2\eta}.\]
On the other hand, when   $\overline{\BL}^{\prime}\overset{\Gamma}{\sim}\overline{\BL}^{\prime\prime}$, by Proposition \ref{joinings202104.29} and Lemma \ref{joinings202104.35}, we have
\[d(\overline{\BL}^{\prime},\overline{\BL}^{\prime\prime})\geq\left[\max\{r_{0},R_{0},\min\{|\overline{\BL}^{\prime}|,|\overline{\BL}^{\prime\prime}|\}\}\right]^{1+\eta}.\]
Thus, we conclude from Proposition \ref{joinings202104.27} that
\begin{prop}[Effective gaps of $\beta_{2}(A_{\lambda})$]\label{joinings202104.37} Let the notation and assumptions be as above. For any $\overline{\BL}^{\prime},\overline{\BL}^{\prime\prime}\in\beta_{2}(A_{\lambda})$, we have
\[d(\overline{\BL}^{\prime},\overline{\BL}^{\prime\prime})\geq\left[\max\{r_{0},R_{0},\min\{|\overline{\BL}^{\prime}|,|\overline{\BL}^{\prime\prime}|\}\}\right]^{1+\eta}.\]
Thus, for any $\zeta\in[0,1]$, if
\[\frac{1}{\lambda}\Leb( A_{\lambda})\geq \overline{\theta}_{\eta}(\zeta)=1-\theta(\eta,\zeta)= 1-\prod_{n=0}^{\infty} \left(1+ C\zeta^{n\eta}\right)^{-1}\]
then there is an $\epsilon$-block $\overline{\BL}\in\beta_{2}(A_{\lambda})$ that has
\[|\overline{\BL}|\geq\zeta\lambda.\]
\end{prop}

\subsection{Non-shifting time}

Now assume that for some $\lambda,\zeta>0$, we know that
\[\Leb(A_{\lambda})\geq \overline{\theta}_{\eta}(\zeta)\lambda.\]
 Then Proposition \ref{joinings202104.37} provides us an $\epsilon$-block $\overline{\BL}=\{(x^{\prime},y^{\prime}),(\overline{x}^{\prime},\overline{y}^{\prime})\}\in\beta_{2}(A_{\lambda})$ with $|\overline{\BL}|\geq\zeta\lambda$. In other words, if we write
 \begin{equation}\label{joinings202104.41}
   x^{\prime}=u^{s(R_{1})}x,\ \ \ \overline{x}^{\prime}=u^{s(R_{2})}x,\ \ \ y^{\prime}=u^{t(R_{1})}y,\ \ \ \overline{y}^{\prime}=u^{t(R_{2})}y,
 \end{equation}
then we can find $R_{1},R_{2}>0$ with $R_{2}-R_{1}\geq\zeta\lambda$ such that
\[d_{C^{\rho}\backslash G}(u^{t(R_{1})}.\overline{g_{y}},u^{s(R_{1})}.\overline{g_{x}})<\epsilon,\ \ \ d_{C^{\rho}\backslash G}(u^{t(R_{2})}.\overline{g_{y}},u^{s(R_{2})}.\overline{g_{x}})<\epsilon.\]
It is already quite surprising. However, it is still possible that
\[d_{C^{\rho}\backslash G}(u^{t(r)}.\overline{g_{y}},u^{s(r)}.\overline{g_{x}})>\epsilon\]
for some $r\in[R_{1},R_{2}]\cap A$. Thus, define
\[\overline{A}_{R_{1}R_{2}}\coloneqq\{r\in[R_{1},R_{2}]\cap A:d_{C^{\rho}\backslash G}(u^{t(r)}.\overline{g_{y}},u^{s(r)}.\overline{g_{x}})>\epsilon\}\]
and we want to show that $\Leb(\overline{A}_{R_{1}R_{2}})/\lambda$ has a upper bound in certain situations.
\begin{rem}\label{joinings202104.42}
   By (\ref{joinings202104.36}), we can estimate the difference between $x^{\prime},y^{\prime}$; more precisely, we have
   \[ \overline{g_{y^{\prime}}g_{x^{\prime}}^{-1}}=C^{\rho} h\exp(v)\]
where
\[h=\left[
            \begin{array}{ccc}
               1+O((\zeta\lambda)^{-2\eta}) &   O((\zeta\lambda)^{-1-2\eta})  \\
              O(\epsilon) &   1+O((\zeta\lambda)^{-2\eta})\\
            \end{array}
          \right],\ \ \ v=O((\zeta\lambda)^{-\xi\varsigma})v_{0}+\cdots+O(\epsilon)v_{\varsigma}.\]
\end{rem}
\noindent
\textbf{Construction of  $\widetilde{\beta}_{1},\widetilde{\beta}_{2}$.} Now we consider the shifting time of the $\epsilon$-block $\overline{\BL}=\{(x^{\prime},y^{\prime}),(\overline{x}^{\prime},\overline{y}^{\prime})\}\in\beta_{2}(A_{\lambda})$.
Define a collection $\widetilde{\beta}_{1}(\overline{A}_{R_{1}R_{2}})$ of $\epsilon$-blocks on the orbit intervals $[x^{\prime},x^{\prime\prime}]$, $[y^{\prime},y^{\prime\prime}]$ according to the following steps.
Suppose that
\[r_{1}\coloneqq\min\{r\in[R_{1},R_{2}]:r\in\overline{A}_{R_{1}R_{2}}\},\ \ \ x_{1}\coloneqq u^{s(R_{1})}x^{\prime},\ \ \ y_{1}\coloneqq u^{t(R_{1})}y^{\prime}\]
and that $(\overline{g_{x_{1}}}, \overline{g_{y_{1}}})\in C^{\rho}\backslash G\times C^{\rho}\backslash G$ covers $(x_{1},y_{1})$ and
\[ \overline{r}_{1}\coloneqq  \sup\{R\in \overline{A}_{R_{1}R_{2}}: d_{G}( u^{t(r)}g_{y_{1}}  ,u^{s(r)}g_{x_{1}}  )<\epsilon \text{ for any }r\in\overline{A}_{R_{1}R_{2}}\cap[0,R]\}.\]
Let $\BL_{1}\in\widetilde{\beta}_{1}(\overline{A}_{R_{1}R_{2}})$  be the $\epsilon$-block of $x_{1},y_{1}$ of length $\overline{r}_{1}$, and write $\BL_{1}=\{(x_{1},y_{1}),(\overline{x}_{1},\overline{y}_{1})\}$.
To define $\BL_{2}$,  we take
\[r_{2}\coloneqq\inf\{r\in \overline{A}_{R_{1}R_{2}}:r>\overline{r}_{1} \} \]
and apply the above procedure to
\[x_{2}\coloneqq u^{s(r_{2})}x_{1},\ \ \ y_{2}\coloneqq  u^{t(r_{2})}y_{1}.\]
 This process defines  a collection $\widetilde{\beta}_{1}(\overline{A}_{R_{1}R_{2}})=\{\BL_{1},\ldots, \BL_{m}\}$ of $\epsilon$-blocks on the orbit intervals $[u^{s(r_{1})}x^{\prime},u^{s(\overline{r}_{m})}x^{\prime}]$, $[u^{t(r_{1})}y^{\prime},u^{t(\overline{r}_{m})}y^{\prime}]$.
 Completely  similar to $\beta_{1}$, we can connect some of the $\epsilon$-blocks in $\widetilde{\beta}_{1}(\overline{A}_{R_{1}R_{2}})$ and form a new collection $\widetilde{\beta}_{2}(\overline{A}_{R_{1}R_{2}})$ such that each pair of  $\epsilon$-blocks in $\widetilde{\beta}_{2}(\overline{A}_{R_{1}R_{2}})$ has an effective gap. Then, we conclude again from Proposition \ref{joinings202104.27} that
\begin{lem}[Difference and effective gaps of $\widetilde{\beta}_{2}(\overline{A}_{R_{1}R_{2}})$]\label{joinings202104.39}
For any $\widetilde{\BL}_{i}=\{(\widetilde{x}_{i}^{\prime},\widetilde{y}_{i}^{\prime}),(\overline{\widetilde{x}}_{i}^{\prime},\overline{\widetilde{y}}_{i}^{\prime})\}$ in the collection  $\widetilde{\beta}_{2}(\overline{A}_{R_{1}R_{2}})=\{\widetilde{\BL}_{1},\ldots,\widetilde{\BL}_{M}\}$ of $\epsilon$-blocks, we have
   \[ \overline{g_{\widetilde{y}_{i}^{\prime}}g_{\widetilde{x}_{i}^{\prime}}^{-1}}=C^{\rho}h_{i}\exp(v_{i})\]
    where
   \begin{equation}\label{joinings202104.40}
h_{i}=\left[
            \begin{array}{ccc}
               1+O(\mathbf{r}_{i}^{-2\eta}) &   O(\mathbf{r}_{i}^{-1-2\eta})  \\
              O(\epsilon) &   1+O(\mathbf{r}_{i}^{-2\eta})\\
            \end{array}
          \right],\ \ \ v_{i}=O(\mathbf{r}_{i}^{-\xi\varsigma})v_{0}+\cdots+O(\epsilon)v_{\varsigma}
   \end{equation}
  for some  $\mathbf{r}_{i}\geq\max\{r_{0},R_{0},|\widetilde{\BL}_{i}|\}$.

 Moreover, for any $\widetilde{\BL}^{\prime},\widetilde{\BL}^{\prime\prime}\in\widetilde{\beta}_{2}(\overline{A}_{R_{1}R_{2}})$, we have
\[d(\widetilde{\BL}^{\prime},\widetilde{\BL}^{\prime\prime})\geq\left[\max\{r_{0},R_{0},\min\{|\widetilde{\BL}^{\prime}|,|\widetilde{\BL}^{\prime\prime}|\}\}\right]^{1+\eta}.\]
Thus, for any $\widetilde{\zeta}\in[0,1]$, if
\[\frac{1}{\lambda}\Leb(\overline{A}_{R_{1}R_{2}})\geq \overline{\theta}_{\eta}(\widetilde{\zeta})= 1-\prod_{n=0}^{\infty} \left(1+ C\widetilde{\zeta}^{n\eta}\right)^{-1}\]
then there is an $\epsilon$-block $\widetilde{\BL}\in\widetilde{\beta}_{2}(\overline{A}_{R_{1}R_{2}})$ that has
\[|\widetilde{\BL}|\geq\widetilde{\zeta}\lambda.\]
\end{lem}

Thus, given $\widetilde{\zeta}\in(0,\zeta)$, we can apply Lemma \ref{joinings202104.39} and obtain an $\epsilon$-block $\widetilde{\BL}=\{(\widetilde{x},\widetilde{y}),(\overline{\widetilde{x}},\overline{\widetilde{y}})\}\in\widetilde{\beta}_{2}(\overline{A}_{R_{1}R_{2}})$
that has length $|\widetilde{\BL}|\geq\widetilde{\zeta}\lambda$. Then by (\ref{joinings202104.40}), we get that
\[ \overline{g_{\widetilde{y}}g_{\widetilde{x}}^{-1}}=C^{\rho}\widetilde{h}\exp(\widetilde{v})\]
where
\[\widetilde{h}=\left[
            \begin{array}{ccc}
               1+O((\widetilde{\zeta}\lambda)^{-2\eta}) &   O((\widetilde{\zeta}\lambda)^{-1-2\eta})  \\
              O(\epsilon) &   1+O((\widetilde{\zeta}\lambda)^{-2\eta})\\
            \end{array}
          \right],\ \ \ \widetilde{v}=O((\widetilde{\zeta}\lambda)^{-\xi\varsigma})v_{0}+\cdots+O(\epsilon)v_{\varsigma}.\]
Then combining   Remark \ref{joinings202104.42} and Proposition \ref{joinings202104.29}, we conclude that
\[r_{1}>(\widetilde{\zeta}\lambda)^{1+\eta}.\]
Since $r_{1}\in[R_{1},R_{2}]$, we obtain $(\widetilde{\zeta}\lambda)^{1+\eta}\leq \zeta\lambda$ or
\[\widetilde{\zeta}\leq (\zeta\lambda^{-\eta})^{\frac{1}{1+\eta}}.\]
In other words, we obtain
\begin{lem}[Shifting is sparse in a big $\epsilon$-block]\label{joinings202104.43}
    Given $\lambda>0,\zeta\in(0,1),\eta\approx0$, assume that
\[\Leb(A_{\lambda})\geq \overline{\theta}_{\eta}(\zeta)\lambda.\]
Then there is an $\epsilon$-block $\overline{\BL}\in\beta_{2}(A_{\lambda})$ with the corresponding time interval $[R_{1},R_{2}]$ and $|\overline{\BL}|=R_{2}-R_{1}\geq\zeta\lambda$. Besides, denote the shifting time of $\overline{\BL}$   by
  \[\overline{A}_{R_{1}R_{2}}\coloneqq\{r\in A\cap [R_{1},R_{2}]:d_{C^{\rho}\backslash G}(u^{t(r)}.\overline{g_{y}},u^{s(r)}.\overline{g_{x}})>\epsilon\}.\]
  Then we have
  \[\Leb(\overline{A}_{R_{1}R_{2}})/\lambda\leq \overline{\theta}_{\eta}\left((\zeta\lambda^{-\eta})^{\frac{1}{1+\eta}}\right)= 1-\prod_{n=0}^{\infty} \left(1+ C(\zeta\lambda^{-\eta})^{\frac{n\eta}{1+\eta}}\right)^{-1}.\]
  In particular, $\Leb(\overline{A}_{R_{1}R_{2}})/\lambda=o(\lambda)$.
\end{lem}

In the following, we present a key proposition below that will be used in the proof of Proposition \ref{joinings202104.58}. It basically says that non-shifting is always observable when the time scale is large.
\begin{prop}[Non-shifting time is not negligible]\label{joinings202104.57} Given an integer $n\geq 2$, $\kappa\in(0,2\eta_{0})$, there exist  $\lambda_{0}>0$, $\sigma_{0}\approx0$,   $\vartheta\approx0$ such that for any
\begin{itemize}
  \item  disjoint subsets $A^{1},\ldots,A^{n}\subset[0,\infty)$  that satisfy (\ref{joinings202104.30}) (\ref{joinings202104.5}),
  \item  $\lambda>\lambda_{0}$,
  \item $\sigma\in(0,\sigma_{0})$ satisfying
  \[\Leb\left(\coprod_{i=1}^{n}A^{i}\cap[0,\lambda]\right)>(1-2\sigma)\lambda,\]
\end{itemize}
   there exists  one $A^{i(\lambda)}$ and $[R_{1}^{\prime}(\lambda),R_{2}^{\prime}(\lambda)]\subset[0,\lambda]$ such that there exists an $\epsilon$-block $\overline{\BL}\in \beta_{2}(A^{i(\lambda)}\cap[R_{1}^{\prime},R_{2}^{\prime}])$ with the corresponding time
interval $[R_{1},R_{2}]$ such that
   \[R_{2}-R_{1}>\vartheta\lambda,\ \ \ \Leb\left(A^{i(\lambda)}_{\epsilon}\cap [R_{1},R_{2}]\right)>\vartheta\lambda\]
   where $A^{i(\lambda)}_{\epsilon}\coloneqq\{r\in A^{i(\lambda)}:d_{C^{\rho}\backslash G}(u^{t(r)}.\overline{g_{y}},u^{s(r)}.\overline{g_{x}})<\epsilon\}$ is the non-shifting time of $A^{i(\lambda)}$.
\end{prop}
\begin{proof}
   First, fix $\eta$ satisfying (\ref{joinings202104.31}), $\zeta_{1}\in(0,1)$ so that $\overline{\theta}_{\eta}(\zeta_{1})=1/(n+1)$ and  choose  $\zeta_{2}\approx 0$ such that
         \begin{equation}\label{joinings202104.53}
    \overline{\theta}_{\eta}(\zeta_{2})<\frac{\zeta_{1}^{-1}-1}{2(\zeta_{1}^{-n}-1)}
   \end{equation}  and then $\lambda_{0}>0$ such that
    \begin{equation}\label{joinings202104.44}
    \overline{\theta}_{\eta}(\zeta_{2})\zeta_{1}-\overline{\theta}_{\eta}\left((\zeta_{2}\lambda^{-\eta})^{\frac{1}{1+\eta}}\right)>\frac{1}{2} \overline{\theta}_{\eta}(\zeta_{2})\zeta_{1}
   \end{equation}
    for $\lambda>\lambda_{0}$.
  Then choose
  \begin{align}
    \sigma_{0}= & \min\left\{\frac{1}{4}\zeta_{1}^{n},\frac{1}{2(n+1)}\right\},  \;\label{joinings202104.54}\\
  \vartheta =&  \frac{1}{2} \overline{\theta}_{\eta}(\zeta_{2})\zeta_{1}^{n}. \; \label{joinings202104.56}
\end{align}
Given   $\sigma\in(0,\sigma_{0})$, $\lambda>\lambda_{0}$,
  we write $[R_{1}^{(0)},R_{2}^{(0)}]=[0,\lambda]$, $b_{0}=2\sigma$  and then  apply the following algorithm on $ k=0,1,\ldots, n-1$ orderly:

  First, assume that
  \begin{itemize}
    \item  $i_{1},\ldots,i_{k}\in\{1,\ldots,n\}$ have been chosen without repetition,
    \item  $b_{0},\ldots,b_{k}>0$ have been chosen,
  \end{itemize}
  and they satisfy
   \begin{equation}\label{joinings202104.47}
 \Leb\left(\coprod_{i\not\in\{i_{1},\ldots,i_{k}\}}A^{i}\cap[R_{1}^{(k)},R_{2}^{(k)}]\right)/\Leb([R_{1}^{(k)},R_{2}^{(k)}])>1-b_{k}.
 \end{equation}
 (Note that by the choice of $\zeta_{1}$ and $\sigma_{0}$, (\ref{joinings202104.47}) is possible for $k=0$.)
  Then there is one $A^{i_{k+1}}$ for some $i_{k+1}\not\in\{i_{1},\ldots,i_{k}\}$ with
\[\Leb\left(A^{i_{k+1}}\cap[R_{1}^{(k)},R_{2}^{(k)}]\right)>\overline{\theta}(\zeta_{1}) \cdot\Leb([R_{1}^{(k)},R_{2}^{(k)}]).\]
 Applying Lemma \ref{joinings202104.43} to $A^{i_{k+1}}$,   we obtain  an $\epsilon$-block $\overline{\BL}_{k+1}$ with the corresponding time interval $[R_{1}^{(k+1)},R_{2}^{(k+1)}]\subset[R_{1}^{(k)},R_{2}^{(k)}]$ and
 \begin{equation}\label{joinings202104.48}
|\overline{\BL}_{k+1}|=R_{2}^{(k+1)}-R_{1}^{(k+1)}\geq\zeta_{1}\cdot\Leb([R_{1}^{(k)},R_{2}^{(k)}])\geq\zeta_{1}^{k+1}\lambda>\vartheta\lambda.
 \end{equation}
 It follows from (\ref{joinings202104.47}) that
 \begin{align}
  &  \Leb\left(\coprod_{i\not\in\{i_{1},\ldots,i_{k}\}}A^{i}\cap[R_{1}^{(k+1)},R_{2}^{(k+1)}]\right) \;\nonumber\\
    =& \Leb([R_{1}^{(k+1)},R_{2}^{(k+1)}])- \Leb\left(\left(\coprod_{i\not\in\{i_{1},\ldots,i_{k}\}}A^{i}\right)^{c}\cap[R_{1}^{(k+1)},R_{2}^{(k+1)}]\right)  \;\nonumber\\
  \geq & \Leb([R_{1}^{(k+1)},R_{2}^{(k+1)}])- \Leb\left(\left(\coprod_{i\not\in\{i_{1},\ldots,i_{k}\}}A^{i}\right)^{c}\cap[R_{1}^{(k)},R_{2}^{(k)}]\right) \;\nonumber\\
  > & \Leb([R_{1}^{(k+1)},R_{2}^{(k+1)}])-  b_{k}\cdot\Leb([R_{1}^{(k)},R_{2}^{(k)}])  \; \nonumber
\end{align}
and so by (\ref{joinings202104.48}), we obtain
 \begin{equation}\label{joinings202104.45}
  \Leb\left(\coprod_{i\not\in\{i_{1},\ldots,i_{k}\}}A^{i}\cap[R_{1}^{(k+1)},R_{2}^{(k+1)}]\right)/\Leb([R_{1}^{(k+1)},R_{2}^{(k+1)}])>1-b_{k}\zeta_{1}^{-1}.
 \end{equation}
   Then we face a  dichotomy:
   \begin{enumerate}[\ \ \ (1)]
     \item     $\Leb(A^{i_{k+1}}\cap[R_{1}^{(k+1)},R_{2}^{(k+1)}])/\Leb([R_{1}^{(k+1)},R_{2}^{(k+1)}])\geq\overline{\theta}_{\eta}(\zeta_{2})$;
     \item     $\Leb(A^{i_{k+1}}\cap[R_{1}^{(k+1)},R_{2}^{(k+1)}])/\Leb([R_{1}^{(k+1)},R_{2}^{(k+1)}])<\overline{\theta}_{\eta}(\zeta_{2})$.
   \end{enumerate}
    In the case (1), we take $i(\lambda)=i_{k+1}$, $[R_{1}^{\prime}(\lambda),R_{2}^{\prime}(\lambda)]=[R_{1}^{(k)},R_{2}^{(k)}]$, $\overline{\BL}=\overline{\BL}_{k+1}$. By (\ref{joinings202104.44}) (\ref{joinings202104.56}) (\ref{joinings202104.48}), we have
 \begin{align}
 &  \Leb\left(A_{\epsilon}^{i(\lambda)}\cap[R_{1}^{(k+1)},R_{2}^{(k+1)}]\right)\;\nonumber\\
=  & \Leb\left(A^{i(\lambda)}\cap[R_{1}^{(k+1)},R_{2}^{(k+1)}]\right)-\Leb\left((A^{i(\lambda)}_{\epsilon})^{c}\cap A^{i(\lambda)}\cap[R_{1}^{(k+1)},R_{2}^{(k+1)}]\right)\;\nonumber\\
\geq & \overline{\theta}_{\eta}(\zeta_{2})\cdot\Leb([R_{1}^{(k+1)},R_{2}^{(k+1)}])-\overline{\theta}_{\eta}\left((\zeta_{2}\lambda^{-\eta})^{\frac{1}{1+\eta}}\right)\cdot\Leb([R_{1}^{(k)},R_{2}^{(k)}]) \;\nonumber\\
\geq&   \left(\overline{\theta}_{\eta}(\zeta_{2})\zeta_{1} -\overline{\theta}_{\eta}\left((\zeta_{2}\lambda^{-\eta})^{\frac{1}{1+\eta}}\right)\right)\cdot\Leb([R_{1}^{(k)},R_{2}^{(k)}])\;\nonumber \\
  >& \frac{1}{2} \overline{\theta}_{\eta}(\zeta_{2})\zeta_{1}\cdot\zeta_{1}^{k}\lambda \geq \vartheta\lambda\; \label{joinings202104.55}
\end{align}
and the consequence of Proposition \ref{joinings202104.57} follows.
In    the case (2),  by (\ref{joinings202104.45}), we have
\begin{equation}\label{joinings202104.49}
\Leb\left(\coprod_{i\not\in\{i_{1},\ldots,i_{k+1}\}}A^{i}\cap[R_{1}^{(k+1)},R_{2}^{(k+1)}]\right)/\Leb([R_{1}^{(k+1)},R_{2}^{(k+1)}])>1-b_{k}\zeta_{1}^{-1}-\overline{\theta}_{\eta}(\zeta_{2}).
\end{equation}
Now note that
\begin{itemize}
    \item  $i_{k+1}\not\in\{i_{1},\ldots,i_{k}\}$ has been chosen,
    \item choose $b_{k+1}=b_{k}\zeta_{1}^{-1}+\overline{\theta}_{\eta}(\zeta_{2})$
\end{itemize}
and then (\ref{joinings202104.49}) coincides with (\ref{joinings202104.47}) by replacing $k$ by $k+1$.
Thus, we can apply  the algorithm again by    replacing $k$ by $k+1$.

After applying the algorithm, we either stop in the middle and finish the proof, or we determine
\begin{itemize}
    \item  $i_{1},\ldots,i_{n-1}\in\{1,\ldots,n\}$ without repetition,
    \item a sequence $\{b_{k}\}_{k=0}^{n-1}$ of positive numbers with $b_{0}=2\sigma$ and
    \begin{equation}\label{joinings202104.50}
      b_{k+1}=b_{k}\zeta_{1}^{-1}+\overline{\theta}_{\eta}(\zeta_{2}).
    \end{equation}
  \end{itemize}
Let $i(\lambda)$ be the only element in $\{1,\ldots,n\}\setminus \{i_{1},\ldots,i_{n-1}\}$. Let $[R_{1}^{\prime}(\lambda),R_{2}^{\prime}(\lambda)]=[R_{1}^{(n-1)},R_{2}^{(n-1)}]$. Besides, by (\ref{joinings202104.50}) we calculate
\[b_{n-1}=2\sigma\zeta_{1}^{-(n-1)}+\overline{\theta}_{\eta}(\zeta_{2})\frac{\zeta_{1}^{-(n-1)}-1}{\zeta_{1}^{-1}-1}.\]
Now we try to do the algorithm one more time. Thus,  we apply again Lemma \ref{joinings202104.43} to $A^{i(\lambda)}$, and then  we obtain  an $\epsilon$-block   $\overline{\BL}=\overline{\BL}_{n}$   with the corresponding time interval $[R_{1}^{(n)},R_{2}^{(n)}]\subset[R_{1}^{(n-1)},R_{2}^{(n-1)}]$ satisfying
(\ref{joinings202104.48}) (\ref{joinings202104.45}), i.e.
\begin{equation}\label{joinings202104.52}
  |\overline{\BL}_{n}|=\Leb([R_{1}^{(n)},R_{2}^{(n)}])\geq\zeta_{1}\cdot\Leb([R_{1}^{(n-1)},R_{2}^{(n-1)}])\geq \zeta_{1}^{n}\lambda>\vartheta\lambda,
\end{equation}
\begin{multline}\label{joinings202104.51}
  \Leb\left(A^{i(\lambda)}\cap[R_{1}^{(n)},R_{2}^{(n)}]\right)/\Leb([R_{1}^{(n)},R_{2}^{(n)}])\\
>1-b_{n-1}\zeta_{1}^{-1}=1-2\sigma\zeta_{1}^{-n}-\overline{\theta}_{\eta}(\zeta_{2})\frac{\zeta_{1}^{-n}-\zeta_{1}^{-1}}{\zeta_{1}^{-1}-1}\geq\overline{\theta}_{\eta}(\zeta_{2})
\end{multline}
where the last inequality of (\ref{joinings202104.51}) follows from (\ref{joinings202104.53}) (\ref{joinings202104.54}).
Then, as in (\ref{joinings202104.55}), we calculate
\[  \Leb\left(A_{\epsilon}^{i(\lambda)}\cap[R_{1}^{(n)},R_{2}^{(n)}]\right) \geq \left(\overline{\theta}_{\eta}(\zeta_{2})\zeta_{1} -\overline{\theta}_{\eta}\left((\zeta_{2}\lambda^{-\eta})^{\frac{1}{1+\eta}}\right)\right)\cdot\Leb([R_{1}^{(n-1)},R_{2}^{(n-1)}])
  > \vartheta\lambda\]
  where the last inequality follows from (\ref{joinings202104.44}) (\ref{joinings202104.56}) (\ref{joinings202104.52}).
\end{proof}

\section{Invariance}\label{joinings202106.166}
Let $G_{X}=SO(n_{X},1)$ and $\Gamma_{X}\subset G_{X}$ be a lattice. Let $(X,\mu)$ be the homogeneous space $X=G_{X}/\Gamma_{X}$ equipped with the Lebesgue measure $\mu$, and let $\phi^{U_{X}}_{t}=u_{X}^{t}$ be a unipotent flow on $X$ as before.  Besides, let $G_{Y}$ be a Lie group and $\Gamma_{Y}\subset G_{Y}$ be a lattice.  $(Y,m_{Y})$ be the homogeneous space $Y=G_{Y}/\Gamma_{Y}$ equipped with the Lebesgue measure $m_{Y}$ and let $\phi^{U_{Y}}_{t}=u^{t}_{Y}$ be a unipotent flow on $Y$.
Next, choose $\tau_{Y}\in \mathbf{K}_{\kappa}(Y)$  a  positive integrable function  $\tau_{Y}$ on $Y$ such that $\tau_{Y},\tau_{Y}^{-1}$ are bounded and satisfies (\ref{joinings202104.26}). Then define the measure $d\nu\coloneqq\tau_{Y} dm_{Y}$ and so the time-change flow $\phi^{U_{Y},\tau_{Y}}_{t}=\widetilde{u}_{Y}^{t}$ preserves the measure $\nu$ by  Remark \ref{joinings202104.59}. Also recall from (\ref{joinings202105.9}) that
\[ u^{t}_{Y}y=\phi^{U_{Y},\tau_{Y}}_{z(y,t)}(y)=\widetilde{u}_{Y}^{z(y,t)}(y). \]

 We shall to study the joinings of $(X,\mu,u_{X}^{t})$ and  $(Y,\nu,\widetilde{u}_{Y}^{t})$. Let $\rho$ be an ergodic \textit{joining} of $u_{X}^{t}$ and  $\widetilde{u}_{Y}^{t}$\index{joinings}, i.e. $\rho$ is a probability measure on $X\times Y$, whose marginals on $X$ and $Y$ are $\mu$ and $\nu$ respectively, and which is $(u_{X}^{t}\times\widetilde{u}_{Y}^{t})$-ergodic. As indicated at the end of Section \ref{joinings202105.3}, when $\rho$ is not the product measure $\mu\times\nu$,
 we apply   Theorem \ref{joinings202012.1} and then obtain a compact subgroup $C^{\rho}\subset C_{G_{X}}(U_{X})$ such that $\overline{\rho}\coloneqq\pi_{\ast}\rho$ is an    ergodic joining $u_{X}^{t}$ and  $\widetilde{u}_{Y}^{t}$ on $C^{\rho}\backslash X\times Y$ under the natural projection $\pi:   X\times Y\rightarrow C^{\rho}\backslash X\times Y$. Besides, it is a finite extension of $\nu$, i.e. $\supp\overline{\rho}_{y}$ consists of exactly $n$ points $\overline{\psi}_{1}(y),\ldots,\overline{\psi}_{n}(y)$ for $\nu$-a.e. $y\in Y$ (without loss of generality, we shall assume that it holds for all $y\in Y$).
By Kunugui's theorem, we obtain $\psi_{i}:Y\rightarrow X$ so that $P_{X}\circ \psi_{i}=\overline{\psi}_{i}$ where $P_{X}:X\rightarrow C^{\rho}\backslash X$.

\subsection{Central direction}\label{joinings202105.24} We want to study the behavior of $\overline{\psi}_{p}$ along the central direction $C_{G_{Y}}(U_{Y})$ of $U_{Y}$.
In the following, assume that $\rho$ is a $(u_{X}^{t}\times  \widetilde{u}_{Y}^{t})$-joining. Then by (\ref{joinings202104.23}), we get that
   \[\overline{\psi}_{p}(u^{t}_{Y}y)=\overline{\psi}_{p}(\widetilde{u}_{Y}^{z(y,t)}(y))=u_{X}^{z(y,t)}\overline{\psi}_{i_{p}}(y)\]
   where the index $i_{p}=i_{p}(y,t)\in\{1,\ldots,n\}$ is determined by
   \[(u_{X}^{-z(y,t)}\times \widetilde{u}_{Y}^{-z(y,t)} )(\overline{\psi}_{p}(\widetilde{u}_{Y}^{z(y,t)}(y)),\widetilde{u}_{Y}^{z(y,t)}(y))\in\hat{\psi}_{i_{p}}(Y).\]
   Now we orderly fix the following data so that the propositions in Section \ref{joinings202103.13} can be used:
\begin{itemize}
  \item  fix $\kappa\in(0,2\eta_{0})$ satisfying (\ref{joinings202104.26}), where $\eta_{0}>0$ comes from   Proposition  \ref{joinings202104.29};
  \item  fix $\sigma\in(0,\sigma_{0})$, where $\sigma_{0}\approx0$ comes from both Proposition  \ref{joinings202104.29} and Proposition   \ref{joinings202104.57};
  \item  fix  $\epsilon\in(0,\epsilon_{0})$ as in (\ref{joinings202104.34})
\end{itemize}
such that the following holds:
   \begin{itemize}
     \item (Effective ergodicity) By (\ref{joinings202105.45}),   there is $K_{1}\subset Y$ with $\nu(K_{1})>1-\sigma/6$ and $t_{K_{1}}>0$ such that
         \begin{equation}\label{joinings202104.24}
           |t-z(y,t)|=O(t^{1-\kappa})
         \end{equation}
         for   all $t\geq t_{K_{1}}$ and $y\in K_{1}$. Note that using ergodic theorem, we have
         \begin{equation}\label{joinings202104.25}
          |t-z(y,t)|=o(t)
         \end{equation}
         for $\nu$-almost all $y\in Y$.
     \item (Distinguishing $\overline{\psi}_{p},\overline{\psi}_{q}$) There is  $K_{2}\subset Y$ with $\nu(K_{2})>1-\sigma/6$ such that
         \begin{equation}\label{joinings202105.8}
          d(\overline{\psi}_{p}(y),\overline{\psi}_{q}(y))>100\epsilon
         \end{equation}
          for $y\in K_{2}$, $1\leq p< q\leq n$.
     \item (Lusin's theorem) There is  $K_{3}\subset Y$ such that $\nu(K_{3})>1-\sigma/6$ and  $\overline{\psi}_{p}|_{K_{3}}$ is uniformly continuous for all $p\in\{1,\ldots,n\}$. Thus,   there is $\delta>0$ such that
         \begin{equation}\label{joinings202105.33}
          d_{\overline{X}}(\overline{\psi}_{p}(y_{1}),\overline{\psi}_{p}(y_{2}))<\epsilon
         \end{equation}
          for $p\in\{1,\ldots,n\}$, $d_{Y}(y_{1},y_{2})<\delta$ and $y_{1},y_{2}\in K_{3}$.
   \end{itemize}
 Given  $K\subset\overline{X}$  by Proposition \ref{joinings202104.29}, let
 \begin{equation}\label{joinings202105.17}
  K^{0}\coloneqq K_{1}\cap K_{2}\cap K_{3}\cap\bigcap_{p=1}^{n}\overline{\psi}_{p}^{-1}(K).
 \end{equation}
 Here we choose   $\overline{\mu}(K)$ being so large that  $m_{Y}(K^{0})>1-\sigma/2$.

Fix   $c\in C_{G_{Y}}(U_{Y})\cap B_{G_{Y}}(e,\delta)$. We choose arbitrarily a representative $g_{\overline{\psi}_{p}(y)}\in G_{X}$ of  $\overline{\psi}_{p}(y)$. Then there is a representative $g_{\overline{\psi}_{p}(cy)}\in G_{X}$ so that
 \begin{itemize}
   \item  $\overline{g_{\overline{\psi}_{p}(y)}}$ and $\overline{g_{\overline{\psi}_{p}(cy)}}$ lie in the same fundamental domain;
   \item the difference $g(y)=g_{\overline{\psi}_{p}(cy)}g_{\overline{\psi}_{p}(y)}^{-1}=h^{(p)}(y)\exp(v^{(p)}(y))$ where
       \begin{equation}\label{joinings202105.18}
         h^{(p)}(y)=\left[
            \begin{array}{ccc}
              a^{(p)}(y)&  b^{(p)}(y) \\
             c^{(p)}(y)&  d^{(p)}(y)\\
            \end{array}
          \right]\in SO_{0}(2,1),\ \ \ v^{(p)}=b_{0}^{(p)}(y)v_{0}+\cdots+b_{\varsigma}^{(p)}(y)v_{\varsigma}\in V_{\varsigma}.
       \end{equation}
 \end{itemize}
 Further, applying the effectiveness of the unipotent flow, we shall show that the difference $g(y)$ has to lie in the centralizer $C_{G_{X}}(U_{X})$.

\begin{prop}\label{joinings202104.58}
   Let the notation and assumptions be as above.    For the quantities in (\ref{joinings202105.18}),  there is a measurable set $S(c)\subset Y$ with $\nu(S(c))>0$ such that
   \[b^{(p)}(y)=0,\ \ \ a^{(p)}(y)=d^{(p)}(y)=1,\ \ \ b^{(p)}_{0}(y)=\cdots=b^{(p)}_{\varsigma-1}(y)=0\]
   for $y\in S(c)$, $p\in\{1,\ldots,n\}$.
\end{prop}
\begin{proof}
Consider the measure of the set
\begin{multline}
  Y_{l}(c)\coloneqq\{y\in Y:|b^{(p)}(y)|, |a^{(p)}(y)-1|, |d^{(p)}(y)-1|, |b^{(p)}_{0}(y)|,\cdots,|b^{(p)}_{\varsigma-1}(y)|<1/l, \\
 \text{for any } p\in\{1,\ldots,n\}\}\nonumber
\end{multline}
for $l\in\mathbf{Z}^{+}$. We shall show that $S(c)\coloneqq\bigcap_{l}Y_{l}(c)$ satisfies the requirement.
By ergodic theorem, we have
\begin{equation}\label{joinings202104.62}
 m_{Y}(Y_{l}(c))=\lim_{\lambda\rightarrow\infty}\frac{1}{\lambda}\int_{0}^{\lambda}\mathbf{1}_{Y_{l}(c)}(u^{r}_{Y}y)dr
\end{equation}
   for $m_{Y}$-a.e. $y\in Y$, where $m_{Y}$ denotes the Lebesgue measure on $Y$.

On the other hand,  by ergodic theorem, for $m_{Y}$-a.e. $y\in Y$, there is $A_{c,y}\subset\mathbf{R}^{+}$ and $\lambda_{0}(y)>0$  such that
    \begin{itemize}
      \item for $r\in A_{c,y}$, we have
     \[ u^{r}_{Y}y ,u^{r}_{Y}cy\in K^{0};\]
     \item
      $\Leb(A_{c,y}\cap[0,\lambda])\geq (1-2\sigma)\lambda$ whenever $\lambda\geq\lambda_{0}(y)$.
    \end{itemize}
 Then by the assumptions, we have
 \begin{equation}\label{joinings202105.19} A_{c,y} \subset \left\{r\in[0,\infty):d_{\overline{X}}(\overline{\psi}_{p}(u^{r}_{Y}y),\overline{\psi}_{p}(u^{r}_{Y}cy))<\epsilon,\ p\in\{1,\ldots,n\}\right\}.
 \end{equation}
It follows that for $r\in A_{c,y}$, we have
\begin{equation}\label{joinings202105.34}
  d_{\overline{X}}(u_{X}^{z(y,r)}\overline{\psi}_{i_{p}(y,r)}(y),u_{X}^{z(cy ,r)}\overline{\psi}_{i_{p}(cy,r)}(cy))<\epsilon
\end{equation}
     for any $p\in\{1,\ldots,n\}$.
  Now we restrict our attention on $A_{c,y}\cap[0,\lambda]$ with $\lambda\geq\lambda_{0}(y)$. For simplicity, we assume that $0\in A_{c,y}$.
   Let $I=((p_{1},p_{2}),\ldots, (p_{2n-1},p_{2n}))\in\{1,\ldots,n\}^{2n}$ be a sequence of  indexes and
   \begin{equation}\label{joinings202105.20}
     A_{c,y}^{I}\coloneqq\{r\in A_{c,y}: p_{2k-1}=i_{k}(y,r),\ p_{2k}=i_{k}(cy,r)\text{ for all }k\in\{1,\ldots,n\}\}.
   \end{equation}
  Then  $A=A_{c,y}^{I}$, $R_{0}=t_{K_{1}}$, $t(r)=z(cy ,r)$, $s(r)=z(y ,r)$ satisfy (\ref{joinings202104.30}) (\ref{joinings202104.5}) for  points
   \[\overline{\psi}_{p_{2k-1}}(y),\overline{\psi}_{p_{2k}}(cy)\in K\]
  for all $k\in\{1,\ldots,n\}$.

Since $A_{c,y}=\coprod_{I\in\{1,\ldots,n\}^{2n}}A^{I}_{c,y}$ (is a disjoint union because of (\ref{joinings202105.8})), by Proposition \ref{joinings202104.57}, for any $\lambda\geq \lambda_{0}$, there exists  one $A_{c,y}^{I(\lambda)}$  and $[R_{1}^{\prime},R_{2}^{\prime}]\subset[0,\lambda]$  such that there exists an $\epsilon$-block $\overline{\BL}=\{(x^{\prime},y^{\prime}),(x^{\prime\prime},y^{\prime\prime})\}\in \beta_{2}(A_{c,y}^{I(\lambda)}\cap[R_{1}^{\prime},R_{2}^{\prime}])$ with the corresponding time
interval $[R_{1},R_{2}]$ such that
   \[R_{2}-R_{1}>\vartheta\lambda,\ \ \ \Leb\left(A^{I(\lambda)}_{\epsilon}\cap [R_{1},R_{2}]\right)>\vartheta\lambda\]
   where $A^{I(\lambda)}_{\epsilon}$ is the non-shifting time of $A_{c,y}^{I(\lambda)}$. Then by the definition of $A^{I(\lambda)}_{\epsilon}$, we know that
    \[d_{C^{\rho}\backslash G}\left(u_{X}^{z(cy,r)}.\overline{g_{\overline{\psi}_{i_{p}(cy,r)}(cy)}},u_{X}^{z(y,r)}.\overline{g_{\overline{\psi}_{i_{p}(y,r)}(y)}}\right)<\epsilon\]
    for $r\in A^{I(\lambda)}_{\epsilon}$, $p\in\{1,\ldots,n\}$.        Recall from  (\ref{joinings202104.22}) that points in $K$ have injectivity radius at least $\epsilon_{0}$. Thus,  for $r\in A^{I(\lambda)}_{\epsilon}$,
    \[u_{X}^{z(y,r)}.\overline{g_{\overline{\psi}_{i_{p}(y,r)}(y)}}\ \ \ \text{ and } \ \ \ u_{X}^{z(cy,r)}.\overline{g_{\overline{\psi}_{i_{p}(cy,r)}(cy)}}\]
     lie in the same fundamental domain. Thus, if $r\in A^{I(\lambda)}_{\epsilon}$ and
    \[\overline{g_{\overline{\psi}_{p}(u^{r}_{Y}y)}}=u_{X}^{z(y,r)}.\overline{g_{\overline{\psi}_{i_{p}(y,r)}(y)}}\]
    then we get
    \[\overline{g_{\overline{\psi}_{p}(u^{r}_{Y}cy)}}=u_{X}^{z(cy,r)}.\overline{g_{\overline{\psi}_{i_{p}(cy,r)}(cy)}}.\]

    Recall that the difference of $u_{X}^{z(y,r)}.\overline{g_{\overline{\psi}_{i_{p}(y,r)}(y)}}$, $u_{X}^{z(cy,r)}.\overline{g_{\overline{\psi}_{i_{p}(cy,r)}(cy)}}$ for $r\in A^{i(\lambda)}_{\epsilon}\cap [R_{1},R_{2}]$ was estimated by (\ref{joinings202104.36}) (see also (\ref{joinings202105.21}) (\ref{joinings202104.61}) (\ref{joinings202104.60})). In particular, for $r\in A^{i(\lambda)}_{\epsilon}\cap [R_{1},R_{2}]$, the quantities of
    \[g(u^{r}_{Y}y)=g_{\overline{\psi}_{p}(cu^{r}_{Y}y)}g_{\overline{\psi}_{p}(u^{r}_{Y}y)}^{-1}=u_{X}^{z(cy,r)}g_{\overline{\psi}_{i_{p}(cy,r)}(cy)}\left(u_{X}^{z(y,r)}g_{\overline{\psi}_{i_{p}(y,r)}(y)}\right)^{-1}\]
    that need to estimate in $Y_{l}(c)$ are all decreasing as $\lambda\rightarrow\infty$.
Then   given   $l\in\mathbf{Z}^{+}$, there is a sufficiently large $\lambda$ such that
\[ \int_{0}^{\lambda}\mathbf{1}_{Y_{l}(c)}(u^{r}_{Y}y)dr \geq\Leb\left(A^{i(\lambda)}_{\epsilon}\cap [R_{1},R_{2}]\right)>\vartheta\lambda.\]
Thus, by (\ref{joinings202104.62}), we have $m_{Y}(Y_{l}(c))>\vartheta$. Now letting $\lambda\rightarrow\infty$ and then $l\rightarrow\infty$, we see that $m_{Y}(\bigcap_{l}Y_{l}(c))>\vartheta$. Finally, by Remark \ref{joinings202104.59} and $\tau_{Y}\in \mathbf{K}_{\kappa}(Y)$, we obtain   $\nu(\bigcap_{l}Y_{l}(c))>0$.
\end{proof}

Using Proposition \ref{joinings202104.58}, we immediately obtain
\begin{cor}\label{joinings202105.7}
  There is a measurable map $\varpi:C_{G_{Y}}(U_{Y})\times X\times Y\rightarrow C_{G_{X}}(U_{X})$ that induces a map  $\widetilde{S}_{c}:\supp (\rho)\rightarrow\supp (\rho)$ by
   \begin{equation}\label{joinings202105.1}
    \widetilde{S}_{c}: (x,y) \mapsto (\varpi(c,x,y)x,cy)
   \end{equation}
   for all $c\in C_{G_{Y}}(U_{Y})$,   $\rho$-a.e. $(x,y)\in X\times Y$. Moreover, we have
    \begin{align}
\varpi(c,x,y)= \ &  u_{X}^{-z(cy,t)}\varpi(c,(u_{X}^{z(y,t)}\times \widetilde{u}_{Y}^{z(y,t)}).(x,y))u_{X}^{z(y,t)}\;\label{joinings202105.4}\\
 \varpi(c_{1}c_{2},x,y) =\ &  \varpi(c_{1},\varpi(c_{2},x,y)x,c_{2}y)\varpi(c_{2},x,y)\; \label{joinings202105.6}
\end{align}
   for $c,c_{1},c_{2}\in C_{G_{Y}}(U_{Y})$,   $\rho$-a.e. $(x,y)\in X\times Y$, $t\in\mathbf{R}$.
\end{cor}
\begin{rem}
   Note that when $c\in\exp (\mathbf{R}U_{Y})$, $\varpi$ reduces to an element in $\exp (\mathbf{R}U_{X})$; in fact, we have
    \[\varpi(u^{t}_{Y},x,y)=u_{X}^{z(y,r)}=\exp(z(y,t)U_{X})\]
     for all $t\in\mathbf{R}$.

     On the other hand, for distinct $q_{1},q_{2}\in\{1,\ldots,n\}$, any $c\in C_{G_{Y}}(U_{Y})$, we  have
     \begin{equation}\label{joinings202105.15}
      w(c,\psi_{q_{1}}(y),y)\psi_{q_{1}}(y)\in C^{\rho}\psi_{p_{1}}(y), \ \ \ w(c,\psi_{q_{2}}(y),y)\psi_{q_{2}}(y)\in C^{\rho}\psi_{p_{2}}(y)
     \end{equation}
   for distinct $p_{1},p_{2}\in\{1,\ldots,n\}$; for otherwise it would lead to $\psi_{q_{2}}(y)\in C_{G_{X}}(U_{X})\psi_{q_{1}}(y)$, which contradicts the definition of $\psi$ (cf. Section \ref{joinings202103.10}).
\end{rem}
\begin{proof}[Proof of Corollary \ref{joinings202105.7}]
 Fix  $c\in C_{G_{Y}}(U_{Y})\cap B(e,\delta)$. Proposition \ref{joinings202104.58} provides us a subset $S(c)\subset Y$ with $\nu(S(c))>0$ such that
 \begin{equation}\label{joinings202105.29}
   \psi_{p}(cy)=w_{p}(c,y)\psi_{p}(y)
 \end{equation}
   for $y\in S(c)$, $w_{p}(c,y)\in C_{G_{X}}(U_{X})$. Besides, for $y,u^{r}_{Y}y\in S(c)$, we know that
   \[w_{p}(c,u^{r}_{Y}y)u_{X}^{z(y,r)}\psi_{i_{p}(y,r)}(y)=\psi_{p}(u^{r}_{Y}cy)=u_{X}^{z(cy,r)}w_{i_{p}(cy,r)}(c,y)\psi_{i_{p}(cy,r)}(y).\]
   Thus, $\psi_{i_{p}(y,r)}(y)\in C_{G_{X}}(U_{X})\psi_{i_{p}(cy,r)}(y)$ and so $i_{p}(y,r)=i_{p}(cy,r)$. It follows that
   \begin{equation}\label{joinings202105.22}
     w_{p}(c,u^{r}_{Y}y)u_{X}^{z(y,r)}=u_{X}^{z(cy,r)}w_{i_{p}(cy,r)}(c,y)=u_{X}^{z(cy,r)}w_{i_{p}(y,r)}(c,y)
   \end{equation}
   for $y,u^{r}_{Y}y\in S(c)$.

    Thus,  for $y\in S(c)$, we define
   \[\varpi(c,\psi_{p}(y),y)\coloneqq w_{p}(c,y).\]
   Let $\pi_{Y}:\supp(\rho)\rightarrow Y$ be the natural projection. Then for  $(x,y)\in \pi^{-1}_{Y}(S(c))$, we know that $C^{\rho}x= C^{\rho}\psi_{p_{x}}(y)$ for some $p_{x}\in\{1,\ldots,n\}$. Thus, given $\psi_{p_{x}}(y)=k^{\rho}_{x}x$ for some $k^{\rho}_{x}\in C^{\rho}$, we define
   \begin{equation}\label{joinings202105.11}
    \varpi(c,x,y)\coloneqq (k^{\rho}_{x})^{-1}w_{p_{x}}(c,y)k^{\rho}_{x}.
   \end{equation}
   Thus, we successfully define $\varpi(c,\cdot,\cdot)$ for $\pi^{-1}_{Y}(S(c))$. Then the $(u_{X}^{t}\times\tilde{u}_{Y}^{t})$-flow helps us to define $\varpi(c,\cdot,\cdot)$ for all $\rho$-a.e. $(x,y)\in X\times Y$. More precisely, for $(x,y)\in X\times Y$ (in a $\rho$-conull set), we can choose $t=t(x,y)\in\mathbf{R}$ such that $(u_{X}^{z(y,t)}x,u_{Y}^{t}y)\in\pi_{Y}^{-1}(S(c))$. Then define
      \begin{align}
\varpi(c,x,y)\coloneqq  &  u_{X}^{-z(cy,t)}\varpi(c,u_{X}^{z(y,t)}x,u_{Y}^{t}y)u_{X}^{z(y,t)}\;\nonumber\\
  =& u_{X}^{-z(cy,t)}\varpi(c,(u_{X}^{z(y,t)}\times \widetilde{u}_{Y}^{z(y,t)}).(x,y))u_{X}^{z(y,t)}.\; \label{joinings202105.2}
\end{align}
 (Note that (\ref{joinings202105.22}) tells us that  (\ref{joinings202105.2}) holds true for $y,u_{Y}^{t}y\in S(c)$ and thus $\varpi$ is well-defined.)  Finally, for general $c\in C_{G_{Y}}(U_{Y})$, choose $k\in C_{G_{Y}}(U_{Y})\cap B(e,\delta)$ such that $k^{m}=c$, and then define iteratively
   \[\varpi(k^{i+1},x,y)\coloneqq \varpi(k^{i},\varpi(k,x,y)x,ky)\varpi(k,x,y)\]
   and finally reach $c=k^{m}$. Then the map (\ref{joinings202105.1}) is well defined on $\supp(\rho)$.
\end{proof}

In light of Corollary \ref{joinings202105.7}, we consider the  decomposition (\ref{time change184}) and write
\begin{equation}\label{time change188}
  \varpi(c,x,y)=u_{X}^{\alpha(c,x,y)} \beta(c,x,y)
\end{equation}
where $\alpha(c,x,y)\in\mathbf{R}$ and $\beta(c,x,y)\in\exp V^{\perp}_{C_{X}}$. Then by (\ref{joinings202105.4}), we have
 \begin{align}
z(cy,t)+\alpha(c,x,y)=\ & \alpha(c,(u^{z(y,t)}\times \widetilde{u}^{z(y,t)}).(x,y))+z(y,t),\;\label{joinings202105.10}\\
 \beta(c,x,y)=\ & \beta(c,(u^{z(y,t)}\times \widetilde{u}^{z(y,t)}).(x,y))\; \label{joinings202105.5}
\end{align}
for all $t\in\mathbf{R}$.

First consider $\alpha$. Recall that for fixed $y\in Y$, $\supp(\rho_{y})=\bigsqcup_{p=1}^{n}C^{\rho}\psi_{p}(y)$. Then by (\ref{joinings202105.10}), for $\nu$-a.e. $y\in Y$, $x\in\supp(\rho_{y})$,  we have
\begin{equation}\label{joinings202105.12}
  \alpha(c,x,y)-\alpha(c,(u^{z(y,t)}\times \widetilde{u}^{z(y,t)}).(x,y))=z(y,t)-z(cy,t)
\end{equation}
for all $r\in\mathbf{R}$. Besides, by (\ref{joinings202105.11}), we have
\begin{equation}\label{joinings202105.43}
 \alpha(c,x,y)=\alpha(c,kx,y)
\end{equation}
for all $x\in\supp(\rho_{y})$, $k\in C^{\rho}$. By (\ref{joinings202105.10}), for any $(x_{1},y),(x_{2},y)\in\supp(\rho)$, we have
 \begin{equation}\label{joinings202105.42}
   \alpha(c,x_{1},y)-\alpha(c,x_{2},y)= \alpha(c,(u^{t}\times \widetilde{u}^{t}).(x_{1},y))-\alpha(c,(u^{t}\times \widetilde{u}^{t}).(x_{2},y)).
 \end{equation}
Define $\alpha_{\max}:C_{G_{Y}}(U_{Y})\times X\times Y\rightarrow\mathbf{R}$ by
\[ \alpha_{\max}:(c,x,y)\mapsto\max\left\{r\in\mathbf{R}:\rho_{y}\{x^{\prime}\in X:\alpha(c,x^{\prime},y)-\alpha(c,x,y)=r\}>0\right\}.   \]
Then by (\ref{joinings202105.42}), we have
\[\alpha_{\max}(c,(x,y))=\alpha_{\max}(c,(u_{X}^{t}\times \widetilde{u}_{Y}^{t}).(x,y))\]
for any $t\in\mathbf{R}$, $\rho$-a.e. $(x,y)\in X\times Y$. Thus, $\alpha_{\max}(c,x,y)\equiv\alpha_{\max}(c)$. Now if $\alpha_{\max}(c)>0$, then for $\rho$-a.e. $(x,y)$, there is $x^{\prime}\in X$ such that $\alpha(c,x^{\prime},y)=\alpha(c,x,y)+\alpha_{\max}(c)$, which contradicts the fact that $\alpha_{\max}(c,x,y)$ take at most finitely many different values for fixed $y$ (by (\ref{joinings202105.43})). Thus, we conclude that $\alpha_{\max}(c)\equiv0$ and so
\[\alpha(c,x,y)\equiv \alpha(c,y)\]
 for all $c\in C_{G_{Y}}(U_{Y})$, $\rho$-a.e. $(x,y)\in X\times Y$.

On the other hand, via the ergodicity of the flow $u_{X}^{t}\times \widetilde{u}_{Y}^{t}$, we conclude from (\ref{joinings202105.5}) that
\[\beta(c,x,y)\equiv\beta(c)\]
 for all $c\in C_{G_{Y}}(U_{Y})$. In particular, we have
 \[\varpi(c,x,y)=\varpi(c,y)=u_{X}^{\alpha(c,y)}\beta(y)\]
for all $c\in C_{G_{Y}}(U_{Y})$, $\rho$-a.e. $(x,y)\in X\times Y$.
Besides, we know from (\ref{joinings202105.6}) that $\beta(c_{1} c_{2})=\beta(c_{1})\beta(c_{2})$  via the definition of $\beta$. Further,
 we always have $d\beta(U_{Y})\equiv 0$. Therefore, we can restrict our attention to $V^{\perp}_{C}$ and conclude that $d\beta|_{V^{\perp}_{C}}: V^{\perp}_{C_{Y}}\rightarrow V^{\perp}_{C_{X}}$ is a Lie algebra homomorphism.

In sum, we obtain Theorem \ref{joinings202106.159} for the centralizer $C_{G_{Y}}(U_{Y})$.
\begin{thm}[Extra central invariance  of $\rho$]\label{joinings202105.41}
  For any $c\in C_{G_{Y}}(U_{Y})$, the map $S_{c}:X\times Y\rightarrow X\times Y$ defined by
  \[S_{c}:(x,y)\mapsto (\beta(c)x,\tilde{u}_{Y}^{-\alpha(c,y)}(cy))\]
  commutes with $u_{X}^{t}\times \widetilde{u}_{Y}^{t}$, and is $\rho$-invariant. Besides, $S_{c_{1}c_{2}}=S_{c_{1}}\circ S_{c_{2}}$ for any $c_{1},c_{2}\in C_{G_{Y}}(U_{Y})$, and  $S_{u_{Y}^{t}}=\id$ for $t\in\mathbf{R}$.
\end{thm}
\begin{proof}
 Clearly, $S_{c}$ is well-defined:
\begin{equation}\label{joinings202105.27}
 S_{c}(x,y)=(u_{X}^{-\alpha(c,y)}\times \widetilde{u}_{Y}^{-\alpha(c,y)}).\widetilde{S}_{c}(x,y)\in\supp(\rho)
\end{equation}
whenever $(x,y)\in\supp(\rho)$. Also, one may check that $S_{c_{1}c_{2}}=S_{c_{1}}S_{c_{2}}$ for any $c_{1},c_{2}\in C_{G_{Y}}(U_{Y})$, and  $S_{u_{Y}^{t}}=\id$ for $t\in\mathbf{R}$.
Next, by (\ref{joinings202105.10}), one verifies
\[(u_{X}^{z(y,r)}\times \widetilde{u}_{Y}^{z(y,r)}).S_{c}(x,y)=S_{c}(u_{X}^{z(y,r)}\times \widetilde{u}_{Y}^{z(y,r)}).(x,y)\]
for any $r\in\mathbf{R}$, $(x,y)\in\supp(\rho)$. That is, $(u_{X}^{t}\times \widetilde{u}_{Y}^{t})\circ S_{c}=S_{c}\circ (u_{X}^{t}\times \widetilde{u}_{Y}^{t})$.

Finally, let $\Omega$ be the set of $(u_{X}^{t}\times \widetilde{u}_{Y}^{t})$-generic points, and we want to show that there is a point $(x_{0},y_{0})\in\Omega\cap S_{c}^{-1}\Omega$. By (\ref{joinings202105.27}), it suffices to show that there is a point $(x_{0},y_{0})\in\Omega\cap \widetilde{S}_{c}^{-1}\Omega$. Fix $c\in C_{G_{Y}}(U_{Y})\cap B(e,\delta)$. Recall that
\[1=\rho(\Omega)=\int_{Y}\int_{C^{\rho}}\frac{1}{n} \sum_{p=1}^{n}\mathbf{1}_{\Omega}(k\psi_{p}(y),y)dm(k)d\nu(y).\]
Thus, there is $\Omega_{Y}\subset Y$ with $\nu(\Omega_{Y})=1$ such that
\begin{equation}\label{joinings202105.28}
  \int_{C^{\rho}} \frac{1}{n} \sum_{p=1}^{n}\mathbf{1}_{\Omega}(k\psi_{p}(y),y)dm(k)=1
\end{equation}
for $y\in\Omega_{Y}$. Since $\nu$ and $m_{Y}$ are equivalent, and $\Omega_{Y}\cap k^{-1}\Omega_{Y}$ is $m_{Y}$-conull, we get that  $\Omega_{Y}\cap c^{-1}\Omega_{Y}$ is $\nu$-conull. Choose $y_{0}\in \Omega_{Y}\cap c^{-1}\Omega_{Y}\cap S(c)$, where $S(c)$ is given by Proposition \ref{joinings202104.58} (cf. (\ref{joinings202105.29})). Then (\ref{joinings202105.28}) leads to
\[ \int_{C^{\rho}}  \mathbf{1}_{\Omega}(k\psi_{1}(y_{0}),y_{0})dm(k)=1,\ \ \  \int_{C^{\rho}}  \mathbf{1}_{\Omega}(k\psi_{1}(cy_{0}),cy_{0})dm(k)=1.\]
Then we can choose $k_{0}\in C^{\rho}$ such that $(k_{0}\psi_{1}(y_{0}),y_{0}),(k_{0}\psi_{1}(cy_{0}),cy_{0})\in\Omega$. Let $x_{0}\coloneqq k_{0}\psi_{1}(y_{0})$. Then by (\ref{joinings202105.29}) (\ref{joinings202105.11}), we have
\[\widetilde{S}_{c}(x_{0},y_{0})= (\varpi(c,y_{0})x_{0},cy_{0})=(k_{0}w_{p}(c,y_{0})k_{0}^{-1}k_{0}\psi_{1}(y_{0}),cy_{0})=(k_{0}\psi_{1}(cy_{0}),cy_{0}).\]
Thus,  $(x_{0},y_{0})\in\Omega\cap \widetilde{S}_{c}^{-1}\Omega$.

Hence, since $u_{X}^{t}\times \widetilde{u}_{Y}^{t}$ is $\rho$-ergodic, by ergodic theorem, for   any bounded continuous function $f$, we have
\begin{multline}
 \int f d\rho =\lim_{T\rightarrow\infty}\frac{1}{T}\int_{0}^{T} f((u_{X}^{t}\times\tilde{u}_{Y}^{t}).S_{c}(x_{0},y_{0}))dt \\
  =\lim_{T\rightarrow\infty}\frac{1}{T}\int_{0}^{T} f(S_{c}(u_{X}^{t}x_{0},\tilde{u}_{Y}^{t}y_{0}))dt=   \int f\circ S_{c} d\rho \nonumber
\end{multline}
and so $\rho=(S_{c})_{\ast}\rho$.
\end{proof}

In particular, we obtain
\begin{cor}[Extra central invariance of $\nu$]
   For any $c\in C_{G_{Y}}(U_{Y})$, the map $S^{Y}_{c}:Y\rightarrow Y$ defined by
  \[S_{c}^{Y}:y\mapsto \tilde{u}_{Y}^{-\alpha(c,y)}(cy)\]
  commutes with $\widetilde{u}^{t}$, and is $\nu$-invariant. Besides, $S^{Y}_{c_{1}c_{2}}=S^{Y}_{c_{1}}S^{Y}_{c_{2}}$ for any $c_{1},c_{2}\in C_{G_{Y}}(U_{Y})$, and  $S^{Y}_{u_{Y}^{t}}=\id$ for $t\in\mathbf{R}$.
\end{cor}
It is worth noting that (\ref{joinings202105.1}) can be interpreted through the language of cohomology.  More precisely, (\ref{joinings202105.1}) implies the time change $\tau_{Y}$ and $\tau_{Y}\circ c$ are measurably cohomologous.
\begin{thm}\label{dynamical systems2001} Let  $\tau_{Y}\in \mathbf{K}_{\kappa}(Y)$. Suppose that there is a nontrivial ergodic joining  $\rho\in J(u_{X}^{t},\phi_{t}^{U_{Y},\tau_{Y}})$ . Then
            $\tau_{Y}(y)$ and $\tau_{Y}( cy)$ are (measurably)  cohomologous along $u_{Y}^{t}$ for all $c\in C_{G_{Y}}(U_{Y})$. More precisely, the transfer function can be taken to be
            \[F_{c}(y)= \alpha(c,y).\]
\end{thm}
\begin{proof}
  By (\ref{joinings202105.10}), for $m_{Y}$-a.e. $y\in Y$, $x\in\supp(\rho_{y})$, we have
  \begin{align}
 &\int_{0}^{t}\tau_{Y}(u_{Y}^{s}y)-\tau_{Y}(u_{Y}^{s}cy)ds\;\nonumber\\
 =&\int_{0}^{t}\tau_{Y}(u_{Y}^{s}y)ds-\int_{0}^{t}\tau_{Y}(u_{Y}^{s}cy)ds\;\nonumber\\
=& z(y,t)-z(cy,t)\;\nonumber\\
=&\alpha(c,y)- \alpha(c,u_{Y}^{t}y). \;   \nonumber
\end{align}
Thus, we can take the transfer function as
\[F_{c}(y)\coloneqq \alpha(c,y) .\]
  Then $\tau_{Y}(y)$ and $\tau_{Y}( cy)$ are (measurably) cohomologous for all $c\in C_{G_{Y}}(U_{Y})$.
\end{proof}

If $\tau_{Y}(y)$ and $\tau_{Y}(cy)$ are cohomologous with a $L^{1}$ transfer function, then we are able to do more via the \textit{ergodic theorem}.
\begin{lem}\label{joinings202106.135} Given $c\in C_{G_{Y}}(U_{Y})$, if
\begin{itemize}
  \item  $c$ is $m_{Y}$-ergodic (as a left action on $Y$),
  \item $\tau_{Y}(y)$ and $\tau_{Y}(cy)$ are cohomologous with a $L^{1}$ transfer function $F_{c}(y)$,
\end{itemize}
then for $m_{Y}$-a.e. $y\in Y$, we have
\[\lim_{t\rightarrow\infty}\frac{1}{t}\alpha(c^{t},y)=\int \alpha(c,y)dm_{Y}(y).\]
\end{lem}
\begin{proof}
By (\ref{joinings202105.6}) (\ref{joinings202105.15}), for $c_{1},c_{2}\in C_{G_{Y}}(U_{Y})$, $m_{Y}$-a.e. $y\in Y$, we have the cocycle identity
\[  \alpha(c_{1}c_{2},y)= \alpha(c_{1},c_{2}y)+\alpha(c_{2},y) .\]
Thus, if $F_{c}(\cdot)\in L^{1}(Y)$, then by the ergodicity, we get
\begin{equation}\label{joinings202105.16}
 \lim_{k\rightarrow\infty}\frac{1}{k} \alpha(c^{k},y)=\lim_{k\rightarrow\infty}\frac{1}{k}\sum_{i=0}^{k}\alpha(c^{i},y)=\int \alpha(c,y)dm_{Y}(y).
\end{equation}
\end{proof}

\begin{rem}\label{joinings202106.134} The results obtained in Section \ref{joinings202105.24} also hold true for $\rho$ being a finite extension of $\nu$, when $(X,\phi^{U_{X},\tau_{X}}_{t})$ is a time-change of the unipotent flow on $X=SO(n_{X},1)/\Gamma_{X}$. For example,  we consider the case when $n_{X}=2$, $\tau_{X}\in C^{1}(X)$, $\tau_{Y}\equiv 1$ (in other words, $\phi^{U_{Y},\tau_{Y}}_{t}=\phi^{U_{Y}}_{t}=u_{Y}^{t}$ is the usual unipotent flow, and $\nu=m_{Y}$). First, \cite{ratner1987rigid} shows that $(X,\phi^{U_{X},\tau_{X}}_{t})$ has H-property.    In particular, suppose that $\rho\in J(\phi^{U_{X},\tau_{X}}_{t},\phi^{U_{Y}}_{t})$ is not the product measure $\mu\times\nu$. Then H-property of $\tilde{u}_{X}^{t}\coloneqq\phi^{U_{X},\tau_{X}}_{t}$ deduces that $\rho$ is a finite extension of $\nu$ (see Theorem 3, \cite{ratner1983horocycle}):
   \[\int f(x,y)d\rho(x,y)=\int\frac{1}{n} \sum_{p=1}^{n}f(\psi_{p}(y),y) d\nu(y).\]
On the other hand, since $V^{\perp}_{C_{X}}=0$, by Corollary \ref{joinings202105.7} (and (\ref{time change188})), we again have a map $ \widetilde{S}_{c}: \supp (\rho)\rightarrow\supp (\rho)$ given by
\begin{equation}\label{joinings202105.25}
\widetilde{S}_{c}:(x,y) \mapsto ( u_{X}^{\alpha(c,y)}x,cy)
\end{equation}
In contrast to Theorem \ref{joinings202105.41}, $\widetilde{S}_{c}$
is $\rho$-invariant in this situation. We can further specify $\alpha(c,x,y)$ in certain situation as follows:

First, under the current setting, (\ref{joinings202105.10})  changes to
\[\xi(\psi_{p}(cy),t)+\alpha(c,y)= \alpha(c,u_{Y}^{t}y)+\xi(\psi_{p}(y),t)\]
for $t\in\mathbf{R}$. It follows that
\begin{align}
 0=&\int_{0}^{\xi(\psi_{p}(y),t)}\tau(u_{X}^{s}\psi_{p}(y))-\tau(u_{X}^{s}\psi_{p}(y))ds\;\nonumber\\
 =&\int_{0}^{\xi(\psi_{p}(cy),t)}\tau(u_{X}^{s}\psi_{p}(cy))ds-\int_{0}^{\xi(\psi_{p}(y),t)}\tau(u_{X}^{s}\psi_{p}(y))ds\;\nonumber\\
=&\int_{0}^{\xi(\psi_{p}(cy),t)}\tau(u_{X}^{\alpha(c,y)+s}\psi_{p}(y))ds-\int_{0}^{\xi(\psi_{p}(y),t)}\tau(u_{X}^{s}\psi_{p}(y))ds\;\nonumber\\
=&\int_{0}^{\alpha(c,y)+\xi(\psi_{p}(cy),t)}\tau(u_{X}^{s}\psi_{p}(y))ds-\int_{0}^{\alpha(c,y)}\tau(u_{X}^{s}\psi_{p}(y))ds-\int_{0}^{\xi(\psi_{p}(y),t)}\tau(u_{X}^{s}\psi_{p}(y))ds\;\nonumber\\
=&\int_{0}^{\alpha(c,u_{Y}^{t}y)+\xi(\psi_{p}(y),t)}\tau(u_{X}^{s}\psi_{p}(y))ds-\int_{0}^{\xi(\psi_{p}(y),t)}\tau(u_{X}^{s}\psi_{p}(y))ds-\int_{0}^{\alpha(c,y)}\tau(u_{X}^{s}\psi_{p}(y))ds\;\nonumber\\
=&\int_{0}^{\alpha(c,u_{Y}^{t}y)}\tau(u_{X}^{s}\tilde{u}_{X}^{t}(\psi_{p}(y)))ds-\int_{0}^{\alpha(c,y)}\tau(u_{X}^{s}\psi_{p}(y))ds. \;  \nonumber
\end{align}
In other words, we have
\[\int_{0}^{\alpha(c, u_{Y}^{t}y)}\tau(u_{X}^{s}\tilde{u}_{X}^{t}(x))ds=\int_{0}^{\alpha(c,y)}\tau(u_{X}^{s}x)ds\]
for $\rho$-a.e. $(x,y)\in X\times Y$ and therefore
\[\int_{0}^{\alpha(c,y)}\tau(u_{X}^{s}x)ds\equiv r_{c}\]
for some $r_{c}\in\mathbf{R}$. It follows that
\begin{equation}\label{joinings202106.148}
  \alpha(c,y)=\xi(x,r_{c})
\end{equation}
for $\rho$-a.e. $(x,y)\in X\times Y$.
 Moreover, we apply $\tilde{u}_{X}^{-r_{c}}\times u_{Y}^{-r_{c}}$ to
(\ref{joinings202105.25}), and  get that
\begin{equation}\label{joinings202105.44}
  (x,y) \mapsto ( u_{X}^{\alpha(c,y)}x,cy)\mapsto ( x,u_{Y}^{-r_{c}}cy)
\end{equation}
  is $\rho$-invariant. In particular, suppose that $G_{Y}$ is a semisimple Lie group  with finite center and no compact factors  and   $\Gamma_{Y}\subset G_{Y}$ is a irreducible lattice.
  If the $\mathfrak{sl}_{2}$-weight decomposition $\mathfrak{g}_{Y}=\mathfrak{sl}_{2}+V^{\perp}$ of $\mathfrak{g}_{Y}$ (see (\ref{joinings202103.2})) contains at least one $\mathfrak{sl}_{2}$-irreducible representation $V_{\varsigma}\subset V^{\perp}$ with a  positive highest weight $\varsigma>0$. Choosing $c=\exp(v_{\varsigma})$,  by  \textit{Moore’s ergodicity theorem}\index{Moore’s ergodicity theorem}, we must have $\rho=\mu\times\nu$ (cf. Lemma \ref{joinings202012.3}). Note that this coincides with the result obtained in \cite{dong2020rigidity}. Besides, even if    the highest weight of $V_{\varsigma}$ is $\varsigma=0$ for any $V_{\varsigma}\subset V^{\perp}$, the only possible situation for $\rho\neq \mu\times\nu$ is that $\alpha(\exp v,y)\equiv 0$ for all $v\in V^{\perp}$. Thus, by (\ref{joinings202105.44}), we conclude that $\rho$ is $(\id\times\exp(v))$-invariant for any $v\in V^{\perp}$. In Section \ref{joinings202106.172}, we shall see that $\langle \exp(v)\rangle\subset G_{Y}$ is a normal subgroup, which leads to a contradiction. Thus, we conclude that $V^{\perp}=0$ and so $\mathfrak{g}_{Y}=\mathfrak{sl}_{2}$.
\end{rem}

\subsection{Normal direction}\label{joinings202106.167}
Applying a similar argument in Section \ref{joinings202105.24}, we can  study the behavior of $\overline{\psi}_{p}$ along the normal direction $N_{G_{Y}}(U_{Y})$ of $U_{Y}$ as well. Here we only study the diagonal action provided by the  $\mathfrak{sl}_{2}$-triple. Thus,  let
 \[\Span\{U_{Y},A_{Y},\overline{U}_{Y}\}\subset\mathfrak{g}_{Y},\ \ \  \Span\{U_{X},Y_{n},\overline{U}_{X}\}\subset\mathfrak{g}_{X}\]
 be  $\mathfrak{sl}_{2}$-triples in $\mathfrak{g}_{Y}$, $\mathfrak{g}_{X}$ respectively, where $Y_{n}$ is given in Section \ref{joinings202105.23}. Denote
 \[a^{t}_{Y}\coloneqq\exp(tA_{Y}),\ \ \ a^{t}_{X}\coloneqq\exp(tY_{n}).\]

 We adopt the same notation and orderly fix the data as in Section  \ref{joinings202105.24}; thus, $\sigma,\epsilon,t_{K_{1}},\delta,K,K^{0}$ are chosen so that (\ref{joinings202104.24}) (\ref{joinings202105.8}) (\ref{joinings202105.33}) hold. (Here we further assume $\delta<\epsilon$.)
Fix   $|t_{0}|<\delta$, $a_{Y}=a_{Y}^{t_{0}}$ and $a_{X}=a_{X}^{t_{0}}$.  By ergodic theorem, there is $A_{a_{Y},y}\subset\mathbf{R}^{+}$ and $\lambda_{0}>0$  such that
    \begin{itemize}
      \item for $r\in A_{a_{Y},y}$, we have
     \[ u^{r}_{Y}y ,a_{Y}u^{r}_{Y}y\in K^{0};\]
     \item
     $\Leb(A_{a_{Y},y}\cap[\lambda^{\prime},\lambda^{\prime\prime}])\geq (1-2\sigma)(\lambda^{\prime\prime}-\lambda^{\prime})$ whenever $\lambda^{\prime\prime}-\lambda^{\prime}\geq\lambda_{0}$ and $\lambda^{\prime}\in A_{a_{Y},y}$.
    \end{itemize}
 Then by the assumptions, we have
 \begin{equation}\label{joinings202105.32} A_{a_{Y},y} \subset \left\{r\in[0,\infty):d_{\overline{X}}(a_{X}\overline{\psi}_{p}(u^{r}_{Y}y),\overline{\psi}_{p}(a_{Y}u^{r}_{Y}y))<2\epsilon,\ p\in\{1,\ldots,n\}\right\}.
 \end{equation}
It follows that for $r\in A_{a_{Y},y}$, we have
 \begin{align}
 2\epsilon>&d_{\overline{X}}(a_{X}\overline{\psi}_{p}(u^{r}_{Y}y),\overline{\psi}_{p}(a_{Y}u^{r}_{Y}y))\;\nonumber\\
 =&d_{\overline{X}}(a_{X}\overline{\psi}_{p}(u^{r}_{Y}y),\overline{\psi}_{p}(u^{e^{-t_{0}}r}_{Y}a_{Y}y))\;\nonumber\\
  =&d_{\overline{X}}\left(a_{X}u_{X}^{z(y,r)}\overline{\psi}_{i_{p}(y,r)}(y),u_{X}^{z(a_{Y}y ,e^{-t_{0}}r)}\overline{\psi}_{i_{p}(a_{Y}y,e^{-t_{0}}r)}(a_{Y}y)\right)\;\nonumber\\
 =&d_{\overline{X}}\left(u_{X}^{e^{-t_{0}}z(y,r)}a_{X}\overline{\psi}_{i_{p}(y,r)}(y),u_{X}^{z(a_{Y}y ,e^{-t_{0}}r)}\overline{\psi}_{i_{p}(a_{Y}y,e^{-t_{0}}r)}(a_{Y}y)\right) \;   \nonumber
\end{align}
     for any $p\in\{1,\ldots,n\}$ (cf. (\ref{joinings202105.34})).

     Assume that $0\in A_{a_{Y},y}$. Let $I=((p_{1},p_{2}),\ldots, (p_{2n-1},p_{2n}))\in\{1,\ldots,n\}^{2n}$ be a sequence of  indexes and
   \[  A_{a_{Y},y}^{I}\coloneqq\{r\in A_{a_{Y},y}: p_{2k-1}=i_{k}(y,r),\ p_{2k}=i_{k}(a_{Y}y,e^{-t_{0}}r)\text{ for all }k\in\{1,\ldots,n\}\}.\]
  Then  $A= A_{a_{Y},y}^{I}$, $R_{0}=t_{K_{1}}$, $s(r)=e^{-t_{0}}z(y,r)$, $t(r)=z(a_{Y}y ,e^{-t_{0}}r)$ satisfy (\ref{joinings202104.30}) (\ref{joinings202104.5}) for  points
  \[a_{X}\overline{\psi}_{p_{2k-1}}(y)\in \overline{X},\ \ \ \overline{\psi}_{p_{2k}}(a_{Y}y)\in K\]
  for all $k\in\{1,\ldots,n\}$. We can then apply Proposition \ref{joinings202104.57} to   $A_{a_{Y},y}=\coprod_{I\in\{1,\ldots,n\}^{2n}}A^{I}_{a_{Y},y}$    for any $\lambda\geq \lambda_{0}$. Then we follow the same argument as in Proposition \ref{joinings202104.58} (see also Corollary \ref{joinings202105.7}), and obtain
  \begin{prop}
      There is a measurable map $\varpi:\exp(\mathbf{R}A_{Y})\times X\times Y\rightarrow C_{G_{X}}(U_{X})$ that induces a map  $\widetilde{S}_{a^{r}_{Y}}:\supp (\rho)\rightarrow\supp (\rho)$ by
      \begin{equation}\label{joinings202105.35}
     \widetilde{S}_{a^{r}_{Y}}: (x,y) \mapsto (\varpi(a^{r}_{Y},x,y)a^{r}_{X}x,a^{r}_{Y}y)
      \end{equation}
   for all $r\in\mathbf{R}$,   $\rho$-a.e. $(x,y)\in X\times Y$. Moreover, we have
    \begin{align}
\varpi(a_{Y}^{r},x,y)= \ &  u_{X}^{-z(a_{Y}y,t)}\varpi(a_{Y}^{r},(u_{X}^{z(y,e^{r}t)}\times \widetilde{u}_{Y}^{z(y,e^{r}t)}).(x,y))u_{X}^{e^{-r}z(y,e^{r}t)}\;\label{joinings202105.36}\\
 \varpi(a_{Y}^{r_{1}+r_{2}},x,y) =\ &  \varpi(a_{Y}^{r_{1}},\varpi(a_{Y}^{r_{2}},x,y)a_{X}^{r_{2}}x,a_{Y}^{r_{2}}y)a_{X}^{r_{1}}\varpi(a_{Y}^{r_{2}},x,y)a_{X}^{-r_{1}}\; \label{joinings202105.37}
\end{align}
   for $r,r_{1},r_{2}\in\mathbf{R}$,   $\rho$-a.e. $(x,y)\in X\times Y$, $t\in\mathbf{R}$.
  \end{prop}

Similar to the discussion after Corollary \ref{joinings202105.7}, we consider the  decomposition (\ref{time change184}) and write
\begin{equation}\label{joinings202105.38}
  \varpi(a_{Y}^{r},x,y)=u_{X}^{\alpha(a_{Y}^{r},x,y)} \beta(a_{Y}^{r},x,y)
\end{equation}
where $\alpha(a_{Y}^{r},x,y)\in\mathbf{R}$ and $\beta(a_{Y}^{r},x,y)\in\exp V^{\perp}_{C_{X}}$. Then by (\ref{joinings202105.36}), we have
 \begin{align}
z(a^{r}_{Y}y,t)+\alpha(a^{r}_{Y},x,y)=\ & \alpha(a^{r}_{Y},(u_{X}^{z(y,e^{r}t)}\times \widetilde{u}_{Y}^{z(y,e^{r}t)}).(x,y))+e^{-r}z(y,e^{r}t),\;\label{joinings202105.39}\\
 \beta(a^{r}_{Y},x,y)\equiv\ & \beta(a^{r}_{Y},(u_{X}^{z(y,e^{r}t)}\times \widetilde{u}_{Y}^{z(y,e^{r}t)}).(x,y))\; \label{joinings202105.40}
\end{align}
for all $r,t\in\mathbf{R}$. The same argument then shows that
\[\alpha(a^{r}_{Y},x,y)\equiv\alpha(a^{r}_{Y},y),\ \ \  \beta(a^{r}_{Y},x,y)\equiv\beta(a_{Y}^{r})\]
for all $r\in\mathbf{R}$, $\rho$-a.e. $(x,y)\in X\times Y$.
Besides, following the same lines as in Theorem \ref{joinings202105.41}, we obtain Theorem \ref{joinings202106.159}:
\begin{thm}[Extra normal invariance  of $\rho$]\label{joinings202106.149}
  For any $a_{Y}\in \exp(\mathbf{R}A_{Y})$, the map $S_{a_{Y}}:X\times Y\rightarrow X\times Y$ defined by
  \[S_{a_{Y}}:(x,y)\mapsto \left(\beta(a_{Y})a_{X}x,\tilde{u}_{Y}^{-\alpha(a_{Y},y)}(a_{Y}y)\right)\]
  satisfies
  \[ S_{a^{r}_{Y}}\circ(u_{X}^{t}\times \widetilde{u}_{Y}^{t})=(u_{X}^{e^{-r}t}\times \widetilde{u}_{Y}^{e^{-r}t})\circ S_{a^{r}_{Y}}\]
  and is $\rho$-invariant. Besides, $S_{a^{r_{1}+r_{2}}_{Y}}=S_{a^{r_{1}}_{Y}}S_{a^{r_{2}}_{Y}}$ for any $r_{1},r_{2}\in \mathbf{R}$. Also, we have
  \[S_{a_{Y}}\circ S_{c}\circ S_{a_{Y}^{-1}}=S_{a_{Y}ca_{Y}^{-1}}\]
  for any $a_{Y}\in \exp(\mathbf{R}A_{Y})$, $c\in C_{G_{Y}}(U_{Y})$.
\end{thm}

\begin{cor}[Extra normal invariance of $\nu$]
   For any $a_{Y}\in \exp(\mathbf{R}A_{Y})$, the map $S^{Y}_{a_{Y}}:Y\rightarrow Y$ defined by
  \[S_{a_{Y}}^{Y}:y\mapsto \tilde{u}_{Y}^{-\alpha(a_{Y},y)}(a_{Y}y)\]
 satisfies
   \[S^{Y}_{a^{r}_{Y}}\circ  \widetilde{u}_{Y}^{t} =  \widetilde{u}_{Y}^{e^{-r}t} \circ S^{Y}_{a^{r}_{Y}}\]
    and is $\nu$-invariant. Besides, $S^{Y}_{a^{r_{1}+r_{2}}_{Y}}=S^{Y}_{a^{r_{1}}_{Y}}S^{Y}_{a^{r_{2}}_{Y}}$ for any $r_{1},r_{2}\in \mathbf{R}$. Also, we have
  \[S^{Y}_{a_{Y}}\circ S^{Y}_{c}\circ S^{Y}_{a_{Y}^{-1}}=S^{Y}_{a_{Y}ca_{Y}^{-1}}\]
  for any $a_{Y}\in \exp(\mathbf{R}A_{Y})$, $c\in C_{G_{Y}}(U_{Y})$.
\end{cor}

\begin{thm} Let  $\tau_{Y}\in \mathbf{K}_{\kappa}(Y)$. Suppose that there is an ergodic joining  $\rho\in J(u_{X}^{t},\phi_{t}^{U_{Y},\tau_{Y}})$. Then
            $\tau_{Y}(y)$ and $\tau_{Y}(a_{Y}y)$ are (measurably)  cohomologous along $u_{Y}^{t}$ for all $a^{r}_{Y}\in \exp(\mathbf{R}A_{Y})$. More precisely, the transfer function can be taken to be
            \[F_{a^{r}_{Y}}(y)= e^{r} \alpha(a_{Y}^{r},y).\]
\end{thm}
\begin{proof}
  By (\ref{joinings202105.10}), for $m_{Y}$-a.e. $y\in Y$, $x\in\supp(\rho_{y})$, we have
  \begin{align}
 &e^{-r}\int_{0}^{e^{r}t}\tau(u_{Y}^{s}y)-\tau(a^{r}_{Y}u_{Y}^{s}y)ds\;\nonumber\\
 =&e^{-r}\int_{0}^{e^{r}t}\tau(u_{Y}^{s}y)ds-\int_{0}^{t}\tau(u_{Y}^{s}a_{Y}y)ds\;\nonumber\\
=&  e^{-r}z(y,e^{r}t)-z(a^{r}_{Y}y,t)\;\nonumber\\
=&  \alpha(a^{r}_{Y},y) -  \alpha(a^{r}_{Y},\widetilde{u}^{z(y,e^{r}t)}(y))\;\nonumber\\
=&\alpha(a_{Y}^{r},y)- \alpha(a_{Y}^{r},u_{Y}^{e^{r}t}y). \;   \nonumber
\end{align}
Thus, we can take the transfer function as
\[F_{a^{r}_{Y}}(y)\coloneqq e^{r}  \alpha(a_{Y}^{r}, y) .\]
  Then $\tau(y)$ and $\tau(a_{Y}y)$ are (measurably) cohomologous for all $a_{Y}\in\exp(\mathbf{R}A_{Y})$.
\end{proof}

\subsection{Opposite unipotent direction} Now we shall study  the opposite unipotent direction $\overline{u}_{Y}^{r}=\exp(r\overline{U}_{Y})$, $\overline{u}_{X}^{r}=\exp(r\overline{U}_{X})$. Unlike previous sections, we cannot directly obtain $\rho$ is invariant under the opposite unipotent direction. However, we compensate it by making the ``$a$-adjustment". More precisely, by choosing appropriate coefficients $\lambda_{k}>0$, set
\[\Psi_{k,p}(y)\coloneqq a^{\lambda_{k}}_{X}\overline{\psi}_{p} (a^{-\lambda_{k}}_{Y}y)\]
for a.e. $y\in Y$. Then we shall show that (see Theorem \ref{joinings202106.122})
\[\lim_{n\rightarrow\infty}d_{\overline{X}}(\Psi_{k,p}(\overline{u}_{Y}^{r}y), \overline{u}_{X}^{r}\Psi_{k,p}(y))=0.\]
 Here we adopt the argument given by Ratner \cite{ratner1987rigid} and make a slight generalization.
  It is again convenient to consider  $u,a,\overline{u}\in SL(2,\mathbf{R})$ as $(2\times 2)$-matrices. We first introduce a basic lemma by Ratner that estimates the time-difference of the  $\phi^{U_{Y},\tau}_{t}$-flow under  the $\overline{u}_{Y}^{r}$-direction.

First of all, one directly calculates
\begin{multline}
u^{t}_{Y}\overline{u}_{Y}^{r}  =\left[
            \begin{array}{ccc}
              1 &   0\\
              t&   1\\
            \end{array}
          \right]\left[
            \begin{array}{ccc}
              1 &  r\\
            0&   1\\
            \end{array}
          \right]=\left[
            \begin{array}{ccc}
              1 &   r\\
              t&   1+rt\\
            \end{array}
          \right]   \\
 =\left[
            \begin{array}{ccc}
              1 &  \frac{r}{1+rt}\\
            0&   1\\
            \end{array}
          \right]\left[
            \begin{array}{ccc}
             \frac{1}{1+rt} &   \\
               &    1+rt \\
            \end{array}
          \right]\left[
            \begin{array}{ccc}
              1 &   0\\
              \frac{t}{1+rt}&   1\\
            \end{array}
          \right]=\overline{u}_{Y}^{\frac{r}{1+rt}}a_{Y}^{-2\log(1+rt)}u_{Y}^{\frac{t}{1+rt}}.\label{joinings202105.67}
\end{multline}
We are interested in the \textit{fastest relative motion} of $u_{Y}^{t}$-shearing
\begin{equation}\label{joinings202106.100}
 \Delta_{r}(t)\coloneqq t-\frac{t}{1+rt}\ \ \  \text{ and }\ \ \ \Delta^{\tau_{Y}}_{r}(y,t)\coloneqq\int_{0}^{t}\tau_{Y}(u_{Y}^{s}\overline{u}^{r}_{Y}y)ds-\int_{0}^{\frac{t}{1+rt}}\tau_{Y}(u_{Y}^{s}y)ds.
\end{equation}
\begin{lem}[\cite{ratner1987rigid} Lemma 1.2]\label{joinings202105.65}
  Assume $\tau_{Y}\in C^{1}(Y)$. Then given sufficiently small $\epsilon>0$, there are
   \begin{itemize}
     \item  $\delta=\delta(\epsilon)\approx 0$,
     \item $l=l(\epsilon)>0$,
     \item $E=E(\epsilon)\subset Y$ with $\mu(E)>1-\epsilon$
   \end{itemize}    such that if $y,\overline{u}^{r}_{Y}y\in E$ for some $|r|\leq \delta/l$ then
  \begin{equation}\label{joinings202106.104}
    |\Delta^{\tau_{Y}}_{r}(y,t)-\Delta_{r}(t)|\leq O(\epsilon)|\Delta_{r}(t)|
  \end{equation}
  for all $t\in[l,\delta |r|^{-1}]$.
\end{lem}
\begin{proof}
   Denote
   \[\tau_{a}(y)=\lim_{t\rightarrow0}\frac{\tau_{Y}(a_{Y}^{t}y)-\tau_{Y}(y)}{t},\ \ \ \tau_{\overline{u}}(y)=\lim_{t\rightarrow0}\frac{\tau_{Y}(\overline{u}^{t}_{Y}y)-\tau_{Y}(y)}{t}.\]
   The function $\tau_{g},\tau_{k}$ are continuous on $Y$ and
   \begin{equation}\label{joinings202105.66}
    |\tau_{Y}(y)|,|\tau_{a}(y)|,|\tau_{\overline{u}}(y)|\leq \|\tau_{Y}\|_{C^{1}(Y)}
   \end{equation}
   for  all $y\in Y$. Besides, we have
   \[\int_{Y}\tau_{a}(y)dm_{Y}(y)=\int_{Y}\tau_{\overline{u}}(y)dm_{Y}(y)=0.\]

Given $\epsilon>0$, we fix the data as follows:
\begin{itemize}
  \item   Let $K\subset Y$ be an open subset of $Y$ such that $\overline{K}$ is compact and
   \[m_{Y}(K)>1-\epsilon,\ \ \ m_{Y}(\partial K)=0\]
   where $\partial K$ denotes the boundary of $K$.
\item   Fix a sufficiently small $\delta^{\prime}=\delta^{\prime}(\epsilon)\approx 0$ such that
\begin{enumerate}
\item  $\mu(B(\partial K,\delta^{\prime}))\leq\epsilon$  where $B(\partial K,\delta^{\prime})$ denotes the $\delta^{\prime}$-neighborhood of $\partial K$ (It follows that $\mu(K\setminus B(\partial K,\delta^{\prime}))\geq 1-2\epsilon$);
\item if $y_{1},y_{2}\in\overline{K}$, $d_{Y}(y_{1},y_{2})\leq \delta^{\prime}$ then
   \begin{equation}\label{joinings202105.61}
    |\tau_{a}(y_{1})-\tau_{a}(y_{2})|\leq\epsilon.
   \end{equation}
    \end{enumerate}
\item   Fix $\delta\in(0,\frac{1}{100} \delta^{\prime})$ such that if $|rt|\leq \delta$ then  for all $s\in[0,t]$
   \begin{equation}\label{joinings202105.62}
      \left|\epsilon_{1,t}(s)\right|\leq\epsilon,\ \ \ \text{ where }\ \ \epsilon_{1,t}(s)\coloneqq\frac{\frac{1}{(1+rs)^{2}}-1}{\frac{1}{t} \Delta_{r}(t)}-\frac{2s}{t}.
   \end{equation}
\item   Fix  $t_{1}=t_{1}(\epsilon)>0$ and a subset $E=E(\epsilon)\subset Y$ with $m_{Y}(E)>1-\epsilon$ such that if $y\in E$, $t\in[ t_{1},\infty)$, then the relative length measure of $K\setminus B(\partial K,\delta^{\prime})$ on the orbit interval $[y,u_{Y}^{t}y]$ is at least $1-3\epsilon$ and $\left|\epsilon_{2}(t)\right|\leq\epsilon$, $\left|\epsilon_{3}(t)\right|\leq\epsilon$, where
   \begin{equation}\label{joinings202105.63}
  \epsilon_{2}(t)\coloneqq  \frac{1}{t}\int_{0}^{t}\tau_{Y}(u_{Y}^{s}y)ds-1,\ \ \  \epsilon_{3}(t)\coloneqq  \frac{1}{t}\int_{0}^{t}\tau_{a}(u_{Y}^{s}y)ds.
   \end{equation}
\item Fix $l=l(\epsilon)>t_{1}$ such that
   \begin{equation}\label{joinings202105.64}
     t_{1}/l\leq\epsilon.
   \end{equation}
\end{itemize}
We shall show that if $y,\overline{u}_{Y}^{r}y\in E$ for some $|r|\leq\delta/l$, and $ t\in[l,\delta|r|^{-1}]$ then (\ref{joinings202106.104}) holds if $\epsilon$ is sufficiently small.

 Now let us estimate   $\Delta^{\tau_{Y}}_{r}(y,t)$. Recall that
\[\Delta^{\tau_{Y}}_{r}(y,t)=\int_{0}^{t}\tau_{Y}(u_{Y}^{s}\overline{u}^{r}_{Y}y)ds-\int_{0}^{\frac{t}{1+rt}}\tau_{Y}(u_{Y}^{s}y)ds.\]
Then by (\ref{joinings202105.67}) and the mean value theorem, we have
\begin{align}
\int_{0}^{\frac{t}{1+rt}}\tau_{Y}(u_{Y}^{s}y)ds=& \int_{0}^{t}\tau_{Y}(u_{Y}^{\frac{s}{1+rs}}y)\cdot\frac{ds}{(1+rs)^{2}} \;\nonumber\\
=&  \int_{0}^{t}\tau_{Y}(a_{Y}^{2\log(1+rs)}\overline{u}_{Y}^{-\frac{r}{1+rs}}u^{s}_{Y}\overline{u}^{r}_{Y}y)\cdot\frac{ds}{(1+rs)^{2}} \;\nonumber\\
=&  \int_{0}^{t}\tau_{Y}(u^{s}_{Y}\overline{u}^{r}_{Y}y)\cdot\frac{ds}{(1+rs)^{2}}\;\nonumber\\
 &  -\int_{0}^{t}\frac{r}{1+rs}\tau_{\overline{u}}(\overline{u}_{Y}^{k_{s}}u^{s}_{Y}\overline{u}^{r}_{Y}y)  \cdot\frac{ds}{(1+rs)^{2}}\;\nonumber\\
 & +\int_{0}^{t}2\log(1+rs) \tau_{a}(a_{Y}^{g_{s}}\overline{u}_{Y}^{-\frac{r}{1+rs}}u^{s}_{Y}\overline{u}^{r}_{Y}y) \cdot\frac{ds}{(1+rs)^{2}} \;  \nonumber
\end{align}
where  $k_{s}\in\left[-\frac{r}{1+rs},0\right]$ and   $g_{s}\in[0,2\log(1+rs)]$. This implies
\begin{align}
\Delta^{\tau_{Y}}_{r}(y,t)=& \int_{0}^{t}\tau_{Y}(u_{Y}^{s}\overline{u}^{r}_{Y}y) \left(1- \frac{1}{(1+rs)^{2}}\right)ds\;\nonumber\\
 &  +\int_{0}^{t}\frac{r}{1+rs}\tau_{\overline{u}}(\overline{u}_{Y}^{k_{s}}u^{s}_{Y}\overline{u}^{r}_{Y}y)  \cdot\frac{ds}{(1+rs)^{2}}\;\nonumber\\
  & -\int_{0}^{t}2\log(1+rs) \tau_{a}(a_{Y}^{g_{s}}\overline{u}_{Y}^{-\frac{r}{1+rs}}u^{s}_{Y}\overline{u}^{r}_{Y}y) \cdot\frac{ds}{(1+rs)^{2}} \;\nonumber\\
= & J_{1}+ J_{2}+J_{3}. \;  \nonumber
\end{align}
We estimate the integrals $J_{1},J_{2},J_{3}$ separately:
\begin{enumerate}[\ \ \ (1)]
  \item  Using  (\ref{joinings202105.62}) (\ref{joinings202105.63}), we have
\begin{align}
  J_{1}=& 2\Delta_{r}(t)\frac{1}{t^{2}}\int_{0}^{t}s\tau_{Y}(u_{Y}^{s}\overline{u}^{r}_{Y}y) ds+ \Delta_{r}(t)\frac{1}{t}\int_{0}^{t}\epsilon_{1,t}(s)\tau_{Y}(u_{Y}^{s}\overline{u}^{r}_{Y}y) ds\;\nonumber\\
  =&  2\Delta_{r}(t)\frac{1}{t^{2}}\int_{0}^{t}s\tau_{Y}(u_{Y}^{s}\overline{u}^{r}_{Y}y) ds +\Delta_{r}(t) O(\epsilon)  \;  \nonumber
\end{align}
since $\overline{u}^{r}_{Y}y\in E$.  Now by the integration by parts and (\ref{joinings202105.63}), (\ref{joinings202105.66}) (\ref{joinings202105.64}), we have
\begin{align}
 &  \frac{1}{t^{2}}\int_{0}^{t}s\tau_{Y}(u_{Y}^{s}\overline{u}^{r}_{Y}y) ds \;\nonumber\\
=&   \frac{1}{t} \int_{0}^{t}\tau_{Y}(u_{Y}^{s}\overline{u}^{r}_{Y}y) ds-\frac{1}{t^{2}} \int_{0}^{t}\left(\int_{0}^{s}\tau_{Y}(u_{Y}^{p}\overline{u}^{r}_{Y}y) dp\right) ds \;\nonumber\\
=& 1+\epsilon_{2}(t)-\frac{1}{t^{2}}\left[\int_{t_{1}}^{t} +\int_{0}^{t_{1}}\right]\left(\int_{0}^{s}\tau_{Y}(u_{Y}^{p}\overline{u}^{r}_{Y}y) dp\right) ds \;\nonumber\\
=&  1+\epsilon_{2}(t)-\frac{1}{t^{2}}\int_{t_{1}}^{t}s\left( 1+\epsilon_{2}(s)\right)ds +O(\epsilon)=  \frac{1}{2} +O(\epsilon).\;  \nonumber
\end{align}
 It follows that
\[\left|\frac{J_{1}}{\Delta_{r}(t)}-1\right|\leq O(\epsilon).\]
\item For $J_{2}$, by (\ref{joinings202105.64}), we have
\[|J_{2}|= \left|\int_{0}^{t}\frac{r}{1+rs}\tau_{\overline{u}}(\overline{u}_{Y}^{k_{s}}u^{s}_{Y}\overline{u}^{r}_{Y}y)  \cdot\frac{ds}{(1+rs)^{2}}\right|\leq O\left(\frac{|\Delta_{r}(t)|}{t}\right)\leq O(\epsilon)|\Delta_{r}(t)|.\]
\item Note that since   $d_{Y}(a_{Y}^{g_{s}}\overline{u}_{Y}^{-\frac{r}{1+rs}}u^{s}_{Y}\overline{u}^{r}_{Y}y ,u^{s}_{Y}\overline{u}^{r}_{Y}y)<\delta^{\prime}$,  we know $a_{Y}^{g_{s}}\overline{u}_{Y}^{-\frac{r}{1+rs}}u^{s}_{Y}\overline{u}^{r}_{Y}y \in \overline{K}$ if $u^{s}_{Y}\overline{u}^{r}_{Y}y\in K\setminus B(\partial K,\delta^{\prime})$. Now set
\[I_{y}\coloneqq\{s\in[0,t]:u^{s}_{Y}\overline{u}^{r}_{Y}y\in K\setminus B(\partial K,\delta^{\prime})\}.\]
Then by (\ref{joinings202105.63}), one has $\Leb(I_{y}^{c})<3\epsilon t$. Then for $J_{3}$, using (\ref{joinings202105.66}) and (\ref{joinings202105.61}),  we have
\begin{align}
 &  \left|J_{3}- \left(-\int_{0}^{t}2\log(1+rs) \tau_{a}(u^{s}_{Y}\overline{u}^{r}_{Y}y) \cdot\frac{ds}{(1+rs)^{2}}\right)\right|\;\nonumber\\
\ll &  \left| \log(1+rt)\right|\left[\int_{I_{y}}\left|\tau_{a}(a_{Y}^{g_{s}}\overline{u}_{Y}^{-\frac{r}{1+rs}}u^{s}_{Y}\overline{u}^{r}_{Y}y) -\tau_{a}(u^{s}_{Y}\overline{u}^{r}_{Y}y) \right|ds+ \epsilon t\|\tau_{Y}\|_{C^{1}(Y)}\right]\;\nonumber\\
\leq & t\left| \log(1+rt)\right|(\epsilon+\epsilon  \|\tau_{Y}\|_{C^{1}(Y)})\ll O(\epsilon)\left| \Delta_{r}(t)\right|.\;  \nonumber
\end{align}
We also have
\begin{align}
 &  \left| \int_{0}^{t}2\log(1+rs) \tau_{a}(u^{s}_{Y}\overline{u}^{r}_{Y}y) \cdot\frac{ds}{(1+rs)^{2}} - \int_{0}^{t}2\log(1+rs) \tau_{a}(u^{s}_{Y}\overline{u}^{r}_{Y}y) ds \right|\;\nonumber\\
  =&  \left| \int_{0}^{t}2\log(1+rs) \tau_{a}(u^{s}_{Y}\overline{u}^{r}_{Y}y) \cdot\left(\frac{1}{(1+rs)^{2}} -1\right) ds\right|\;\nonumber\\
\ll& |\Delta_{r}(t)|\|\tau_{Y}\|_{C^{1}(Y)} \delta\ll O(\epsilon)|\Delta_{r}(t)|. \;  \nonumber
\end{align}
Finally, by using the integration by parts, we get
\begin{align}
 & \left|\int_{0}^{t} \log(1+rs) \tau_{a}(u^{s}_{Y}\overline{u}^{r}_{Y}y) ds\right|\;\nonumber\\
=&  \left|\log(1+rt) \int_{0}^{t} \tau_{a}(u^{s}_{Y}\overline{u}^{r}_{Y}y)ds -\int_{0}^{t}  \left(\int_{0}^{s}\tau_{a}(u^{p}_{Y}\overline{u}^{r}_{Y}y)dp\right)\frac{r}{1+rs} ds\right|\;\nonumber\\
\ll&   \frac{|\Delta_{r}(t)|}{t}\left| \int_{0}^{t} \tau_{a}(u^{s}_{Y}\overline{u}^{r}_{Y}y)ds \right|+ \frac{|\Delta_{r}(t)|}{t^{2}}\left|\int_{0}^{t}  \left(\int_{0}^{s}\tau_{a}(u^{p}_{Y}\overline{u}^{r}_{Y}y)dp\right)ds\right|\;\nonumber\\
=& \epsilon_{3}(t) |\Delta_{r}(t)| +\frac{|\Delta_{r}(t)|}{t^{2}} \left|\left[\int_{0}^{t_{1}}   +\int_{t_{1}}^{t}  \right] \left(\int_{0}^{s}\tau_{a}(u^{p}_{Y}\overline{u}^{r}_{Y}y)dp\right)ds\right|\;\nonumber\\
\ll&  \epsilon_{3}(t) |\Delta_{r}(t)| +\frac{|\Delta_{r}(t)|}{t^{2}}\left|  t_{1}^{2}\|\tau_{Y}\|_{C^{1}(Y)}   +\int_{t_{1}}^{t} s\epsilon_{3}(s) ds \right|\ll O(\epsilon)|\Delta_{r}(t)| .\;  \nonumber
\end{align}
Thus, we conclude that $|J_{3}|\leq O(\epsilon)|\Delta_{r}(t)|$.
\end{enumerate}
Therefore, combining the above estimates, we have
\[|\Delta^{\tau_{Y}}_{r}(y,t)-\Delta_{r}(t)|\leq O(\epsilon)\Delta_{r}(t).\]
This completes the proof of the lemma.
\end{proof}

The following lemma tells us that we only need to know the fastest relative motion at finitely many different time points to determine the difference of two nearby points.

\begin{lem}[Shearing comparison]\label{joinings202106.121} Given $\epsilon>0$, let $x,y,z\in \overline{X}$ be three $\epsilon$-nearby points such that the fastest relative motions between the pairs $(x,z)$ and $(y,z)$ at time $t>0$ are $q_{1}(t)$ and $q_{2}(t)$ respectively. Assume that there are $s_{1},s_{2}>0$ with $s_{1}\in[\frac{1}{3}s_{2},\frac{2}{3}s_{2}]$ such that
\[d_{\overline{X}}(u_{X}^{s_{i}}x,u_{X}^{s_{i}}q_{1}(s_{i})z)<\epsilon,\ \ \ d_{\overline{X}}(u_{X}^{s_{i}}y,u_{X}^{s_{i}}q_{2}(s_{i})z)<\epsilon,\ \ \ d_{G_{X}}(q_{1}(s_{i}),q_{2}(s_{i}))<\epsilon\]
for $i\in\{1,2\}$.
 Then we have
 \begin{equation}\label{joinings202106.119}
  d_{\overline{X}}(u_{X}^{t}x,u_{X}^{t}y)<O(\epsilon)
 \end{equation}
for $t\in[0,s_{2}]$.
\end{lem}
\begin{proof} This is a direct consequence of Lemma \ref{dynamical systems4}.
   Assume that $x=gy$, $x=h_{1}z$, $y=h_{2}z$ for some $g,h_{1},h_{2}\in G_{X}$. Then by the definition (\ref{joinings202106.116}), there are $\delta_{1}(t)$, $\delta_{2}(t)\in G_{X}$ with $d_{G_{X}}(\delta_{1}(t),e)<\epsilon$, $d_{G_{X}}(\delta_{2}(t),e)<\epsilon$ such that
   \[u_{X}^{t}h_{1}u_{X}^{-t}=\delta_{1}(t)q_{1}(t),\ \ \ u_{X}^{t}h_{2}u_{X}^{-t}=\delta_{2}(t)q_{2}(t)\]
    for $t\in[0,s]$. By the assumption, we have
    \begin{equation}\label{joinings202106.118}
      u_{X}^{t}gu_{X}^{-t}=u_{X}^{t}h_{1}h_{2}^{-1}u_{X}^{-t}=\delta_{1}(t)q_{1}(t)q_{2}(t)^{-1}\delta_{2}(t)^{-1}
    \end{equation}
    and
    \begin{equation}\label{joinings202106.117}
      q_{1}(s_{1})q_{2}(s_{1})^{-1}<\epsilon,\ \ \ q_{1}(s_{2})q_{2}(s_{2})^{-1}<\epsilon
    \end{equation}
    Note that $q_{1}(t)q_{2}(t)^{-1}\in C_{G_{X}}(U_{X})$ and so their corresponding vectors in the Lie algebra are polynomials of $t$ with the degree at most $2$ (see (\ref{dynamical systems5})  (\ref{joinings202106.116})); Thus, we can write
     \[h_{1}h_{2}^{-1}=\exp\left(\sum_{j} \sum_{i=0}^{\varsigma(j)}b_{j}^{i}v_{j}^{i}\right),\ \ \ q_{1}(t)q_{2}(t)^{-1}=\exp\left(\sum_{j} p_{j}(t)v_{j}^{\varsigma(j)}\right)\]
     where $p_{j}(t)=\sum^{\varsigma(j)}_{i=0}b_{j}^{\varsigma(j)-i}\binom{\varsigma(j)}{i}t^{i}$ is a polynomial having the degree at most $2$, $|b_{i}|<\epsilon$, $v_{j}^{i}\in V_{j}$ is the $i$-th weight vector of the $\mathfrak{sl}_{2}$-irreducible representation $V_{j}$.
     Then (\ref{joinings202106.117}) and the proof of Lemma \ref{dynamical systems4} (1) with $\kappa=1$ imply that
     \begin{equation}\label{joinings202106.120}
      |b_{j}^{\varsigma(j)-i}|<O(\epsilon)s_{2}^{-i}.
     \end{equation}
     It follows that for $t\in[0,s_{2}]$
     \[|p_{j}(t)|<O(\epsilon)\ \ \  \text{and so}\ \ \ q_{1}(t)q_{2}(t)^{-1}<O(\epsilon).\]
     Then by (\ref{joinings202106.118}), we obtain (\ref{joinings202106.119}).
\end{proof}

     Next, we shall prove Theorem \ref{joinings202106.122}.  The idea is to consider the fastest relative motion of the pairs $(\Psi_{k,p}(\overline{u}_{Y}^{r}y), \Psi_{k,p}(y))$ and $(\overline{u}_{X}^{r}\Psi_{k,p}(y),\Psi_{k,p}(y))$ at finitely many time points. And then apply Lemma \ref{joinings202106.121}. First, we orderly fix the following data:
\begin{itemize}
   \item (Injectivity radius)
Since $\Gamma_{X}$ is discrete, there is a compact $K_{1}\subset \overline{X}$ with $\nu(\overline{\psi}_{p}^{-1}(K_{1}))>\frac{999}{1000}$ and $D_{1}=D_{1}(K_{1})>0$ such that if $\overline{g}\in \overline{P}^{-1}(K_{1})$, then $D_{1}$ is an isometry on the ball $B_{C^{\rho}\backslash G_{X}}(\overline{g},D_{1})$ of radius $D_{1}$ centered at $\overline{g}$. Here $\overline{P}:C^{\rho}\backslash G_{X}\rightarrow C^{\rho}\backslash G_{X}/\Gamma_{X}=\overline{X}$ is the projection
\[\overline{P}:C^{\rho} g\mapsto C^{\rho}  g\Gamma_{X}.\]
\item (Distinguishing $\overline{\psi}_{p},\overline{\psi}_{q}$) There is  $K_{2}\subset Y$ with $\nu(K_{2})>\frac{999}{1000}$ such that
         \begin{equation}\label{joinings202105.49}
          d_{\overline{X}}(\overline{\psi}_{p}(y),\overline{\psi}_{q}(y))>D_{2}
         \end{equation}
          for $y\in K_{2}$, $1\leq p< q\leq n$.
  \item Define $D=\min\{D_{1},D_{2},1\}$.
\item (Lemma \ref{joinings202105.65}) Let $\delta_{k}=\min\left\{\delta\left(\frac{1}{10}2^{-k}D\right), \frac{1}{10}2^{-k}D\right\}$, $l_{k}=l\left(\frac{1}{10}2^{-k}D\right)$ and $E_{k}=E\left(\frac{1}{10}2^{-k}D\right)\subset Y$ be as in Lemma \ref{joinings202105.65} for $\tau_{Y}$.
     \item (Lusin's theorem) There is  $K^{\prime}_{k}\subset Y$ such that $\nu(K^{\prime}_{k})>1-\frac{1}{10}2^{-k}$ and  $\overline{\psi}_{p}|_{K^{\prime}_{k}}$ is uniformly continuous for all $p\in\{1,\ldots,n\}$. Thus,   for any $\epsilon>0$, there is $\delta^{\prime}(\epsilon)>0$ such that  for $p\in\{1,\ldots,n\}$, $d_{Y}(y_{1},y_{2})<\delta^{\prime}(\epsilon)$ and $y_{1},y_{2}\in K^{\prime}_{k}$, we have
         \begin{equation}\label{joinings202105.50}
          d_{\overline{X}}(\overline{\psi}_{p}(y_{1}),\overline{\psi}_{p}(y_{2}))<\epsilon.
         \end{equation}
        Let $\delta^{\prime}_{k}=\min\left\{\delta^{\prime}\left(\frac{1}{10}2^{-k} D\right), \frac{1}{10}2^{-k}D\right\}$.

            \item (Ergodicity) Fix $\tau_{Y}\in C^{1}(Y)$. By the ergodicity of unipotent flows, there are $T_{k}\geq\max\{l_{k},20\delta_{k}^{-1},20\delta_{k}^{\prime-1}\}$ and subsets $K^{\prime\prime}_{k}\subset Y$ with $\nu(K^{\prime\prime}_{k})>1-\frac{1}{10}2^{-k}$ such that if $y\in K^{\prime\prime}_{k}$, $t\geq T_{k}$ then
                \begin{enumerate}[\ \ \ (1)]
                  \item  the relative length measure of $K^{\prime}_{k}\cap E_{k}\cap K_{2} \cap \bigcap_{p}\overline{\psi}_{p}^{-1}(K_{1})$ on the orbit interval $[y,u^{t}_{Y}y]$ is at least $\frac{998}{1000}$;
                  \item we have by the ergodic theorem
\begin{equation}\label{joinings202106.110}
 \left|\frac{1}{t}z(y,t)-1\right|=\left|\frac{1}{t}\int_{0}^{t}\tau_{Y}(u_{Y}^{s}y)ds-1\right|\leq\frac{1}{10}2^{-k}D.
\end{equation}
                \end{enumerate}

    \item (Fastest relative motion)
     \begin{enumerate}[\ \ \ (1)]
      \item  For $r\in\mathbf{R}$, let $L_{1}^{i}(r)$ denote the first $t>0$ with $\Delta_{r}(t)=i^{2}D/10$ for $i\in\{1,2\}$ where $\Delta_{r}(t)$ is defined in (\ref{joinings202106.100}). Note that for sufficiently small $r$, one may calculate that
          \begin{equation}\label{joinings202106.123}
           L_{1}^{1}(r)\in[\frac{9}{20}L_{1}^{2}(r),\frac{11}{20}L_{1}^{2}(r)].
          \end{equation}
      \item As in (\ref{dynamical systems1}), for $\overline{x_{1}},\overline{x_{2}}\in \overline{X}$ close enough,  we can write $\overline{x_{1}}=\overline{gx_{2}}$ where  $g=\exp (v)$ for $v\in \mathfrak{sl}_{2}+V^{\rho\perp}$. Then the H-property (Remark  \ref{joinings202103.4}) tells us that at time $t\in\mathbf{R}$, the \textit{fastest relative motion}\index{fastest relative motion} is given by
           \[ q(\overline{x_{1}},\overline{x_{2}},t)=\pi_{C_{\mathfrak{g}_{\overline{X}}}(U_{X})} \Ad(u_{X}^{t}).v.\]
            Then let $L^{i}_{2}(\overline{x_{1}},\overline{x_{2}})$ denote the first $t>0$ with $\|q(\overline{x_{1}},\overline{x_{2}},t)\|=i^{2}D/10$.
      \end{enumerate}
      For $y\in Y$, $i\in\{1,2\}$, let
      \begin{equation}\label{joinings202106.107}
       L^{i}(y,r)\coloneqq\min\left\{L^{i}_{1}(r),L^{i}_{2}(\overline{\psi}_{1}(\overline{u}^{r}y),\overline{\psi}_{1}(y)),\ldots,,L^{i}_{2}(\overline{\psi}_{n}(\overline{u}^{r}y),\overline{\psi}_{n}(y))\right\}.
      \end{equation}
      By applying Theorem \ref{h-property202103.1} to $Q=B_{C_{G_{Y}}(U_{Y})}(e,i^{2}D/10)$ and $\epsilon=\frac{1}{10}2^{-k}$, we can choose small $0<\omega_{k}\leq \min\{\delta_{k},\delta^{\prime}_{k}\}$ such
        that if $|r|\leq\omega_{k}$, $y,\overline{u}^{r}y\in K^{\prime}_{k}$, $i\in\{1,2\}$, then we have
\begin{equation}\label{joinings202106.101}
 L^{i}=L^{i}(y,r)\geq\max\left\{10T_{k},\frac{10i^{2}D}{\delta^{\prime}_{k}}\right\}
\end{equation}
and for all $p\in\{1,\ldots, p\}$
\begin{equation}\label{joinings202106.106}
  \| q^{i}_{p}\|\leq \frac{i^{2}D}{10},\ \ \  d_{\overline{X}}\left(u_{X}^{L}\overline{\psi}_{p}(\overline{u}^{r}y),u_{X}^{L}\overline{q^{i}_{p}(L)\psi_{p}(y)}\right)\leq\frac{1}{10}2^{-k}D
\end{equation}
where $q^{i}_{p}=q(\overline{\psi}_{p}(\overline{u}^{r}y),\overline{\psi}_{p}(y),L^{i})$.
   \end{itemize}
Now let
 \begin{equation}\label{joinings202106.103}
 K_{k}^{0}\coloneqq K^{\prime}_{k}\cap K^{\prime\prime}_{k}\cap E_{k}.
 \end{equation}
 It follows that $\nu(K_{k}^{0})>1-  2^{-k}$. Let
 \begin{equation}\label{joinings202106.130}
    \lambda_{k}\coloneqq2 \cdot\max\left\{ \log\frac{10}{\omega_{k}},\log T_{k}\right\},\ \ \ \Omega\coloneqq\bigcup_{l\geq 1}\bigcap_{k\geq l}a^{\lambda_{k}}_{Y}(K_{k}^{0}),\ \ \ \Psi_{k,p}(y)\coloneqq a_{X}^{\lambda_{k}}\overline{\psi}_{p}(a_{Y}^{-\lambda_{k}}y).
 \end{equation}
  It follows that $\nu(\Omega)>1$.
  \begin{thm}\label{joinings202106.122}  Let the notation and assumption be as above.    Then  for   $r\in\mathbf{R}$, $y\in \Omega$, we have
 \[\lim_{n\rightarrow\infty}d_{\overline{X}}(\Psi_{k,p}(\overline{u}_{Y}^{r}y), \overline{u}_{X}^{r}\Psi_{k,p}(y))=0.\]
  \end{thm}
\begin{proof}
        Suppose that   $y,\overline{u}_{Y}^{r}y\in \bigcup_{l\geq 1}\bigcap_{k\geq l}a^{\lambda_{k}}_{Y}(K_{k}^{0})$.   Then $y,\overline{u}_{Y}^{r}y\in  a^{\lambda_{k}}_{Y}(K_{k}^{0})$ for sufficiently large $k$.  For $r\in\mathbf{R}$,
        let $r_{k}=e^{-\lambda_{k}}r$.  Then for sufficiently large $k$,
            \[a_{Y}^{-\lambda_{k}}\overline{u}_{Y}^{r}y=\overline{u}_{Y}^{r_{k}}a_{Y}^{-\lambda_{k}}y\ \ \ \text{ and } \ \ \  |r_{k}|\leq |r|\omega_{k}^{2}\leq \omega_{k}.\]
        Thus, (\ref{joinings202106.101}) holds true for $L^{i}(y,r_{k})$ for any  sufficient large $k$, $i\in\{1,2\}$. In the following, we fix $i=1$ (for the case $i=2$ is similar).

Next, since by (\ref{joinings202106.101}) $L^{1}(y,r_{k})>10T_{k}$, there exists $t_{k}\in\left[\frac{98}{100}L^{1}(y,r_{k}),\frac{99}{100}L^{1}(y,r_{k})\right]$ such that
\begin{equation}\label{joinings202106.113}
 u_{Y}^{t_{k}}a_{Y}^{-\lambda_{k}}\overline{u}_{Y}^{r}y,\ u_{Y}^{t^{\prime}_{k}}a_{Y}^{-\lambda_{k}}y\in K^{\prime}_{k}\cap K_{2} \cap \bigcap_{p}\overline{\psi}_{p}^{-1}(K_{1})
\end{equation}
where $t_{k}^{\prime}\coloneqq\frac{t_{k}}{1+r_{k}t_{k}}$. Then by (\ref{joinings202105.67}), we get
 \begin{multline}
 d_{Y}(u_{Y}^{t_{k}}a_{Y}^{-\lambda_{k}}\overline{u}_{Y}^{r}y,u_{Y}^{t^{\prime}_{k}}a_{Y}^{-\lambda_{k}}y)=d_{Y}(u_{Y}^{t_{k}} \overline{u}_{Y}^{r_{k}}a_{Y}^{-\lambda_{k}}y,u_{Y}^{t^{\prime}_{k}}a_{Y}^{-\lambda_{k}}y)\\
 =d_{Y}\left( \left[
            \begin{array}{cc}
           \frac{1}{1+r_{k}t_{k}}  & r_{k} \\
              0 &  1+r_{k}t_{k}\\
            \end{array}
          \right]u_{Y}^{t_{k}^{\prime}}a_{Y}^{-\lambda_{k}}y,u_{Y}^{t_{k}^{\prime}}a_{Y}^{-\lambda_{k}}y\right)\leq\min\{\delta_{k},\delta_{k}^{\prime}\}\label{joinings202106.111}
\end{multline}
where the last inequality follows from (\ref{joinings202106.100})
\begin{equation}\label{joinings202106.132}
 |r_{k}t_{k}|\leq 2\frac{\Delta_{r_{k}}(t_{k})}{t_{k}}\leq 4 \frac{\Delta_{r_{k}}(L^{1}(y,r_{k}))}{T_{k}}\leq 4\frac{D}{10}\cdot \frac{\min\{\delta_{k},\delta_{k}^{\prime}\}}{20}\leq \min\{\delta_{k},\delta_{k}^{\prime}\}.
\end{equation}
This implies via Lemma \ref{joinings202105.65} that
\begin{equation}\label{joinings202106.105}
  |\Delta^{\tau_{Y}}_{r_{k}}(a_{Y}^{-\lambda_{k}}y,t_{k})-\Delta_{r_{k}}(t_{k})|\leq  \frac{1}{10}2^{-k} D
\end{equation}
since $a_{Y}^{-\lambda_{k}}y,\overline{u}_{Y}^{r_{k}}a_{Y}^{-\lambda_{k}}y\in E_{k}$ and $t_{k}\in[T_{k},\delta_{k} |r_{k}|^{-1}]\subset [l_{k},\delta_{k} |r_{k}|^{-1}]$.

 Next, consider
\[  u_{X}^{s_{k}}\overline{\psi}_{p}(a_{Y}^{-\lambda_{k}}\overline{u}_{Y}^{r}y)=\overline{\psi}_{i(p,k)}(u_{Y}^{t_{k}}a_{Y}^{-\lambda_{k}}\overline{u}_{Y}^{r}y),\ \ \ u_{X}^{h^{\prime}_{k}}\overline{\psi}_{p}(a_{Y}^{-\lambda_{k}}y)=\overline{\psi}_{j(p,k)}(u_{Y}^{t^{\prime}_{k}}a_{Y}^{-\lambda_{k}}y)\]
where $s_{k}$ and $h^{\prime}_{k}$ are defined by
\begin{equation}\label{joinings202106.131}
 z(a_{Y,k}^{-1}\overline{u}^{r}y,t_{k})= s_{k},\ \ \   z(a_{Y,k}^{-1}y,t_{k}^{\prime}) = h^{\prime}_{k}.
\end{equation}
Then $\Delta^{\tau_{Y}}_{r_{k}}(a_{Y,k}^{-1}y,t_{k})=s_{k}-h_{k}^{\prime}$ and by (\ref{joinings202106.110}), we have  $s_{k}\in\left[\frac{97}{100}L^{1}(y,r_{k}),\frac{995}{1000}L^{1}(y,r_{k})\right]$.

\begin{cla}\label{joinings202106.127} For $p\in\{1,\ldots,n\}$,
   \[d_{G}(q_{p}(s_{k}),u_{X}^{h^{\prime}_{k}-s_{k}})\leq \frac{2}{10}2^{-k}D \]
   where   $q_{p}(s_{k})=q(\overline{\psi}_{p}(\overline{u}^{r_{k}}a_{Y}^{-\lambda_{k}}y),\overline{\psi}_{p}(a_{Y}^{-\lambda_{k}}y),s_{k})$.
\end{cla}
\begin{proof}
   Since $|r_{k}|\leq\omega_{k}$ and $a^{-\lambda_{k}}_{Y}y,\overline{u}_{Y}^{r_{k}}a^{-\lambda_{k}}_{Y}y\in  K_{k}^{0}$, by (\ref{joinings202106.107}) and Lemma \ref{joinings202105.65}, we know that
   \begin{equation}\label{joinings202106.124}
    |\Delta^{\tau_{Y}}_{r_{k}}(a_{Y,k}^{-1}y,t_{k})|\leq \frac{11}{10}|\Delta_{r_{k}}(t_{k})|\leq\frac{11}{100}D.
   \end{equation}
It follows that
\begin{equation}\label{joinings202106.108}
  d_{\overline{X}}\left(u_{X}^{s_{k}}\overline{\psi}_{p}(a_{Y}^{-\lambda_{k}}y),u_{X}^{h^{\prime}_{k}}\overline{\psi}_{p}(a_{Y}^{-\lambda_{k}}y)\right)<\frac{1}{3}D.
\end{equation}
On the other hand, by
(\ref{joinings202106.106}), we have
\begin{equation}\label{joinings202106.112}
  \| q_{p}(s_{k})\|\leq \frac{D}{10},\ \ \  d_{\overline{X}}\left(u_{X}^{s_{k}}\overline{\psi}_{p}(\overline{u}^{r_{k}}a_{Y}^{-\lambda_{k}}y),u_{X}^{s_{k}}\overline{q_{p}(s_{k})\psi_{p}(a_{Y}^{-\lambda_{k}}y)}\right)\leq\frac{1}{10}2^{-k}D
\end{equation}
 It follows that
\begin{equation}\label{joinings202106.109}
  d_{\overline{X}}\left(u_{X}^{s_{k}}\overline{\psi}_{p}(\overline{u}^{r_{k}}a_{Y}^{-\lambda_{k}}y),u_{X}^{s_{k}}\overline{\psi}_{p}(a_{Y}^{-\lambda_{k}}y)\right)<\frac{1}{3}D
\end{equation}
for $p\in\{1,\ldots,p\}$. Therefore, (\ref{joinings202106.108}) and (\ref{joinings202106.109}) tell us that
\begin{align}
 & d_{\overline{X}}\left(\overline{\psi}_{i(p,k)}(u_{Y}^{t_{k}}a_{Y}^{-\lambda_{k}}\overline{u}_{Y}^{r}y), \overline{\psi}_{j(p,k)}(u_{Y}^{t^{\prime}_{k}}a_{Y}^{-\lambda_{k}}y)\right)\;\nonumber\\
=&  d_{\overline{X}}\left(u_{X}^{s_{k}}\overline{\psi}_{p}(\overline{u}_{Y}^{r_{k}}a_{Y}^{-\lambda_{k}}y),u_{X}^{h^{\prime}_{k}}\overline{\psi}_{p}(a_{Y}^{-\lambda_{k}}y)\right)< D.\;  \nonumber
\end{align}
Then by (\ref{joinings202105.49}), we must have $i(p,k)=j(p,k)$. Then by Lusin theorem (\ref{joinings202105.50}) (\ref{joinings202106.111}), we further obtain
\begin{align}
 &d_{\overline{X}}\left(u_{X}^{s_{k}}\overline{\psi}_{p}(\overline{u}_{Y}^{r_{k}}a_{Y}^{-\lambda_{k}}y),u_{X}^{h^{\prime}_{k}}\overline{\psi}_{p}(a_{Y}^{-\lambda_{k}}y)\right)\;\label{joinings202106.114}\\
=&   d_{\overline{X}}\left(\overline{\psi}_{i(p,k)}(u_{Y}^{t_{k}}a_{Y}^{-\lambda_{k}}\overline{u}_{Y}^{r}y), \overline{\psi}_{i(p,k)}(u_{Y}^{t^{\prime}_{k}}a_{Y}^{-\lambda_{k}}y)\right)\leq \frac{1}{10}2^{-k}D.\;  \nonumber
\end{align}
Combining   (\ref{joinings202106.112}), we get
\begin{align}
 &d_{\overline{X}}\left(\overline{q_{p}(s_{k})\cdot u_{X}^{s_{k}}\psi_{p}(a_{Y}^{-\lambda_{k}}y)},\overline{u_{X}^{h^{\prime}_{k}-s_{k}}\cdot u_{X}^{s_{k}}\psi_{p}(a_{Y}^{-\lambda_{k}}y)}\right)\;\nonumber\\
=&  d_{\overline{X}}\left(u_{X}^{s_{k}}\overline{q_{p}(s_{k})\psi_{p}(a_{Y}^{-\lambda_{k}}y)},u_{X}^{h^{\prime}_{k}}\overline{\psi}_{p}(a_{Y}^{-\lambda_{k}}y)\right)\leq \frac{2}{10}2^{-k}D.\;  \nonumber
\end{align}
Since by (\ref{joinings202106.113}) $u_{X}^{h^{\prime}_{k}}\overline{\psi}_{p}(a_{Y}^{-\lambda_{k}}y)\in K_{1}$, $\|q_{p}(s_{k})\|\leq \frac{1}{10}D$, $|s_{k}-h^{\prime}_{k}|=|\Delta^{\tau_{Y}}_{r_{k}}(a_{Y,k}^{-1}y,t_{k})|\leq \frac{11}{100}D$, we conclude that
\[d_{G}(q_{p}(s_{k}),u_{X}^{h^{\prime}_{k}-s_{k}})\leq \frac{2}{10}2^{-k}D\]
for any $p\in\{1,\ldots,n\}$.
\end{proof}

It then follows from the definition of $L^{1}(y,r_{k})$ (\ref{joinings202106.107}) that
\begin{equation}\label{joinings202106.115}
  \|q^{1}_{p}(s_{k})\|\geq\frac{9}{100}D ,\ \ \ |h^{\prime}_{k}-s_{k}|\geq\frac{9}{100}D
\end{equation}
for any $p\in\{1,\ldots,n\}$.

On the other hand, denote $h_{k}=\frac{h^{\prime}_{k}}{1-r_{k}h^{\prime}_{k}}$.
\begin{cla}\label{joinings202106.128}
   We have
   \[ |h_{k}-s_{k}|<2^{1-k}D.\]
\end{cla}
\begin{proof}
   One can calculate via (\ref{joinings202106.105})
   \begin{align}
   |h_{k}-s_{k}|= &  |h_{k}-h^{\prime}_{k}-(s_{k}-h^{\prime}_{k})| \;\nonumber\\
   =&  |\Delta_{r_{k}} ( h_{k})-\Delta^{\tau_{Y}}_{r_{k}}(a_{Y}^{-\lambda_{k}}y,t_{k})|\;\nonumber\\
   \leq&   |\Delta_{r_{k}} ( h_{k})-\Delta_{r_{k}} (t_{k})|+|\Delta_{r_{k}} ( t_{k})-\Delta^{\tau_{Y}}_{r_{k}}(a_{Y}^{-\lambda_{k}}y,t_{k})|\;\nonumber\\
\leq& |\Delta_{r_{k}} ( h_{k})-\Delta_{r_{k}} (t_{k})|+ \frac{1}{10}2^{-k} D. \label{joinings202106.133}
\end{align}
  On the other hand,  by the ergodicity (\ref{joinings202106.131}) (\ref{joinings202106.110}), we have
  \[|h_{k}^{\prime}-t_{k}^{\prime}|\leq \frac{1}{10}2^{-k}D\cdot t_{k}^{\prime}\leq  \frac{2}{10}2^{-k}D\cdot t_{k}.\]
  Then by (\ref{joinings202106.132}) and $|\Delta_{r_{k}}(t_{k})|\leq D/10$, we have
  \[|h_{k}-t_{k}|=  \left|\frac{h^{\prime}_{k}}{1-r_{k}h^{\prime}_{k}}-\frac{t^{\prime}_{k}}{1-r_{k}t^{\prime}_{k}}\right| =\left|\frac{h^{\prime}_{k}-t^{\prime}_{k}}{(1-r_{k}h^{\prime}_{k})(1-r_{k}t^{\prime}_{k})}\right|   \leq  \frac{4}{10}2^{-k}D\cdot t_{k}.\]
  It follows that
  \begin{align}
|\Delta_{r_{k}} ( h_{k})-\Delta_{r_{k}} (t_{k})|= &  |r_{k}h_{k}h^{\prime}_{k}-r_{k}t_{k}t^{\prime}_{k}| \;\nonumber\\
   \leq&  |r_{k}h_{k}(h^{\prime}_{k}-t^{\prime}_{k})|+|r_{k}t^{\prime}_{k}(h_{k}-t_{k})| \;\nonumber\\
   \leq&  \frac{2}{10}2^{-k}D\cdot |r_{k}h_{k}t_{k}|+\frac{4}{10}2^{-k}D\cdot |r_{k}t_{k}^{\prime}t_{k}|\;\nonumber\\
   \leq&  \frac{4}{10}2^{-k}D\cdot |\Delta(t_{k})|+\frac{8}{10}2^{-k}D\cdot |\Delta(t_{k})|\leq\frac{12}{10}2^{-k}D. \nonumber
\end{align}
Then (\ref{joinings202106.133}) is clearly not greater than $2^{1-k}D$.
\end{proof}
Now Claim \ref{joinings202106.127} and \ref{joinings202106.128} imply that  $h_{k}\in\left[\frac{96}{100}L^{1}(y,r_{k}),\frac{999}{1000}L^{1}(y,r_{k})\right]$,    $|h^{\prime}_{k}-h_{k}|\in[\frac{9}{100}D,\frac{11}{100}D]$  and
\begin{align}
 d_{\overline{X}}(u_{X}^{h_{k}}\overline{\psi}_{p}(\overline{u}_{Y}^{r_{k}}a_{Y}^{-\lambda_{k}}y),u_{X}^{h^{\prime}_{k}}\overline{\psi}_{p}(a_{Y}^{-\lambda_{k}}y))\leq &\frac{2}{10}2^{1-k}D\;\nonumber\\ d_{\overline{X}}(u_{X}^{h_{k}}\overline{u}_{X}^{r_{k}}\overline{\psi}_{p}(a_{Y}^{-\lambda_{k}}y),u_{X}^{h^{\prime}_{k}}\overline{\psi}_{p}(a_{Y}^{-\lambda_{k}}y))\leq&   \frac{2}{10}2^{1-k}D \;\nonumber\\
  d_{G_{X}}(q_{p}(h_{k}),u_{X}^{h^{\prime}_{k}-h_{k}})\leq &  \frac{2}{10}2^{1-k}D\;  \nonumber
\end{align}
 for $p\in\{1,\ldots,n\}$.

Similarly, for $i=2$,   there exists  $h_{k,2}\in\left[\frac{96}{100}L^{2}(y,r_{k}),\frac{999}{1000}L^{2}(y,r_{k})\right]$ and $h^{\prime}_{k,2}\in\mathbf{R}$ with $|h^{\prime}_{k,2}-h_{k,2}|\in[\frac{9}{100}2^{2}D,\frac{11}{100}2^{2}D]$ such that
\begin{align}
 d_{\overline{X}}(u_{X}^{h_{k,2}}\overline{\psi}_{p}(\overline{u}_{Y}^{r_{k}}a_{Y}^{-\lambda_{k}}y),u_{X}^{h^{\prime}_{k,2}}\overline{\psi}_{p}(a_{Y}^{-\lambda_{k}}y))\leq &\frac{2}{10}2^{1-k}D\;\nonumber\\ d_{\overline{X}}(u_{X}^{h_{k,2}}\overline{u}_{X}^{r_{k}}\overline{\psi}_{p}(a_{Y}^{-\lambda_{k}}y),u_{X}^{h^{\prime}_{k,2}}\overline{\psi}_{p}(a_{Y}^{-\lambda_{k}}y))\leq&   \frac{2}{10}2^{1-k}D \;\nonumber\\
  d_{G_{X}}(q_{p}^{2}(h_{k,2}),u_{X}^{h^{\prime}_{k,2}-h_{k,2}})\leq &  \frac{2}{10}2^{1-k}D\;  \nonumber
\end{align}
 for $p\in\{1,\ldots,n\}$. Note that by (\ref{joinings202106.123}), we have $h_{k}\in[\frac{1}{3}h_{k,2},\frac{2}{3}h_{k,2}]$. Thus, we have met the requirement of Lemma \ref{joinings202106.121} with pairs
\[  (\overline{\psi}_{p}(\overline{u}_{Y}^{r_{k}}a_{Y}^{-\lambda_{k}}y),\overline{\psi}_{p}(a_{Y}^{-\lambda_{k}}y))\ \ \ \text{ and }\ \ \ (\overline{u}_{X}^{r_{k}}\overline{\psi}_{p}(a_{Y}^{-\lambda_{k}}y),\overline{\psi}_{p}(a_{Y}^{-\lambda_{k}}y))\]
  at time $t=h_{k},h_{k,2}$. Then Lemma \ref{joinings202106.121} implies that
\[d_{\overline{X}}\left(u_{X}^{t}\overline{\psi}_{p}(\overline{u}^{r_{k}}a_{Y}^{-\lambda_{k}}y),u_{X}^{t}\overline{u}^{r_{k}}_{X}\overline{\psi}_{p}(a_{Y}^{-\lambda_{k}}y)\right)\leq O\left(\frac{2}{10}2^{1-k}D\right)=O(2^{-k}).\]
for $t\in[0,h_{k,2}]$. Moreover, if we write $\overline{\psi}_{p}(\overline{u}_{Y}^{r_{k}}a_{Y}^{-\lambda_{k}}y) =g_{p,k}\overline{u}_{X}^{r_{k}}\overline{\psi}_{p}(a_{Y}^{-\lambda_{k}}y)$ and
\[g_{p,k}=\exp\left(\sum_{j} \sum_{i=0}^{\varsigma(j)}b_{j}^{i}v_{j}^{i}\right)\]
where $v^{i}_{j}$ are the weight vectors of the $\mathfrak{sl}_{2}$-irreducible representation $V_{j}$, then by (\ref{joinings202106.120}) we deduce
\[      |b_{j}^{\varsigma(j)-i}|<O(2^{-k})h_{k,2}^{-i}.\]
Finally, one calculates via (\ref{joinings202106.129}) (\ref{joinings202106.101}) (\ref{joinings202106.130})
\begin{align}
a_{X}^{\lambda_{k}}g_{p,k}a_{X}^{-\lambda_{k}}\leq&\exp\left(\sum_{j} \sum_{i=0}^{\varsigma(j)}O(2^{-k})h_{k,2}^{\varsigma(j)-2i}\cdot h_{k,2}^{i-\varsigma(j)}v_{j}^{i}\right)\;\nonumber\\
 =&  \exp\left(\sum_{j} \sum_{i=0}^{\varsigma(j)}O(2^{-k})h_{k,2}^{-i}v_{j}^{i}\right)\leq O(2^{-k}).\;  \nonumber
\end{align}
Therefore, we conclude that
\[d_{\overline{X}}(\Psi_{k,p}(\overline{u}_{Y}^{r}y), \overline{u}_{X}^{r}\Psi_{k,p}(y))\leq O(2^{-k})\]
 for $p\in\{1,\ldots,n\}$. The theorem follows.
\end{proof}

\begin{rem} Similar to Remark \ref{joinings202106.134}, Theorem \ref{joinings202106.122} also holds true for $\rho$ being a finite extension of $\nu$, when $(X,\phi^{U_{X},\tau_{X}}_{t})$ is a time-change of the unipotent flow on $X=SO(n_{X},1)/\Gamma_{X}$: if for $f\in C(X\times Y)$
   \[\int f(x,y)d\rho(x,y)=\int \frac{1}{n}\sum_{p=1}^{n}f(\psi_{p}(y),y) d\nu(y)\]
then we still have
\[\lim_{n\rightarrow\infty}d_{X}(\Psi_{k,p}(\overline{u}_{Y}^{r}y), \overline{u}_{X}^{r}\Psi_{k,p}(y))=0\]
for $p\in\{1,\ldots, n\}$ and a.e. $y\in Y$.
\end{rem}

\section{Applications}\label{joinings202106.168}
In previous sections, we considered the measure of the form
\[\int fd\rho=\int\frac{1}{n}\sum_{p=1}^{n}f(\overline{\psi}_{p}(y),y)d\nu(y)\]
for some measurable functions $\overline{\psi}_{p}$. Besides, we studied the equivariant properties of $\overline{\psi}_{p}$. In this section, we use these results to develop the rigidity of $\rho$.
\subsection{Unipotent flows of $SO(n,1)$ vs. time-changes of unipotent flows}\label{joinings202106.146} In this section, we shall prove   Theorem \ref{joinings202106.160} and \ref{joinings202106.161}.
Let $G_{X}=SO(n_{X},1)$, $G_{Y}$ be a semisimple Lie group  with finite center and no compact factors  and $\Gamma_{X}\subset G_{X}$,  $\Gamma_{Y}\subset G_{Y}$ be irreducible lattices. Let $(X,\mu)$ be the homogeneous space $X=G_{X}/\Gamma_{X}$ equipped with the Lebesgue measure $\mu$, and let $\phi^{U_{X}}_{t}=u_{X}^{t}$ be a unipotent flow on $X$. Suppose that
\begin{itemize}
  \item  $Y$ is the homogeneous space $Y=G_{Y}/\Gamma_{Y}$,
  \item $m_{Y}$ is the Lebesgue measure on $Y$,
  \item $u_{Y}\in G_{Y}$ is a unipotent element that $C_{\mathfrak{g}_{Y}}(u_{Y})$ only contains vectors of   weight at most $2$,
  \item $\tau_{Y}\in \mathbf{K}_{\kappa}(Y)\cap C^{1}(Y)$ is a  positive integrable and $C^{1}$ function  on $Y$ such that $\tau_{Y},\tau_{Y}^{-1}$ are bounded and satisfies (\ref{joinings202104.26}),
  \item $\tilde{u}_{Y}^{t}=\phi^{U_{Y},\tau_{Y}}_{t}$ of the unipotent flow $u_{Y}$,
  \item $\nu$ is a $\tilde{u}_{Y}^{t}$-invariant measure on $Y$,
  \item $\rho\in J(u_{X}^{t},\phi_{t}^{U_{Y},\tau_{Y}})$ is a nontrivial (i.e. not the product $\mu\times\nu$) ergodic joining.
\end{itemize}
\begin{prop}\label{joinings202106.140}
   $\tau_{Y}(y)$ and $\tau_{Y}( cy)$ are (measurably)  cohomologous along $u_{Y}^{t}$ for all $c\in C_{G_{Y}}(U_{Y})$. Further, if $\tau_{Y}(y)$ and $\tau_{Y}(cy)$ are  $L^{1}$-cohomologous, then  after passing a subsequence if necessary,
  \[\Psi^{\ast}(y)\coloneqq\lim_{n\rightarrow\infty} \Psi^{\ast}_{k}(y) \]
  exists for $\nu$-a.e. $y\in Y$,  where $\Psi^{\ast}_{k}(y)\coloneqq\{\Psi_{k,p}(y):p\in\{1,\ldots,n\}\}$ and $\Psi_{k,p}(y)$ is given by (\ref{joinings202106.130}).
\end{prop}
\begin{proof}
   The first consequence follows from Theorem  \ref{dynamical systems2001}. For the second one, we first apply Lemma \ref{joinings202106.135} and obtain
\[\lim_{t\rightarrow\infty}\frac{1}{t}\alpha(c^{t},y)=\int \alpha(c,y)dm_{Y}(y)\]
   for $m$-a.e. $y\in Y$ whenever $c$ is $m_{Y}$-ergodic. Note that $d\beta:C_{\mathfrak{g}_{Y}}(U_{Y})\rightarrow V_{C_{X}}^{\perp}$ sends nilpotent elements to nilpotent elements.
   Thus, for weight vector $v\in C_{\mathfrak{g}_{Y}}(U_{Y})$ of   weight $\varsigma\leq 2$, $\nu$-almost all $y\in Y$, we have
   \[\Psi^{\ast}_{k}(\exp(v)y)=\left\{
    \begin{array}{ll}  u_{X}^{e^{-\lambda_{k}}\alpha(\exp(e^{\varsigma\lambda_{k}/2}v),y)}\beta(\exp(e^{\varsigma\lambda_{k}/2}v))^{e^{-\lambda_{k}}}\Psi^{\ast}_{k}(y) &,\text{ for } \varsigma\geq 1\\
      & \\
  u_{X}^{e^{-\lambda_{k}}\alpha(\exp(v),y)}a_{X}^{\lambda_{k}}\beta(\exp(v))a_{X}^{-\lambda_{k}}\Psi^{\ast}_{k}(y) &,\text{ for } \varsigma=0
                     \end{array}\right. . \]
  Thus, after passing to a subsequence if necessary, we have
   \begin{equation}\label{joinings202106.139}
     \lim_{k\rightarrow\infty}\Psi^{\ast}_{k}(\exp(v)y)=\left\{
    \begin{array}{ll}  u_{X}^{\int\alpha(\exp(v),\cdot)}\beta(\exp(v))\lim_{k\rightarrow\infty}\Psi^{\ast}_{k}(y) &,\text{ for } \varsigma= 2\\
      & \\
   \lim_{k\rightarrow\infty}\Psi^{\ast}_{k}(y) &,\text{ for } \varsigma= 1\\
      & \\
   \exp(v_{0})\lim_{k\rightarrow\infty}\Psi^{\ast}_{k}(y) &,\text{ for } \varsigma=0
                     \end{array}\right.
   \end{equation}
  where $\beta(\exp(v))=\exp(v_{0}+v_{2})$ for $v_{0},v_{2}\in V_{C_{X}}^{\perp}$ of weight $0$ and $2$ respectively. In particular, $\lim_{k\rightarrow\infty}\Psi^{\ast}_{k}(\exp(v)y)$ exists whenever $\lim_{k\rightarrow\infty}\Psi^{\ast}_{k}(y)$ exists. Besides,  by Theorem \ref{joinings202106.122}, we have
 \[\lim_{n\rightarrow\infty}d_{\overline{X}}(\Psi^{\ast}_{k}(\overline{u}_{Y}^{r}y), \overline{u}_{X}^{r}\Psi^{\ast}_{k}(y))=0\]
    for   $r\in\mathbf{R}$, $\nu$-a.e. $y\in Y$.

    It remains to show that  for $\nu$-almost all $y\in Y$, there exists a subsequence $\{k(y,l)\}_{l\in\mathbf{N}}\subset\mathbf{N}$ and $\Psi_{p}(y)\in \overline{X}$ such that
    \begin{equation}\label{joinings202106.137}
  \lim_{l\rightarrow\infty}\Psi_{k(y,l),p}(y)=  \Psi_{p}(y).
    \end{equation}
 To do this,  write $\overline{X}= \bigcup_{i=1}K_{i}$, where $K_{i}$ are compact and $\overline{\mu}(  K_{i})\nearrow 1$  as $i\rightarrow\infty$.  Let
  \[\Omega\coloneqq\bigcup_{i\geq1}\bigcap_{k\geq1}\bigcup_{j\geq k}\bigcap_{p=1}^{n}\Psi_{j,p}^{-1}(K_{i}).\]
  \begin{cla}\label{joinings202106.136}
      $\nu(\Omega)=1$.
  \end{cla}
  \begin{proof}
     From a direct calculation (recall that  $d\nu\coloneqq\tau dm_{Y}$), we know
     \begin{multline}
       m_{Y}\left(\bigcup_{i\geq1}\bigcap_{k\geq1}\bigcup_{j\geq k}\bigcap_{p=1}^{n}\Psi_{j,p}^{-1}(K_{i})\right)\geq  m_{Y}\left( \bigcap_{k\geq1}\bigcup_{j\geq k}\bigcap_{p=1}^{n}\Psi_{j,p}^{-1}(K_{i})\right) \\
       =\lim_{k\rightarrow\infty} m_{Y}\left(  \bigcup_{j\geq k}\bigcap_{p=1}^{n}\Psi_{j,p}^{-1}(K_{i})\right) \geq m_{Y}(\psi_{p}^{-1}a^{-\lambda_{j}} K_{i})\nonumber
     \end{multline}
     for any $p$, $j$ and $i$. As $\overline{\mu}(K_{i})\nearrow 1$ as $i\rightarrow\infty$, the claim follows.
  \end{proof}
 Then by Claim \ref{joinings202106.136} for  $y\in \Omega$,   there exists $i\geq 1$ such that $ \Psi_{j,p}(y)\in K_{i}$ for infinitely many $j$. Thus, we proved (\ref{joinings202106.137}).
 Therefore, since the opposite unipotent and central directions generate the whole group $\langle\overline{u}^{r}_{Y},C_{G_{Y}}(U_{Y})\rangle=G_{Y}$, we conclude that after passing a subsequence if necessary,
  \[\lim_{n\rightarrow\infty} \Psi_{k,p}(y) \]
  exists for $\nu$-a.e. $y\in Y$.
\end{proof}

Then, define a measure $\widetilde{\rho}$ on $\overline{X}\times Y$ by
  \[\int f d\widetilde{\rho}\coloneqq \int_{Y} \frac{1}{n}\sum_{p=1}^{n}f(\Psi_{p}(y),y)dm_{Y}(y)\]
  for $f\in C(\overline{X}\times Y)$  where $\Psi^{\ast}(y)=\{\Psi_{1}(y),\ldots,\Psi_{n}(y)\}$. Then $\widetilde{\rho}$ is a nontrivial $(u^{t}_{X}\times u^{t}_{Y})$-invariant measure on $\overline{X}\times Y$ such that $(\pi_{\overline{X}})_{\ast}\widetilde{\rho}=\overline{\mu}$ and  $(\pi_{Y})_{\ast}\widetilde{\rho}=m_{Y}$. Then, \textit{Ratner's theorem}\index{Ratner's theorem} \cite{ratner1990measure} asserts that  $C^{\rho}=\{e\}$ and
  \[\widetilde{\rho}(\stab(\widetilde{\rho}).(x_{0},y_{0}))=1 \]
  for some $(x_{0},y_{0})\in X\times Y$, where $\stab(\widetilde{\rho})\coloneqq\{(g_{1},g_{2})\in G_{X}\times G_{Y}: (g_{1},g_{2})_{\ast}\widetilde{\rho}=\widetilde{\rho}\}$. Then let
   \begin{itemize}
     \item  $\stab_{Y}(\widetilde{\rho})\coloneqq  \{(e,g_{2})\in G_{X}\times G_{Y}: (e,g_{2})_{\ast}\widetilde{\rho}=\widetilde{\rho}\}$  (note that $\stab_{Y}(\widetilde{\rho})\lhd G_{Y}$ is a normal subgroup of $G_{Y}$),
     \item $\Gamma_{X}^{g}\coloneqq\{ \gamma:g^{-1}\gamma g\in \Gamma_{X}\}$ for $g\in G_{X}$.
   \end{itemize}
   Then \textit{Ratner's theorem}\index{Ratner's theorem} \cite{ratner1990measure} further asserts that there is $g_{0}\in G_{Y}$ and a continuous surjective homomorphism  $\Phi:G_{Y}\rightarrow G_{X}$ with kernel $\stab_{Y}(\widetilde{\rho})$, $\Phi(g)=g$ for $g\in SL_{2}$ such that
   \begin{equation}\label{joinings202106.141}
    \{\Psi_{1}(h\Gamma_{Y}),\ldots,\Psi_{n}(h\Gamma_{Y})\}=\{\Phi(h)\gamma_{1}g_{0}\Gamma_{X},\ldots,\Phi(h)\gamma_{n}g_{0}\Gamma_{X}\}
   \end{equation}
   for all $h\in G_{Y}$, where the intersection $\Gamma_{0}\coloneqq \Phi(\Gamma_{Y})\cap\Gamma_{X}^{g_{0}}$ is of finite index in $\Phi(\Gamma_{Y})$ and in $\Gamma_{X}^{g_{0}}$, $n=|\alpha(\Gamma_{Y})/\Gamma_{0}|$ and $\Phi(\Gamma_{Y})=\{\gamma_{p}\Gamma_{0}:p\in\{1,\ldots,n\}\}$.

   Next, by using Proposition \ref{joinings202106.140} and (\ref{joinings202106.141}), for any $\sigma>0$ $\epsilon>0$, there exists a subset $K\subset Y$ with $\nu(K)>1-\sigma$ and $k_{0}>0$ such that
   \[\max_{p}\min_{q}d_{X}\left(\Psi_{k,p}(h\Gamma_{Y}),\Phi(h)\gamma_{q}g_{0}\Gamma_{X}\right)<\epsilon\]
  for $h\Gamma_{Y}\in K$, $k\geq k_{0}$. In particular, by the ergodic theorem, we know that for $\nu$-a.e. $y\in Y$, there is $A_{y}\subset\mathbf{R}^{+}$ and $\lambda_{0}(y)>0$  such that
    \begin{itemize}
      \item for $r\in A_{y}$, we have $u^{r}_{Y}y \in K$;
     \item
      $\Leb(A_{y}\cap[0,\lambda])\geq (1-2\sigma)\lambda$ whenever $\lambda\geq\lambda_{0}(y)$.
    \end{itemize}
  Therefore, one can repeat the same argument as in Section \ref{joinings202105.24}, and then conclude that there exists $c^{\prime}(h\Gamma_{Y})\in C_{G_{Y}}(U_{Y})$, $q^{\prime}(p,h\Gamma_{Y})\in\{1,\ldots,n\}$ such that
\[\Psi_{k,p}(h\Gamma_{Y})=c^{\prime}(h\Gamma_{Y})\Phi(h)\gamma_{q^{\prime}(p,h\Gamma_{Y})}g_{0}\Gamma_{X}\]
for $\nu$-a.e. $h\Gamma_{Y}\in Y$.
We can then  write
\[\psi_{p}(h\Gamma_{Y})=c(h\Gamma_{Y})\Phi(h)\gamma_{q(p,h\Gamma_{Y})}g_{0}\Gamma_{X}\]
for some $c(h\Gamma_{Y})\in C_{G_{Y}}(U_{Y})$, $q(p,h\Gamma_{Y})\in\{1,\ldots,n\}$, $\nu$-a.e. $h\Gamma_{Y}\in Y$.
 Thus,  let $I=(q_{1},q_{2},\ldots,q_{n})$ be a permutation of $\{1,\ldots,n\}$,
   \[S_{I}\coloneqq\{y\in Y:q(1,y)=q_{1},\ldots,q(n,y)=q_{n}\}\]
and let
\[\widetilde{\psi}_{p}(y)\coloneqq \psi_{q_{p}}(y)\ \ \text{ when }\ \ y\in S_{(q_{1},\ldots,q_{n})}.\]
Then $\widetilde{\psi}_{p}(y)$ plays the same role as $\psi_{p}(y)$ and satisfies
\begin{equation}\label{joinings202106.144}
\widetilde{\psi}_{p}(h\Gamma_{Y})=c(h\Gamma_{Y})\Phi(h)\gamma_{p}g_{0}\Gamma_{X}.
\end{equation}
 for  $\nu$-a.e. $h\Gamma_{Y}\in Y$. Thus, without loss of generality, we assume that $\psi_{p}$ satisfies (\ref{joinings202106.144}).
It follows that the map $\Upsilon:\supp(\rho)\rightarrow X\times Y$ defined by
\[\Upsilon:(\psi_{p}(h\Gamma_{Y}),h\Gamma_{Y})\mapsto(\Phi(h)\gamma_{p}g_{0}\Gamma_{X},h\Gamma_{Y})\ \ \ \text{ for }\  \  p\in\{1,\ldots,n\}\]
is bijective and satisfies
\begin{equation}\label{joinings202106.143}
 \Upsilon(u_{X}^{t}x,\tilde{u}_{Y}^{t}(y))=(u_{X}^{\xi(y,t)}\times u_{Y}^{\xi(y,t)}).\Upsilon(x,y)
\end{equation}
for $\rho$-a.e. $(x,y)$ and $t\in\mathbf{R}$. Equivalently,  we  obtain:
 \begin{prop}\label{joinings202106.142}
            Assume   that
            $\tau_{Y}(y)$ and $\tau_{Y}( cy)$ are $L^{1}$-cohomologous for all $c\in C_{G_{Y}}(U_{Y})$. Then $\tau_{X}\equiv1$ and $\tau_{Y}$ are joint cohomologous.
          \end{prop}
\begin{proof} By (\ref{joinings202106.143}), we can write down the decomposition (\ref{time change184}) for $c(y)$:
\[c(y)=u_{X}^{a(y)}b\]
and
\[a(y)+t=\xi(y,t)+a(u_{Y}^{\xi(y,t)}y).\]
 It follows that
 \[\int_{0}^{\xi(y,t)}\tau_{Y}(u_{Y}^{s}y)-1 ds= t-\xi(y,t)=  a(u_{Y}^{\xi(y,t)}y)-a(y).\]
Thus, $1$ and $\tau_{Y}$ are joint cohomologous via $(\tilde{\rho},a)$.
\end{proof}

Recall (\ref{joinings202106.139}) that when a weight vector $v\in C_{\mathfrak{g}_{Y}}(U_{Y})$ of   weight $\varsigma\geq 1$, we know that $\widetilde{\rho}$ is invariant under
\begin{equation}\label{joinings202106.145}
  \left\{
    \begin{array}{ll}  u_{X}^{\int\alpha(\exp(v),\cdot)}\beta(\exp(v))\times \exp(v)  &,\text{ for } \varsigma= 2\\
      & \\
        \id\times \exp(v)  &,\text{ for } \varsigma= 1\\
      & \\
   \exp(v_{0})\times \exp(v)  &,\text{ for } \varsigma= 0
                     \end{array}\right.
   \end{equation}
   where $\beta(\exp(v))=\exp(v_{0}+v_{2})$.
 Since  $\widetilde{\rho}$ is also $(u_{X}^{t}\times u^{t}_{Y})$-invariant,  if $\beta(\exp(v))=e$, then \textit{Moore’s ergodicity theorem}\index{Moore’s ergodicity theorem} and Lemma \ref{joinings202012.3} imply that $\langle\exp(v)\rangle\subset \ker\Phi$   is a compact normal subgroup of $G_{Y}$. It is a contradiction. Thus, we conclude
 \begin{prop}
    The map  $d\beta|_{V^{\perp}_{C}}: V^{\perp}_{C_{Y}}\rightarrow V^{\perp}_{C_{X}}$ is an injective Lie algebra homomorphism.
 \end{prop}

\subsection{Time-changes of unipotent flows of $SO(n,1)$ vs. unipotent flows}\label{joinings202106.172} In this section, we shall prove Theorem \ref{joinings202106.163}.
Let $G_{X}=SO(n_{X},1)$, $G_{Y}$ be a semisimple Lie group with finite center and no compact factors and $\Gamma_{X}\subset G_{X}$,  $\Gamma_{Y}\subset G_{Y}$ be irreducible lattices.
Let $(Y,\nu)$ be the homogeneous space $Y=G_{Y}/\Gamma_{Y}$ equipped with the Lebesgue measure $\nu$, and let $\phi^{U_{Y}}_{t}=u_{Y}^{t}$ be a unipotent flow on $Y$.
  Suppose that
\begin{itemize}
  \item  $X$ is the homogeneous space $X=G_{X}/\Gamma_{X}$,
  \item $u_{X}\in G_{X}$ is a unipotent element,
  \item $\tau_{X}\in \mathbf{K}_{\kappa}(X)$ is a  positive integrable and $C^{1}$ function  on $Y$ such that $\tau_{X},\tau_{X}^{-1}$ are bounded and satisfies (\ref{joinings202104.26}),
  \item $\tilde{u}_{X}^{t}=\phi^{U_{X},\tau}_{t}$ of the unipotent flow $u_{X}$,
  \item $\mu$ is a $\tilde{u}_{X}^{t}$-invariant measure on $X$,
  \item $\rho\in J(\tilde{u}_{X}^{t},u_{Y}^{t})$ is an  ergodic joining that is a compact extension of $\nu$, i.e. has the form
      \[ \rho(f)=\int_{Y} \int_{C^{\rho}}\frac{1}{n}\sum_{p=1}^{n}f(k\psi_{p}(y),y)dm(k)d\nu(y)\]
      for $f\in C(X\times Y)$ and compact $C^{\rho}\in C_{G_{X}}(U_{X})$.
\end{itemize}

Recall that in Remark \ref{joinings202106.134},  for $c\in C_{G_{Y}}(U_{Y})$, we know that $\rho$ is invariant under the map
\[\widetilde{S}_{c}:(x,y) \mapsto ( u_{X}^{\alpha(c,y)}\beta(c)x,cy)\]
(cf. (\ref{joinings202105.25})). Besides, $\alpha,\beta$ satisfy
\[\xi(\psi_{p}(cy),t)+\alpha(c,y)= \alpha(c,u_{Y}^{t}y)+\xi(\psi_{p}(y),t),\]
\begin{equation}\label{joinings202106.152}
 \alpha(c_{1}c_{2},y)=\alpha(c_{1},c_{2}y)+\alpha(c_{2},y),\ \ \ \beta(c_{1}c_{2})=\beta(c_{1})\beta(c_{2})
\end{equation}
where
\[t=\int_{0}^{\xi(x,t)}\tau_{X}(u_{X}^{s}x)ds.\]
Moreover, if $\beta(c)=e$ for some $c\in C_{G_{Y}}(U_{Y})$, then we have (\ref{joinings202106.148}):
\begin{equation}\label{joinings202106.156}
 \alpha(c,y)=\xi(x,r_{c})
\end{equation}
for some $r_{c}\in\mathbf{R}$.  Note that (\ref{joinings202106.156}) implies that
\[ (x,y) \mapsto ( u_{X}^{\alpha(c,y)}x,cy)\mapsto ( x,u_{Y}^{-r_{c}}cy)\]
is $\rho$-invariant. Thus, \textit{Moore’s ergodicity theorem}\index{Moore’s ergodicity theorem} and Lemma \ref{joinings202012.3} force
\begin{equation}\label{joinings202106.171}
   \alpha(\exp (v),y)\equiv 0\ \ \ \text{ and } \ \ \   \langle\exp(v)\rangle\subset G_{Y}
\end{equation}
 is compact. In particular, we obtain (\ref{joinings202106.169}):
\[d\beta|_{V^{\perp}_{C}}(v)\neq0\]
 for any weight vector $v\in V^{\perp}_{C_{Y}}$ of positive weight.
 Inspired by this, we deduce
\begin{lem}\label{joinings202106.170}
    For  weight vectors $v\in C_{\mathfrak{g}_{Y}}(U_{Y})$ of   weight $\varsigma\neq0, 2$, we must have
    \[d\beta(v)=0. \]
\end{lem}
\begin{proof}
   Similar to Theorem \ref{joinings202106.149}, one can deduce that  for $r\in\mathbf{R}$,
   \[\widetilde{S}_{a^{r}_{Y}}:(x,y)\mapsto \left(u_{X}^{\alpha(a^{r}_{Y},y)}\beta(a^{r}_{Y})a^{r}_{X}x,a^{r}_{Y}y\right)\]
   is $\rho$-invariant.  Also, we have
  \[\widetilde{S}_{a_{Y}}\circ \widetilde{S}_{c}\circ \widetilde{S}_{a_{Y}^{-1}}=\widetilde{S}_{a_{Y}ca_{Y}^{-1}}\]
  for any $a_{Y}\in \exp(\mathbf{R}A_{Y})$, $c\in C_{G_{Y}}(U_{Y})$. In particular, one deduces
  \[\beta(a_{Y}^{r})a_{X}^{r}\beta(a_{Y}^{-r})a_{X}^{-r}=e,\ \ \ \beta(a_{Y}^{r})a_{X}^{r}\beta(c)\beta(a_{Y}^{-r})a_{X}^{-r}=\beta(a_{Y}^{r}ca_{Y}^{-r}).\]
  Thus, suppose that     $v\in C_{\mathfrak{g}_{Y}}(U_{Y})$ is a  weight vector of   weight $\varsigma\neq0, 2$. Then
\begin{multline}
  \beta(\exp(v))^{e^{r\varsigma/2}}=\beta(\exp(e^{r\varsigma/2}v))= \beta(a_{Y}^{r}\exp(v)a_{Y}^{-r}) \\
   =\beta(a_{Y}^{r})a_{X}^{r}\beta(\exp(v))\beta(a_{Y}^{-r})a_{X}^{-r} =\beta(a_{Y}^{r})a_{X}^{r}\beta(\exp(v))a_{X}^{-r}\beta(a_{Y}^{r})^{-1}.\label{joinings202106.150}
\end{multline}
Assume that $\beta(\exp(v))=\exp(w)$ for some $w\in C_{\mathfrak{g}_{X}}(U_{X})$. By the assumption, $w$ has to be  nilpotent and so
\begin{equation}\label{joinings202106.151}
 a_{X}^{r}\beta(\exp(v))a_{X}^{-r}=a_{X}^{r}\exp(w)a_{X}^{-r}=\exp(e^{r}w).
\end{equation}
Combining (\ref{joinings202106.150}) and (\ref{joinings202106.151}), we get
\[e^{r\varsigma/2}\|w\|=\|e^{r\varsigma/2}w\|= \|\Ad \beta(a_{Y}^{r}).e^{r}w\|= \|e^{r}w\|= e^{r}\|w\|\]
which leads to a contradiction.
\end{proof}
Then by \textit{Moore’s ergodicity theorem}\index{Moore’s ergodicity theorem} and Lemma \ref{joinings202012.3} (cf. Remark \ref{joinings202106.134}), we conclude
\begin{cor}
   If $C_{\mathfrak{g}_{Y}}(U_{Y})$ contains a weight vector of   weight $\varsigma\neq0, 2$, then
   \[\rho=\mu\times\nu.\]
\end{cor}

Now we focus on the case $n_{X}=2$ and $\tau_{X}\in \mathbf{K}(X)\cap C^{1}(X)$. Note that in this case, Ratner \cite{ratner1987rigid} showed that $\tilde{u}^{t}_{X}$ also has H-property. Thus,  we can repeat the same idea as in Section \ref{joinings202106.146} to discuss the case when $C_{\mathfrak{g}_{Y}}(U_{Y})$ consists only of weight vectors of   weight $\varsigma=0, 2$. Note that since $\beta\equiv0$, by (\ref{joinings202106.156}), we must have $\alpha(c,\cdot)\in L^{\infty}(Y)$ for any $c\in C_{G_{Y}}(U_{Y})$.  Then, similar to  Proposition \ref{joinings202106.140}, we have
\begin{prop}\label{joinings202106.174}
  Assume that  $C_{\mathfrak{g}_{Y}}(U_{Y})$ consists only of weight vectors of   weight $\varsigma=0, 2$. Then  after passing a subsequence if necessary,
  \begin{equation}\label{joinings202106.154}
   \Psi^{\ast}(y)\coloneqq\lim_{n\rightarrow\infty} \Psi^{\ast}_{k}(y)
  \end{equation}
  exists for $\nu$-a.e. $y\in Y$,  where $\Psi^{\ast}_{k}(y)\coloneqq\{\Psi_{k,p}(y):p\in\{1,\ldots,n\}\}$ and $\Psi_{k,p}(y)$ is given by (\ref{joinings202106.130}).
\end{prop}
\begin{rem}
   One nontrivial step of Proposition \ref{joinings202106.174} is to obtain a similar version of Theorem \ref{joinings202106.122}. This requires that the time-change $\tilde{u}^{t}_{X}$ also has H-property. See \cite{ratner1987rigid} Lemma 3.1 for further details.
\end{rem}
  Then by \textit{Ratner's theorem}\index{Ratner's theorem} (cf. (\ref{joinings202106.144})), there exists $c(h\Gamma_{Y})\in C_{G_{X}}(U_{X})=\exp(\mathbf{R}U_{X})$, a homomorphism $\Phi(h)$, $\gamma_{p},g_{0}\in G_{X}$ such that $\psi_{p}$ can be  written as
  \begin{equation}\label{joinings202106.155}
    \psi_{p}(h\Gamma_{Y})=c(h\Gamma_{Y})\Phi(h)\gamma_{p}g_{0}\Gamma_{X}
  \end{equation}
for $h\Gamma_{Y}\in Y$. Then as in Proposition \ref{joinings202106.142}, we get

\begin{prop} $\tau_{X}$  and $\tau_{Y}\equiv1$ are joint cohomologous.
          \end{prop}

 Finally, consider $\rho$ is nontrivial$v\in C_{G_{Y}}(U_{Y})$.  Since    $\beta(\exp(v))=e$,  (\ref{joinings202106.171}) asserts that
\[ \alpha(\exp (v),y)\equiv 0\ \ \ \text{ and } \ \ \   \langle\exp(v)\rangle\subset G_{Y}\]
 is compact. However,   \textit{Ratner's theorem}\index{Ratner's theorem} implies that $\langle\exp(v)\rangle\subset \ker\Phi$   is a normal subgroup of $G_{Y}$. It is a  contradiction. Thus, we conclude
 \[V^{\perp}_{C_{Y}}=0.\]
Therefore, we have proved Theorem \ref{joinings202106.173}.

\bibliographystyle{alpha}
 \bibliography{text}
\end{document}